\documentclass[a4paper,12pt,dvipsnames,reqno]{amsart}

\usepackage[a4paper,top=3cm,bottom=2cm,left=3cm,right=3cm,marginparwidth=1.75cm]{geometry}

\usepackage{amssymb,amsmath,amsthm,amscd,epsfig}
\usepackage{graphicx}
\usepackage[colorinlistoftodos]{todonotes}
\usepackage[colorlinks=true, allcolors=blue]{hyperref}

\usepackage[english]{babel}
\usepackage[utf8x]{inputenc}
\usepackage[T1]{fontenc}
\usepackage{cleveref}
\usepackage{amsmath, amssymb, amsthm}
\usepackage{mathrsfs}
\usepackage{mathtools}
\usepackage{tikz-cd}
\usepackage[all]{xy}

\usepackage{bm}

\newtheorem{thm}{Theorem}[section]
\newtheorem{lem}[thm]{Lemma}

\newtheorem{lemma}[thm]{Lemma}
\newtheorem{theorem}[thm]{Theorem}
\newtheorem{proposition}[thm]{Proposition}
\newtheorem{corollary}[thm]{Corollary}

\theoremstyle{definition}

\newtheorem{example}[thm]{Example}
\newtheorem{definition}[thm]{Definition}

\newtheorem{remark}[thm]{Remark}
\newtheorem{assumption}[thm]{Assumption}

 \numberwithin{equation}{section}

\theoremstyle{remark}

\numberwithin{table}{section}

\def\Z{\ifmmode{{\mathbb Z}}\else{${\mathbb Z}$}\fi}
\def\Q{\ifmmode{{\mathbb Q}}\else{${\mathbb Q}$}\fi}
\def\C{\ifmmode{{\mathbb C}}\else{${\mathbb C}$}\fi}
\def\P{\ifmmode{{\mathbb P}}\else{${\mathbb P}$}\fi}
\def\H{\ifmmode{{\mathrm H}}\else{${\mathrm H}$}\fi}
\def\G{\ifmmode{{\mathbb G}}\else{${\mathbb G}$}\fi}
\def\R{\ifmmode{{\mathbb R}}\else{${\mathbb R}$}\fi}
\def\F{\ifmmode{{\mathbb F}}\else{${\mathbb F}$}\fi}
\def\N{\ifmmode{{\mathbb N}}\else{${\mathbb N}$}\fi}
\def\D{\ifmmode{{\cal{D}}^b}\else{${{\cal{D}}^b}$}\fi}

\newcommand\NN{\mathbf{N}}
\renewcommand\AA{\mathbf{A}}

\newcommand\GG{\mathbb{G}}

\newcommand\Gm{\GG_\mathrm{m}}

\newcommand{\Xbar}{{\overline{X}}}

\newcommand{\calA}{{\mathcal A}}
\newcommand{\calB}{{\mathcal B}}
\newcommand{\calC}{{\mathcal C}}

\newcommand{\calO}{{\mathcal O}}

\newcommand{\OO}{{\mathcal O}}

\newcommand{\frakp}{{\mathfrak p}}

\newcommand{\pp}{{\mathfrak p}}
\newcommand{\qq}{{\mathfrak q}}

\DeclareMathOperator{\Ann}{Ann}
\DeclareMathOperator{\HH}{H}

\DeclareMathOperator{\Tr}{Tr}

\DeclareMathOperator{\lcm}{lcm}

\DeclareMathOperator{\Frob}{Frob}

\DeclareMathOperator{\inv}{inv}

\DeclareMathOperator{\re}{Re}

\DeclareMathOperator{\Hom}{Hom}
\DeclareMathOperator{\Ext}{Ext}

\DeclareMathOperator{\Aut}{Aut}
\DeclareMathOperator{\Gal}{Gal}

\DeclareMathOperator{\Res}{Res}
\DeclareMathOperator{\Norm}{N}

\DeclareMathOperator{\divv}{div}
\DeclareMathOperator{\ord}{ord}

\DeclareMathOperator{\Pic}{Pic}

\DeclareMathOperator{\Proj}{Proj}

\DeclareMathOperator{\ev}{ev}

\DeclareMathOperator{\cores}{cores}

\DeclareMathOperator{\rad}{rad}
\DeclareMathOperator{\Li}{Li}
\DeclareMathOperator{\dens}{dens}

\usepackage{todonotes}

\newcommand{\A}{\mathbb{A}}

\DeclareMathOperator{\Br}{Br}
\DeclareMathOperator{\sign}{sign}

\newcommand{\Adele}{\mathbf{A}}

\newcommand{\del}{\partial}

\newcommand{\ov}[1]{\overline{#1}}
\newcommand{\wh}[1]{\widehat{#1}}
\newcommand{\wt}[1]{\widetilde{#1}}

\providecommand{\BA}{{\mathbb{A}}}
\providecommand{\BC}{{\mathbb{C}}}
\providecommand{\BF}{{\mathbb{F}}}

\providecommand{\BP}{{\mathbb{P}}}
\providecommand{\BQ}{{\mathbb{Q}}}
\providecommand{\BR}{{\mathbb{R}}}
\providecommand{\BZ}{{\mathbb{Z}}}

\providecommand{\CA}{{\mathcal{A}}}

\providecommand{\CO}{{\mathcal{O}}}

\providecommand{\RH}{\mathrm{H}}
\providecommand{\RP}{\mathrm{P}}

\providecommand{\ev}{\mathrm{ev}}

\providecommand{\Het}{\mathrm{H}_\mathrm{\acute{e}t}}

\providecommand{\Pic}{\mathrm{Pic}}

\usepackage[OT2,T1]{fontenc}
\DeclareSymbolFont{cyrletters}{OT2}{wncyr}{m}{n}
\DeclareMathSymbol{\Sha}{\mathalpha}{cyrletters}{"58}

\usepackage{color}

\newcommand{\dan}[1]{{\color{Blue} \sf $\clubsuit\clubsuit\clubsuit$ Dan : [#1]}}

\title{Quantitative arithmetic of diagonal degree $2$ K3 surfaces}

\author{Dami\'{a}n Gvirtz}
\address{Institut für Algebra, Zahlentheorie und Diskrete Mathematik\\
Fakultät für Mathematik und Physik\\
Leibniz Universität Hannover\\
Welfengarten 1\\
30167 Hannover}

\author{Daniel Loughran}
\address{
Department of Mathematical Sciences\\
University of Bath\\
Claverton Down\\
Bath\\
BA2 7AY\\
UK}
\urladdr{https://sites.google.com/site/danielloughran/}

\author{Masahiro Nakahara}
\address{
Department of Mathematical Sciences\\
University of Bath\\
Claverton Down\\
Bath\\
BA2 7AY\\
UK}


\begin{document}

\begin{abstract}
	In this paper we study the existence of rational points for the family of 
	K3 surfaces over $\Q$ given by
	$$w^2 = A_1x_1^6 + A_2x_2^6 + A_3x_3^6.$$
	When the coefficients are
	ordered by height, we show that
	the  Brauer group is almost
	always trivial, and find the exact order of magnitude of surfaces for which
	there is a Brauer--Manin obstruction to
	the Hasse principle.
	Our results show definitively that K3 surfaces can have a Brauer--Manin
	obstruction to the Hasse principle
	that is only explained by odd order
	torsion.
\end{abstract}

\maketitle

\tableofcontents

\section{Introduction}
This paper concerns the existence of rational points on varieties over $\Q$, and in particular the quantitative behaviour of the Hasse principle in families.

Recall that a projective variety $X$ over $\Q$ is said to fail the \emph{Hasse principle} if $X(\Q) = \emptyset$ but $X(\Adele_\Q) \neq \emptyset$, where $\Adele_\Q$ denotes the adeles of $\Q$. There are various known obstructions to the Hasse principle, with one of the most popular being the \emph{Brauer--Manin obstruction}: this constructs an intermediate set $X(\Q) \subset X(\Adele_\Q)^{\Br} \subset X(\Adele_\Q)$, and if $X(\Adele_\Q)^{\Br} = \emptyset$ but $X(\Adele_\Q) \neq \emptyset$ then we say that there is a \emph{Brauer--Manin obstruction to the Hasse principle} (we recall background on Brauer groups and the Brauer--Manin set in \S\ref{brauerbasics}).

Our interest in this paper is on K3 surfaces. Such surfaces lie at the frontier of our understanding of the Hasse principle for smooth projective surfaces. A conjecture due to Skorobogatov \cite{Sko09} states that the Brauer--Manin obstruction is the only one for K3 surfaces. This conjecture is wide open in general, and is only known for some special families of K3 surfaces \cite{Har19,HS16}, such as some elliptic or Kummer surfaces, assuming some big conjectures in number theory (e.g.~finiteness of $\Sha$).

Given the difficulty in general, we focus throughout on the following completely explicit family of ``diagonal K3 surfaces of degree 2'', given by
\begin{equation} \label{eqn:X_A}
X_{\mathbf{A}}: \quad w^2=A_1x_1^6+A_2x_2^6+A_3x_3^6 \quad \subset \P(3,1,1,1)
\end{equation}
where $\mathbf{A}=(A_1,A_2,A_3) \in \Z^3$. In this paper we capitalise on recent breakthroughs on the calculation of the Brauer groups of K3 surfaces by Corn--Nakahara \cite{cn17} and Gvirtz--Skorobogatov \cite{qua}, and calculate the complete Brauer group, including the transcendental part, for most of the surfaces in this family. We then study the number of such surfaces for which there is a Brauer--Manin obstruction to the Hasse principle, as the coefficients lie in a box $|A_i| \leq T$ with $ T \to \infty$.

\subsection{The Brauer group}
Our first result concerns the asymptotic triviality of the Brauer group.
\begin{theorem}\label{mainbrauer} As $T \to \infty$ we have
 \[\#\left\{(A_1,A_2,A_3) \in \Z^3 :    
    |A_1|,|A_2|,|A_3| \leq T,
    \Br X_{\mathbf{A}}/ \Br_0 X_\AA \neq 0
    \right\}  \sim 4 \cdot \sqrt{3} T^{5/2}.\]
\end{theorem}
We prove this by showing that the most common way for the Brauer group to be non-constant is for $-3A_i$ to be a square for some $i$; counting this  is rather simple.

Note that \cite[Thm.~1.4]{BBL16} implies that a positive proportion of the surfaces $X_\AA$, when ordered by height, are everywhere locally soluble. In particular, if one believes Skorogobatov's conjecture, then Theorem \ref{mainbrauer} implies that a positive proportion of the surfaces $X_\AA$ have a rational point.

One particular difficulty of our analysis compared to surfaces of negative Kodaira dimension is the presence of \emph{transcendental} Brauer classes. These are classes that do not lie in the \emph{algebraic} Brauer group $\Br_1 X = \ker(\Br X \to \Br X_{\bar{\Q}})$. Traditionally, the transcendental part of the Brauer group has been the least understood, but recent methods developed for diagonal quartic surfaces \cite{qua} make it accessible in the case we consider. We adapt those methods for diagonal surfaces in weighted projective space to prove the following result, which shows that almost always the transcendental part of the Brauer group is trivial.

\begin{theorem}\label{maintranscendental}
For all $\varepsilon > 0$ we have
 \[\#\left\{(A_1,A_2,A_3) \in \Z^3 :    
    |A_1|,|A_2|,|A_3| \leq T,
    \Br X_{\mathbf{A}}/ \Br_1 X_{\mathbf{A}} \neq 0
    \right\}  \ll_\varepsilon T^{1/2+\varepsilon}.\]
\end{theorem}

The key property which allows us to handle the transcendental Brauer group is the fact that our surfaces have \emph{complex multiplication}; this allows one to make explicit the Galois action on $\Br X_{\AA, \bar{\Q}}$. We show that most of the time the Galois invariants are trivial (at least away from $3$), thus certainly no Brauer group element defined over $\Q$. The $3$-primary part is slightly more delicate as there can be non-trivial Galois invariant elements; we show that these elements do not descent to $\Q$ using an explicit description of connecting maps from \cite{cts}.

Further information about the possibilities for the algebraic and transcendental parts of the Brauer group can be found in Table \ref{table:brsub} and \S\ref{sextic}, respectively.

\subsection{Brauer--Manin obstruction}
Our primary aim is to study how often a Brauer--Manin obstruction to the Hasse principle occurs in the family \eqref{eqn:X_A}. By Theorem~\ref{mainbrauer} the Brauer group is almost always constant, so a Brauer--Manin obstruction is rare. However, even within those surfaces with non-constant Brauer group, an obstruction is still incredibly rare by an extra factor of $T$. This is illustrated by our next theorem, which calculates the exact order of magnitude of such surfaces.
\begin{theorem} \label{thm:main1}
	As $T \to \infty$, we have
    $$\#\left\{(A_1,A_2,A_3) \in \Z^3 :
    \begin{array}{c}
    |A_1|,|A_2|,|A_3| \leq T,\\
    X_{\mathbf{A}} \textup{ has a BM-obstruction to the HP}
    \end{array} \right\}  \asymp \frac{T^{3/2}}{(\log T)^{3/8}}.$$
\end{theorem}
The main term comes from special conditions on the coefficients. To elaborate we divide the triple $(A_1,A_2,A_3)$ into the following families:
\begin{enumerate}
	\item[(1)] $-3A_i$ is not a square and $A_j/A_k$ is not a cube for all $i,j,k$ with $j \neq k$;
	\item[(2)] $-3A_i$ is a square for some $i$;
	\item[(3)] $A_j/A_k$ is a cube for some $j,k$ with $j \neq k$.
\end{enumerate}
The largest family corresponds to case (1). Nonetheless our next result, which is the main theorem in the paper, will show that for such surfaces a Brauer--Manin obstruction to the Hasse principle is extremely rare. We will show that the Brauer group of the generic surface in families (2) and (3) is generated by an algebra of order $3$ and $2$ modulo constant algebras, respectively, and that the generic member of the intersection of (2) and (3) has constant Brauer group. The second family (2) is much more common than (3), but our next theorem shows that between these two families a Brauer--Manin obstruction occurs for a  comparable number of surfaces asymptotically (differing by a small power of $\log T$), and that the surfaces in family (3) in fact dominate asymptotically. 
We write
$$N_i(T) = 
\#\left\{(A_1,A_2,A_3) \in \Z^3 :
    \begin{array}{c}
    |A_1|,|A_2|,|A_3| \leq T \text{ lies in case } (i),\\
    X_{\mathbf{A}} \text{ has a BM-obstruction to the HP}. \\
    \end{array} \right\}.$$

\begin{theorem}\label{thm:main2}
	\begin{align*}
        N_1(T) \ll_{\varepsilon} T^{1 + \varepsilon},\ \
        N_2(T) \asymp \frac{T^{3/2} \log \log T}{(\log T)^{2/3}},\ \ N_3(T) \asymp \frac{T^{3/2}}{(\log T)^{3/8}}.
	\end{align*}
\end{theorem}

The authors were particularly surprised to see the term $\log \log T$ appearing here; such expressions are more common in prime number theory than in  arithmetic geometry. This factor appears naturally in the proof: there is a special type of prime $q$ which is only allowed to divide two of the coefficients exactly once for there to be a Brauer--Manin obstruction; when summing over the coefficients the term $1/q$ appears, which gives $\log \log T$ by Mertens' theorem  (see Remark~\ref{rem:Merten} for  details). 

We obtain the power savings over Theorem \ref{mainbrauer} by showing that a Brauer--Manin obstruction  imposes valuative conditions on the primes dividing the coefficients, so that up to small primes $A_i$ is essentially a square and $A_j/A_k$ is essentially a cube, in both cases (2) and (3); this gives the upper bound $O_\varepsilon(T^{3/2 + \varepsilon})$ for all $\varepsilon > 0$. To obtain the powers of $\log T$, we show that the prime divisors of the coefficients must have special splitting types in certain number fields.  One novel tool in our strategy to deal with such conditions is the effective version of the Chebotarev density theorem \cite{LO77}; to get the required uniformity, we need to rule out  exceptional zeros for certain Dedekind zeta functions, which we do with help from work of Heilbronn \cite{Hei73}. The lower bounds in Theorem \ref{thm:main2} are obtained by constructing explicit counter-examples.

Our results illustrate that given a family of varieties, it can be difficult to predict in general how many will have a Brauer--Manin obstruction to the Hasse principle and which subfamilies will dominate the overall count.

\subsection{Odd order obstructions to the Hasse principle}
Our results answer an open question in the literature, asked in various forms in \cite{IS15} (after Cor.~1.3), at the 2014 AIM workshop on Rational and integral points on higher-dimensional varieties \cite[Problem.~8]{AIM}, and in \cite[\S1.2]{cv18}. For a K3 surface $X$ over $\Q$, this asks whether it is possible for there to be a Brauer--Manin obstruction to the Hasse principle which is completely explained by the odd order torsion, i.e.~whether it is possible to have
\begin{equation} \label{eqn:odd_conjecture}
X(\Adele_\Q)^{\Br_{2^{\phantom{\perp\mkern-15mu}}}} \neq \emptyset, \quad \text{ but }
\quad X(\Adele_\Q)^{\Br_{2^\perp}} = \emptyset \quad ?
\end{equation}
Here $\Br_{2^{\phantom{\perp\mkern-15mu}}}:=\Br (X)[2^\infty]$ denotes the $2$-primary part of $\Br X$ and $\Br_{2^\perp}:=\Br(X)[2^\perp]$ those elements of $\Br X$ of order coprime to $2$.

There are partial results in the literature on this problem. In \cite[Thm.~1.9]{cv18} it is shown that if $X$ is a Kummer surface, then 
$$X(\Adele_\Q)^{\Br} = \emptyset \quad \implies X(\Adele_\Q)^{\Br_{2^{\phantom{\perp\mkern-15mu}}}} = \emptyset,$$
thus \eqref{eqn:odd_conjecture} cannot happen for Kummer surfaces. In \cite{cn17} and \cite{BVA19} examples of K3 surfaces $X$ are given with
$$X(\Adele_\Q) \neq \emptyset, \quad  X(\Adele_\Q)^{\Br_{2^\perp}}  = \emptyset,$$
however these works left open the possibility of such counter-examples being explained by an even order element, i.e.~whether $X(\Adele_\Q)^{\Br_{2^{\phantom{\perp\mkern-15mu}}}} = \emptyset$. This is because they were unable to calculate the full transcendental Brauer group.

We are now able to completely answer this question, as our methods allow us to  calculate the transcendental Brauer group of diagonal K3 surfaces of degree $2$.
In case (2) of Theorem \ref{thm:main2}, we will show that for almost all $X_\AA$ we have $\Br X_{\AA}/\Br_0 X_{\AA} \cong \Z/3\Z$, and our final theorem shows that there are infinitely many such surfaces over $\Q$ which have a Brauer--Manin obstruction to the Hasse principle not explained by an element of even order.

\begin{theorem}\label{thm:oddtor}
    $$\#\left\{(A_1,A_2,A_3) \in \Z^3 :
    \begin{array}{c}
    |A_1|,|A_2|,|A_3| \leq T,\\
    X_\AA(\Adele_\Q)^{\Br_{2^{\phantom{\perp\mkern-15mu}}}} \neq \emptyset,
X_\AA(\Adele_\Q)^{\Br_{2^\perp}} = \emptyset 
    \end{array} \right\}  \asymp \frac{T^{3/2} \log \log T}{(\log T)^{2/3}}.$$
\end{theorem}
\noindent An example of a surface in the family above originally appears in \cite{cn17}:
$$w^2=-3x^6+97y^6+97\cdot28\cdot8z^6.$$
By  \Cref{invariantsQ} this surface has no transcendental Brauer classes, which proves that the obstruction by a $3$-torsion element in \cite{cn17} is the only one.

\subsection{Structure of the paper} For ease of reading, we have separated the paper into three parts, which each use independent methods.

In Part 1, we study certain subfamilies of the surfaces $X_{\AA}$ given by imposing various conditions on the coefficients, including those families appearing in Theorem~\ref{thm:main2}. Here we are not concerned with calculating the complete Brauer group; rather we write down explicit elements in \S\ref{sec:Brauer_elements} and perform a systematic study of how often these give a Brauer--Manin obstruction to the Hasse principle. Our proofs proceed by first using geometric methods to determine the image of the local invariant maps and obtain necessary conditions on the coefficients for there to be a Brauer--Manin obstruction. We then count these conditions to obtain an upper bound for the number of such surfaces with a Brauer--Manin obstruction. To obtain lower bounds, we explicitly construct families of counter-examples to the Hasse principle. Part 1 contains essentially all the analytic number theory proofs in the paper.

In Part 2 we calculate the transcendental Brauer groups of the surfaces $X_\AA$ via a spectral sequence framework due to Colliot-Thélène, Skorobogatov and the first author \cite{cts,qua}. As a necessary prerequisite, we completely determine the integral $\ell$-adic middle cohomology of smooth projective weighted diagonal hypersurfaces using techniques from Hodge theory and the equivariant $\ell$-adic Lefschetz trace formula.

In Part 3, we bring everything together to prove the main theorems from the introduction. This requires a careful analysis of the algebraic Brauer groups of the surfaces $X_{\AA}$, which we do using help from \texttt{Magma}.

\subsection{Notation and conventions}

Let $C$ be a (possibly singular and affine) geometrically integral curve over a field $k$. We define the genus of $C$ to be the genus of the smooth compactification of the normalisation of $C$.

We fix a choice of primitive third root of unity $\omega \in \C^\times$.

For complex valued functions $f,g$, we denote by $f \ll g$ the standard Vinogradov notation. We write $f \asymp g$ if $f \ll g$ and $f \gg g$.

We almost always consider $\Z/n\Z$ via its natural embedding $\Z/n\Z \subset \Q/\Z$; in particular we write the elements as $0,1/n,\dots,(n-1)/n$.

For a scheme $X$ we denote by $\Br X = \HH^2(X, \Gm)$ its (cohomological) Brauer group. If $X$ is a variety over a field $k$, then we let 
$$ \Br_0 X = \mathrm{Im}(\Br k \to \Br X), \quad \Br_1 X = \ker(\Br X \to \Br X_{\bar{k}})$$ denote the constant and algebraic parts of the Brauer group of $X$, respectively. An element of $\Br X\setminus \Br_1 X$ is called \emph{transcendental}.

\subsection{Acknowledgements}
We are very grateful to Tim Browning for discussions on some of the analytic arguments and to Igor Dolgachev and Eduard Looijenga for discussions on their work regarding the cohomology of diagonal hypersurfaces. This work was partly undertaken at the Institut Henri Poincaré during the trimester “Reinventing rational points”; the authors thank the organisers and staff for ideal working conditions. The first-named author was supported by EP/L015234/1, the EPSRC Centre for Doctoral Training in Geometry and Number Theory (The London School of Geometry and Number Theory). Part 2 has overlaps with Chapters 11 and 12 of his PhD thesis. The second and third-named authors are supported by EPSRC grant EP/R021422/2.


\part{Counting Brauer--Manin obstructions}

In this part we count the number of surfaces $X_{\AA}$ which have a Brauer--Manin obstruction to the Hasse principle from various explicit algebras given in \S\ref{sec:Brauer_elements}.

\section{Brauer group preliminaries}\label{brauerbasics}
\subsection{The Brauer--Manin obstruction}
We briefly recall the definition of the Brauer--Manin obstruction  and set up related terminology (see \cite[\S8.2]{Poo17} for further background). Let $k$ be a number field.

For a place $v$ of $k$, from class theory we have the local invariant
$$\inv_v: \Br k_v \to \Q/\Z,$$
which is an isomorphism for $v$ non-archimedean, 
and the exact sequence
\begin{equation} \label{eqn:CFT}
0\to\Br k \to\bigoplus_v \Br k_v\xrightarrow{\sum\inv_v}\BQ/\BZ\to 0.
\end{equation}
For a smooth projective variety $X/k$, the \emph{Brauer--Manin pairing} is defined by
$$
\Br X\times X(\Adele_k)\to\BQ/\BZ, \quad
(\CA,(x_v))\mapsto \langle\CA,(x_v)\rangle:=\sum_v \inv_v(\CA(x_v)).
$$
This pairing is locally constant on the right and factors through $\Br X/\Br_0 X$. For a subset $B\subset \Br X$, we define
\[X(\Adele_k)^{B}=\{(x_v)\in X(\Adele_k):\langle\CA,(x_v)\rangle=0\text{ for all }\CA\in B\}.\]
We write $X(\Adele_k)^{\Br}:=X(\Adele_k)^{\Br X}$. By \eqref{eqn:CFT} we have $X(k)\subseteq X(\Adele_k)^{B}\subseteq X(\Adele_k)$.

\begin{definition}
 We say that $X$ has a Brauer--Manin obstruction to the Hasse principle given by $B \subset \Br X$ if $X(\Adele_k)\neq\emptyset$ but $X(\Adele_k)^B=\emptyset$ (and hence $X(k)=\emptyset$).
\end{definition}

\subsection{Cyclic algebras and the Hilbert symbol} \label{sec:cyclic}
\subsubsection{Cyclic algebras}
Let $n \in \N$ and let $k$ be a field of characteristic coprime to $n$ which contains a primitive $n$th root of unity $\zeta \in k$. For $a,b \in k^\times$, we denote by
$(a,b)_\zeta$
the associated cyclic algebra (see \cite[\S2.5]{GS} for its explicit construction). We denote by $(a,b):=(a,b)_{-1}$ (this is a quaternion algebra). In $\Br k$ we have the relations
\begin{equation} \label{eqn:cyclic_properties}
    (aa',b)_\zeta = (a,b)_\zeta +(a',b)_\zeta, \quad (a,b)_\zeta = -(b,a)_\zeta, \quad (a^n,b)_\zeta = 0.
\end{equation}

\subsubsection{The Hilbert symbol}
We now assume that $k$ is a number field and let $v$ be a place of $k$. We define the Hilbert symbol to be
$$(a,b)_{\zeta,v} = \inv_v (a,b)_{\zeta} \in \Q/\Z.$$
This admits a simple expression at almost all places via the tame symbol \cite[Ex.~7.1.5, Prop.~7.5.1]{GS}. We only need to know that for any non-zero prime ideal $\pp$ of $k$ which does not lie over any of the prime divisors of $n$, we have
\begin{equation} \label{eqn:nth_power}
    (u,u')_{\zeta,\pp} = 0 \quad \mbox{and}
    \quad (u,\pi)_{\zeta,\pp}=0 \, \Leftrightarrow
    \, u\in {k_\pp^\times}^n,
\end{equation}
for $u,u' \in \OO_\pp^\times$ and $\pi \in \OO_\pp$ a uniformiser at $\pp$. Here $(a,b)_{-,v}$ is the usual quadratic Hilbert symbol, except by our conventions it takes values in $\Z/2\Z$ rather than $\mu_2$.

\subsection{Some Brauer group elements} \label{sec:Brauer_elements} 

We now write down the explicit Brauer group elements with which we will be working. Recall that $\omega$ denotes a primitive third root of unity. In the statement we take the subscripts modulo $3$.

\begin{proposition}\label{prop:algebras}
    Let $A_i \in \Q$ be non-zero and $X = X_{\mathbf{A}}$ the associated surface over $\Q$ from \eqref{eqn:X_A}. Let $i \in \{1,2,3\}$ and set $j = i+1, k=i+2$.
    \begin{enumerate}
        \item If $-3A_i$ is a square, then
        $$\calA_i=\cores_{\Q(\omega)/\Q}\left(\frac{A_{j}}{A_{k}},\frac{w-\sqrt{A_i}x_i^3}{w+\sqrt{A_i}x_i^3}\right)_\omega \in \Br X.$$
        \item If $A_j/A_k$ is a cube, then
        $$\calB_i=\left(A_i,\frac{x_j^2+\sqrt[3]{A_k/A_j}x_k^2}{x_i^2}\right) \in \Br X.$$
        \item If $27A_j/A_k$ is a sixth power, then
        $$\calC_i=\left(A_i,\frac{x_j^2+\sqrt[6]{27A_k/A_j}x_jx_k+\sqrt[3]{A_k/A_j}x_k^2}{x_i^2}\right)
        \in \Br X.$$        
    \end{enumerate}
    (In the statement subscripts are taken modulo $3$.)
\end{proposition}
\begin{proof} 
    In some special cases, the elements (1) and (2) were found in \cite[\S5.1]{cn17} and \cite[Lem.~3.2]{cn17}, respectively, and verified to be unramified using a computer. We give a more direct  argument using Grothendieck's purity theorem \cite[Thm.~6.8.2]{Poo17}; this states
     that for any smooth variety $Y$ over a field $k$ of characteristic $0$ the sequence
    $$0\to \Br Y \to \Br k(Y)\to \oplus_{D \in Y^{(1)}} \RH^1(k(D),\Q/\Z)$$
    is exact,
   where the last map is given by the residue along the codimension one point $D$. Hence, to prove that our algebras come from a class in $\Br X$, it suffices to show that 
    all their residues are trivial.
    
    (1) We will show that
    \begin{equation} \label{def:A'}
    \calA_i'=\left(\frac{A_j}{A_k},\frac{w-\sqrt{A_i}x_i^3}{w+\sqrt{A_i}x_i^3}\right)_\omega \in \Br X_{\Q(\omega)},
    \end{equation}
    so that its corestriction is  a well-defined element over $\Q$. Any non-trivial residue of $\calA_i'$
    must occur along an irreducible component of one of the two divisors $$D_{\pm}: w=\pm\sqrt{A_i}x_i^3, \quad A_jx_j^6 + A_kx_k^6 = 0.$$
    However, clearly in the function field of any such irreducible component $-A_k/A_j$ is a sixth power, thus $A_k/A_j$ is a cube; standard
    formulae for residues in terms of the tame symbol  \cite[Ex.~7.1.5, Prop.~7.5.1]{GS} therefore
    show that $\calA_i$ is unramified.
    
    (2) (resp.~(3)) The residue along $x_i^2 = 0$ is clearly trivial. Next, note that $x_j^2+\sqrt[3]{A_k/A_j}x_k^2$ (resp.~$x_j^2+\sqrt[6]{27A_k/A_j}x_jx_k+
    \sqrt[3]{A_k/A_j}x_k^2$)
    divides $x_j^6+A_k/A_jx_k^6$. 
    Thus in the function field
    of any irreducible component of the corresponding divisor, we have the relation
    $w^2 = A_ix_i^6$, so $A_i$ is a square.
    Again, the algebra is thus unramified.
\end{proof} 

In the following sections, we study how often various combinations of these elements give a Brauer--Manin obstruction to the Hasse principle.

\section{Counting preliminaries} \label{sec:counting}
We now gather some of the analytic number theory results we will require.

\subsection{Frobenian multiplicative functions}
For our work, we will use the theory of frobenian multiplicative functions introduced in \cite{LM19}, and based on a concept of Serre \cite[\S3.3]{Ser12}. We recall the following definition from \cite[\S 2]{LM19}.

\begin{definition} \label{def:frob}
    Let  $\rho: \N \to \C$ be a  multiplicative function. 
    We say that $\rho$ is a frobenian multiplicative function if
    \begin{enumerate}
		\item $|\rho(n)| \ll_\varepsilon n^\varepsilon$ for all $n \in \N$ and all $\varepsilon > 0$;
		\item there exists $H \in \N$ such that $|\rho(p^k)| \leq H^k$ for all primes $p$ and all $k \geq 1$;
	\end{enumerate}
	and if there exist 
	\begin{enumerate}
		\item[(a)] a finite Galois extension $K/\Q$, with Galois group $\Gamma$;
		\item[(b)] a finite set of primes $S$ containing all the primes ramifying in $K$;
		\item[(c)] a class function $\varphi: \Gamma \to \C$;
	\end{enumerate}
	such that for all $p \not \in S$ we have
	$\rho(p) = \varphi(\Frob_p)$
	where $\Frob_p \in \Gamma$ is the Frobenius element of $p$.
	We define the \emph{mean} of $\rho$ to be 
	$$m(\rho) = \frac{1}{|\Gamma|} \sum_{\gamma \in \Gamma}\varphi(\gamma).$$ 
\end{definition}

\subsection{Lower bounds}
Our first counting result is lower bounds for sums of frobenian multiplicative functions with congruences imposed.

\begin{proposition} \label{prop:lower_bound_frob}
    Let $M \in \N$ and let $\rho$ be a real-valued non-negative 
    frobenian multiplicative function
    which is completely multiplicative
    with $m(\rho) \neq 0$. Then
	$$\sum_{\substack{n \leq x \\ n \equiv 1 \bmod M}}\rho(n)
	\gg_{\rho,M} x(\log x)^{m(\rho) - 1}.$$ 
\end{proposition}
\begin{proof}
	Note that it is not even a priori clear that the sum is non-zero (cf.~Remark \ref{rem:bad}), but our proof will show this.
	We prove the result using some of the results from \cite[\S 2]{LM19}.
	
	We may assume that $S$ in Definition \ref{def:frob}
	contains all $p \mid M$.
    Since we are only interested in a lower bound,
    we are free to replace $\rho$ by a different function which is majorised
    by $\rho$. Thus we may assume that
    \begin{equation} \label{eqn:rho}
        \rho(p) = 0 \quad \mbox{ for all } p \in S.
    \end{equation}
    We have
	$$\sum_{\substack{ n \leq x \\ n \equiv 1 \bmod M}} \rho(n)
	= \frac{1}{\varphi(M)}\sum_{\chi \bmod M}
	\sum_{\substack{ n \leq x}} \chi(n) \rho(n),$$
	where the sum is over all Dirichlet characters $\chi$ modulo $M$.
	Each $\chi \cdot \rho$ is also a frobenian multiplicative
	function (see \cite[Ex.~2.2]{LM19}). If $|m(\chi\rho)| < m(\rho)$,
	then \cite[Lem.~2.8]{LM19} implies that 
	$$\sum_{\substack{ n \leq x}} \chi(n) \rho(n) = o(x(\log x)^{m(\rho) - 1}),$$ thus is negligible. So assume $|m(\chi\rho)| \geq m(\rho)$.
	Then \cite[Lem.~2.4]{LM19} implies that 
	\begin{equation} \label{eqn:rho(p)}
	    \rho(p)\chi(p) = \rho(p)
	\end{equation}
	for all but finitely many primes $p$. We claim that in fact
	\eqref{eqn:rho(p)} holds \emph{for all} primes $p$. For
	$p \in S$ this follows from \eqref{eqn:rho}. For $p\notin S$,
	by our choice of $S$, we have that $\chi(p)$ and $\rho(p)$ only depend
	on $p \bmod M$ and $\Frob_p$, respectively. As \eqref{eqn:rho(p)}
	holds for all but finitely many $p$, the Chebotarev density
	theorem now implies that \eqref{eqn:rho(p)} holds for all $p \notin S$.
	As $\chi$ and $\rho$ are both completely multiplicative,
	it follows that $\rho(n)\chi(n) = \rho(n)$ for all $n \in \N$.
	Thus the contribution from such a $\chi$ is just
	$$\sum_{\substack{ n \leq x}} \chi(n) \rho(n) =\sum_{\substack{ n \leq x}} \rho(n),$$
	i.e.~indistinguishable from the principal character.
	Applying \cite[Lem.~2.8]{LM19} again, one sees that 
	the contribution from the principal character is 
	$\gg_{\rho,M} x(\log x)^{m(\rho) - 1}$, which completes
	the proof.
\end{proof}

\begin{remark} \label{rem:bad}
    One cannot replace the condition $n \equiv 1 \bmod M$ in Proposition \ref{prop:lower_bound_frob}
    by an arbitrary congruence condition in general.
    Let $\rho$  be the indicator function for integers entirely
    composed of primes which are $1 \bmod 4$;
    this is easily seen to be frobenian of positive mean
    and completely multiplicative. But
    $$\sum_{\substack{n \leq x \\ n \equiv 3 \bmod 4}} \rho(n) = 0.$$
\end{remark}

\subsection{Upper bounds}
Proposition \ref{prop:lower_bound_frob} is a general lower bound result for sums of frobenian functions. We will also need upper bounds, but crucially with uniformity  in the discriminant of the field extension $K/\Q$ appearing in Definition \ref{def:frob}. Rather than state the most general technical result we can, we focus on the special cases  which will be of interest. For this, we require the following effective version of the Chebotarev density theorem.

\begin{lemma} \label{lem:Cheb}
	Let $d \in \N$, let $k/\Q$ be a finite Galois extension of degree $d$ with discriminant
	$\Delta_k$ and let $C \subseteq \Gal(k/\Q)$ be a conjugacy invariant subset.
	Then there exist constants $c_1,c_2> 0$, depending only on $d$,
	such that 
	$$\#\left\{p \leq x : 
	\begin{array}{l}
		p \mbox{ unramified in k},\\
		\Frob_p \in C
	\end{array}
	\right\} 
	= \frac{|C|}{d}\left( \Li(x) + \Li(x^{\beta_k})\right)
	+  O\left(x \exp(-c_1(\log x)^{1/2})\right)$$
	for all $\log x \geq c_2 (\log |\Delta_k|)^2$.
	
	Here $\beta_k$ denotes the (possible) exceptional zero of $\zeta_k(s)$, which is real, positive, and simple,
	and we omit the corresponding term if no such zero exists.
\end{lemma}
\begin{proof}
 	See \cite[Thm.~1.3]{LO77} or \cite[\S2.2]{Ser81}.
\end{proof}

Our upper bound result is now as follows.

\begin{proposition} \label{prop:cube}
	Let $A, x > 0$ and let $S$ be a finite set of primes.
	Then the following holds as $x \to \infty$,
	uniformly for all non-cubes $r \leq (\log x)^A$.
	\begin{enumerate}			
	\item $$\#\{ n \leq x : p \notin S, p \nmid r, \mbox{ and } p \mid n \implies r \in \Q_p^{\times3}\}
	\ll_{A,S} \frac{x}{(\log x)^{1/3}}.$$ \label{eqn:1}
	\item $$\#\left\{ n \leq x : 
	\begin{array}{l}
	\mbox{There is at most one prime } q \mbox{ with } q \parallel  n, \\
	q \notin S, \mbox{ and } q \nmid r \mbox{ such that } r \notin \Q_q^{\times3}
	\end{array}
	\right\} 
	\ll_{A,S} \frac{x(\log \log x)}{(\log x)^{1/3}}.$$ \label{eqn:2}
	\end{enumerate}
\end{proposition}
\begin{proof}
    We may assume that $3 \in S$. We 
    let $S_r = S \cup \{p \mid r\}.$
	We first prove \eqref{eqn:1}.
	Let $\varpi_r(n)$ be the indicator function for those $n$ being counted;
	this is multiplicative.
	General uniform upper bounds for sums of multiplicative
	functions (for example  \cite[Thm.]{Nai92} with $P(n) = n, \varrho(p) = 1$ and $y = x$), gives
	\begin{equation} \label{eqn:Shiu}
	\sum_{n \leq x}\varpi_r(n) \ll \frac{x}{ \log x} 
	\exp \left(\sum_{p \leq x}\frac{\varpi_r(p)}{p}\right)
	\end{equation}
	uniformly for all $r$. We handle 
	this using Lemma \ref{lem:Cheb}.
	The function $\varpi_r$ is a frobenian
	multiplicative function. Indeed,
	let $k_r=\Q(r^{1/3},\mu_3)$. For $p \nmid 3r$
	we have $r \in \Q_p^{\times3}$ if and only if $\Frob_p \in \Gal(k_r/\Q)$
	acts with a fixed point on the roots of $x^3 - r$. This is a conjugacy
	invariant subset of $\Gal(k_r/\Q)$, which we denote by $C_r$.
	We have $|C_r| = 4$, thus $m(\varpi_r) = 2/3$.
	
	By our assumptions $|\Delta_{k_r}| \leq 3^9 r^4 \ll (\log x)^{4A}$, so we may apply Lemma \ref{lem:Cheb}.
	There could be problems with uniformity
	if the exceptional zero $\beta_{k_r}$ is extremely close to $\re s = 1$. However,
	we  show that there is in fact no exceptional zero. A result of Heilbronn 
	\cite[Thm.~1]{Hei73} says that for a 
	Galois extension $k/\Q$, if $\beta_k$ exists then $\zeta_L(\beta_k) = 0$
	for some quadratic subfield $\Q \subset L \subset k$. But our number field $k_r$ has a 
	unique quadratic subfield, namely $\Q(\mu_3)$, and it is well-known that 
	$\zeta_{\Q(\mu_3)}(s) = \zeta(s) L(\left(\frac{-3}{\cdot}\right),s)$ has no exceptional zero (see e.g.~\cite[Dirichlet character 3.2]{LFMFDB}).
	Thus there is no exceptional zero $\beta_{k_r}$.
	Lemma \ref{lem:Cheb} thus gives
	$$\sum_{p \leq x}\varpi_r(p) 
	= \#\{ p \leq x : p \notin S_r, \Frob_p \in C_r\}
	= \frac{2}{3}\Li(x) +  O_{A,S}\left(x \exp(-c_1(\log x)^{1/2})\right).$$
	By partial summation we thus obtain
	$$\sum_{p \leq x}\frac{\varpi_r(p)}{p} = 
	\frac{2}{3} \log \log x + O_{A,S}(1).$$
	Part \eqref{eqn:1} now easily follows from \eqref{eqn:Shiu}.
	
	For \eqref{eqn:2}, every such $n$ can be written as $n=qms$,
	where $s$ is squarefull, every prime $p \mid m$ with $p \notin S_r$
	satisfies $r \in \Q_p^{\times3}$, and $q = 1$ or $q$ is a prime
	with $q \notin S_r$.
	We first consider the case where $q$ is a prime.	
	Here the quantity is bounded above by
	\begin{equation} \label{eqn:qm}
		\sum_{\substack{s \leq x \\ s \text{ squarefull} }}
		\sum_{\substack{q \leq x \\ q \text{ prime} \\ q \notin S_r}}
		\sum_{qms \leq x} \varpi_r(m).
	\end{equation}
	We begin by showing that we can take $s$ to be small.
	Namely let $1/2 > \varepsilon > 0$.
	Then the contribution from $s \geq x^{1/3}$ is
	$$ \ll \sum_{\substack{x^{1/3} \leq  s \leq x
	\\ s \text{ squarefull} }}
	\sum_{qm \leq x/s}1 \ll 
	\sum_{\substack{x^{1/3} \leq  s \leq x
	\\ s \text{ squarefull} }} \frac{x(\log x)}{s}
	\ll \frac{x(\log x) }{x^{(1/2-\varepsilon)/3}}
	\sum_{\substack{x^{1/3} \leq  s \leq x
	\\ s \text{ squarefull} }} \frac{1}{s^{1/2 + \varepsilon}}$$
	which is  negligible as the latter sum is easily seen to be
	convergent. Therefore we may assume that $s \leq x^{1/3}$
	(in fact, any exponent smaller than $1/2$ will be sufficient
	in what follows).
	To continue we use Dirichlet's hyperbola
	method. We first consider the case where $q \leq  x^{1/2}$.
	In this case $x/qs \gg x^{1/6} \to \infty$ as $x \to \infty$. It follows from
	\eqref{eqn:1} that the corresponding
	contribution to \eqref{eqn:qm} is
	$$\ll \sum_{\substack{s \leq x^{1/3} \\ s \text{ squarefull} }}
		\sum_{\substack{q \leq x^{1/2} \\ q \text{ prime} \\
		q \notin S_r}}
	\frac{x}{qs (\log x/qs)^{1/3}}
	\ll_{A,S} \frac{x(\log \log x)}{(\log x)^{1/3}},$$
	by Mertens' theorem and the fact that the sum over $s$ is convergent. 
	We now consider the case $m \leq x^{1/2}$. 
	Here we change the order of summation 
	in \eqref{eqn:qm} to obtain
	\begin{align*}
	\sum_{\substack{s \leq x^{1/3} \\ s \text{ squarefull} }}
	\sum_{\substack{ m \leq x^{1/2}}} \varpi_r(m)
	\sum_{\substack{q \leq x/ms \\ q \text{ prime} \\
	q \notin S_r} } 1
	&\ll
	\sum_{\substack{s \leq x^{1/3} \\ s \text{ squarefull} }}	
	\sum_{\substack{ m \leq x^{1/2}}}  \varpi_r(m)
	\frac{x}{ms \log(x/ms)}  \\
	&\ll \frac{x}{\log x} 
	\sum_{\substack{ m \leq x^{1/2}}} 
	\frac{\varpi_r(m)}{m} \ll \frac{x}{(\log x)^{1/3}}
	\end{align*}
	where the last step is by \eqref{eqn:1} and partial summation
	(using that  the integral of $1/t(\log t)^{1/3}$ is 
 	$(3/2)(\log t)^{2/3}$). This completes the proof for $q$ prime.
 	A similar argument for $q =1$ shows that the
 	contribution here is $O(x/(\log x)^{1/3})$.
\end{proof}

\begin{lemma} \label{lem:square}
	Let $A, x > 0$ and let $S$ be a finite set of primes.
	Uniformly with respect to all non-squares 
	$r_1,r_2,r_3 \leq (\log x)^A$, the following holds
	\begin{align*}
	& \left\{ n \leq x : p \notin S, p \nmid r_i, p \parallel  n \mbox{ and }
	\left(\frac{r_1}{p}\right)=
	-1\implies 
	\left(\frac{r_2}{p}\right) = 1
	\mbox{ and }
	\left(\frac{r_3}{p}\right)=
	-1
	\right\}  \\
	&\qquad  \ll_{A,S} 
	\begin{cases}
		\frac{x}{(\log x)^{1/2}}, & \mbox{if }r_1r_2 
		\mbox{ or } r_2r_3 \mbox{ is a square}, \\
		\frac{x}{(\log x)^{1/4}}, & \mbox{if } r_1r_3
		\mbox{ or }r_1r_2r_3 \mbox{ is a square, and } r_1r_2 \mbox{ is not a square}, \\
		\frac{x}{(\log x)^{3/8}}, & \mbox{otherwise}.
	\end{cases}
	\end{align*}
\end{lemma}
\begin{proof}
The proof is a variant of the proof of Proposition \ref{prop:cube}.(1). However instead of using the effective Chebotarev density theorem, we use the more classical Siegel--Walfisz theorem \cite[Cor.~11.21]{MV07}, which gives better control of exceptional zeros.

Let $S_{\mathbf{r}} = S \cup \{p \mid 2r_1r_2r_3\}$.
We denote the set of primes under consideration by
$$\mathcal{P} = \left\{p \notin S_{\mathbf{r}}: \left(\frac{r_1}{p}\right)= 1 \mbox{ or } \left(\frac{r_1}{p}\right)= -1,
\left(\frac{r_2}{p}\right)= 1,	\mbox{ and } \left(\frac{r_3}{p}\right)= -1\right\}.$$
By quadratic reciprocity, there is a set of elements $R \subset (\Z/(4^3 \cdot r_1r_2r_3)\Z)^\times$  such that for $p \notin S_{\mathbf{r}}$ we have $p \in \mathcal{P}$ if and only if $p \in R$.  As in the proof of Proposition \ref{prop:cube}, we apply \cite[Thm.]{Nai92} and Siegel--Walfisz to obtain the upper bound
\[\ll_{A,S} x(\log x)^{\delta-1}\]
where $\delta = \dens(\mathcal{P})$. It therefore suffices to calculate $\delta$ in each of the stated cases, bearing in mind that $r_1,r_2, r_3$ are not squares. We summarise this in the following table, which easily gives the result ($\times$ denotes a non-square and $\square$ denotes a square).
\[
\begin{array}{c|c|c|c|c|c|c|c|}
  \delta & r_1 & r_2 & r_3  & r_1r_2 & r_1r_3 & r_2r_3 & r_1r_2r_3 \\ \hline
  5/8 & \times & \times & \times & \times & \times & \times & \times  \\ \hline
  1/2 & \times & \times & \times & \square & \times & \times & \times \\ \hline   
  3/4 & \times & \times & \times & \times & \square & \times & \times \\ \hline   
  1/2 & \times & \times & \times & \times & \times & \square & \times \\ \hline     
  3/4 & \times & \times & \times & \times & \times & \times & \square \\ \hline       
  1/2 & \times & \times & \times & \square & \square & \square & \times \\ \hline 
\end{array}
\]
\end{proof}

\subsection{Sums of radicals}
We finish with the following elementary results.

\begin{lemma} \label{lem:rad}
    For any $\varepsilon > 0$, the following
    sum is convergent
    $$\sum_{n=1}^\infty \frac{1}{\rad(n)n^\varepsilon}.$$
\end{lemma}
\begin{proof}
    We have
	\begin{equation*}
		\sum_{n=1}^\infty \frac{1}{\rad(n)n^\varepsilon}
		= \prod_p\left(1 + \frac{1}{p^{1+\varepsilon}}
		+ \frac{1}{p^{1+2\varepsilon}} + \dots \right)
		\ll \prod_p\left(1 + \frac{1}{p^{1+\varepsilon/2}}\right)
		\ll_{\varepsilon} 1. \qedhere
	\end{equation*}
\end{proof}

\begin{lemma} \label{lem:rad2}
    Let $a_1,a_2 >0$ be such that $a_1 + a_2 > 1$.
    Then the following sum is convergent
    $$\sum_{\substack{n_1,n_2 \in \N \\ \rad(n_1) = \rad(n_2)}} \frac{1}{n_1^{a_1}n_2^{a_2}}.$$
\end{lemma}
\begin{proof}
    Via a multidimensional Euler product expansion,
    we obtain
    \[\sum_{\substack{n_1,n_2 \in \N \\ \rad(n_1) = \rad(n_2)}} \frac{1}{n_1^{a_1}n_2^{a_2}}= 
    \prod_p\left(1 + 
    \sum_{k_1,k_2=1}^\infty\frac{1}{p^{a_1k_1 + a_2k_2}}\right) \ll_{a_1,a_2} 1. \qedhere\]
\end{proof}

\section{Brauer--Manin obstruction from an order three algebra}

\subsection{Set-up}
Let $A_1,A_2,A_3$ be non-zero integers and consider 
$$X := (X_{\AA}): \quad w^2=A_1x_1^6+A_2x_2^6+A_3x_3^6 \quad \subset \P(3,1,1,1).$$
The main result in this section is the following, which counts how often the algebras $\calA_i$ from \S\ref{sec:Brauer_elements} gives a Brauer--Manin obstruction on $X$. For simplicity of exposition, we focus on the case where $-3A_1$ is a square.

\begin{theorem} \label{thm:2}
	There are $\asymp T^{3/2} \log \log T/(\log T)^{2/3}$ 
	non-zero integers
	$A_1,A_2,A_3$ such that $\max\{|A_i|\} \leq T$,
	$-3A_1$ is a square, and there is a Brauer--Manin
	obstruction to the Hasse principle given by $\calA:=\calA_1$.
\end{theorem}

The proof of Theorem \ref{thm:2} and our other counting results have a similar structure. The main saving comes from a valuative criterion for there to be a Brauer--Manin obstruction (Proposition \ref{prop:surjective_2}); this is already enough to deduce the upper bound $T^{3/2 + \varepsilon}$ in Theorem \ref{thm:2}. To get control of the logarithms one needs more delicate criteria concerning the types of primes which are allowed to divide the coefficients (Lemma \ref{lem:almostsurj} and Proposition \ref{prop:surjective_2}). These criteria will be enough to obtain the correct order upper bound in Theorem \ref{thm:2}. To obtain the right lower bound, we construct an explicit family of counter-examples. 

Theorem \ref{thm:2} will be used in the proof of the count for $N_2(T)$ in Theorem \ref{thm:main2}. However Theorem \ref{thm:2} only considers the special Brauer group element $\mathcal{A}$, but in Theorem \ref{thm:main1} in general there will be other Brauer group elements; we will show in \S\ref{sec:proofs} that asymptotically, these other elements will not affect the main term.

Throughout this section we set $K=\Q(\omega)$.

\subsection{Local invariants}

\subsubsection{Formula}

In this section we study in detail the possibilities for the local invariants of $\calA$ at a prime $p$. Specifically, we study the map
$$\inv_p \calA : X(\Q_p) \to \Z/3\Z, \quad x \mapsto \inv_p \mathcal{A} (x).$$
We denote by $\inv_p \calA(X(\Q_p))$ the image of this map.
We calculate this in terms of the cubic Hilbert symbol from \S \ref{sec:cyclic} using the following lemma. Note that in order to make sense of the formula in Lemma \ref{lem:cubic_invariant}, we will always implicitly assume that $w^2-A_1x_1^6 \neq 0$ when computing invariants. We can do this without loss of generality since $\inv_p\calA(\cdot)$ is continuous with respect to the $p$-adic topology. 

\begin{lemma} \label{lem:cubic_invariant}
    Let $p \neq 3$ be a prime, $x \in X(\Q_p)$
    and $\pp \mid p$ a 
    prime of $\OO_K$ above~$p$.
    $$\inv_p \calA(x) = \beta \cdot  \left(\frac{A_2}{A_3},\,\frac{w-\sqrt{A_1}x_1^3}{w+\sqrt{A_1}x_1^3}\right)_{\omega, \pp},
    \quad \mbox{ where }\beta=
    \begin{cases}
        2, & p \equiv 1 \bmod 3, \\
        1, & p \equiv 2 \bmod 3.
    \end{cases}$$
\end{lemma}
\begin{proof}
Recall that our algebra is given  by a corestriction from $K = \Q(\omega)$ to $\Q$ of
$$
    \calA'=\left(\frac{A_2}{A_3},\frac{w-\sqrt{A_1}x_1^3}{w+\sqrt{A_1}x_1^3}\right)_\omega \in \Br X_{\Q(\omega)},
$$
It thus follows from \cite[Prop.~II.1.4]{Neu13} and \cite[Prop.~III.3.3]{Neu13} that
$$\inv_p \calA(x) = \sum_{\qq \mid p } \inv_{\qq} \left(\frac{A_2}{A_3},\,\frac{w-\sqrt{A_1}x_1^3}{w+\sqrt{A_1}x_1^3}\right)_{\omega},$$
where the sum is over \emph{all} primes ideals $\qq$ of $\OO_K$ over $p$. Therefore in terms of the Hilbert symbol we have
$$
	\inv_p \calA(x) = \sum_{\qq \mid p }  \left(\frac{A_2}{A_3},\,\frac{w-\sqrt{A_1}x_1^3}{w+\sqrt{A_1}x_1^3}\right)_{\omega,\qq}
$$
For $p \equiv 2 \bmod 3$, there is a unique prime over $p$. So assume that $p \equiv 1 \bmod 3$, where there are two primes ideals $\pp_1, \pp_2$ above $p$. It suffices to show that
\begin{equation}\label{eqn:galoisInv}
\left(\frac{A_2}{A_3},\,\frac{w-\sqrt{A_1}x_1^3}{w+\sqrt{A_1}x_1^3}\right)_{\omega,\pp_1} = \left(\frac{A_2}{A_3},\,\frac{w-\sqrt{A_1}x_1^3}{w+\sqrt{A_1}x_1^3}\right)_{\omega,\pp_2}.
\end{equation}\label{eqn:invFormula1mod3}
Let $\sigma\in\Gal(K/\Q)$ be a generator. Then in $\Br X_K$ we have
\begin{align*}
    \sigma(\calA') =&  \left(\sigma\left(\frac{A_2}{A_3}\right),\,
    \sigma \left(\frac{w-\sqrt{A_1}x_1^3}{w+\sqrt{A_1}x_1^3}\right)\right)_{\sigma(\omega)} 
    = \left(\frac{A_2}{A_3},\,
    \frac{w+\sqrt{A_1}x_1^3}{w-\sqrt{A_1}x_1^3}\right)_{\omega^2} \\
    & = -\left(\frac{A_2}{A_3},\,
    \frac{w+\sqrt{A_1}x_1^3}{w-\sqrt{A_1}x_1^3}\right)_{\omega} 
     = - \left(\frac{A_2}{A_3},\,
    \left(\frac{w-\sqrt{A_1}x_1^3}{w+\sqrt{A_1}x_1^3}\right)^{-1}\right)_{\omega} 
     = \calA',
\end{align*}
hence $\calA'$ is Galois invariant. 
Consider next the following diagram.
\begin{align*}
\xymatrixcolsep{4pc}\
\xymatrix{
\Br X_K \ar[d]^{x_p} \ar[rr]^{\sigma} && \Br X_K \ar[d]^{x_p} \\\
\Br K_{\pp_1} \ar[d]^{\inv_{\pp_1}} \ar[r]_{\simeq} & \Br \Q_p \ar[d]^{\inv_p} & \Br K_{\pp_2} \ar[d]^{\inv_{\pp_2}} \ar[l]^{\simeq} \\\
\mathbb{Q}/\mathbb{Z} \ar[r]_{\text{id}} & \mathbb{Q}/\mathbb{Z}   & \ar[l]^{\text{id}} \mathbb{Q}/\mathbb{Z}} 
\end{align*}
This diagram commutes: the commutativity of the top square follows from the definition of the embeddings $K\to K_{\pp_i}$, and that of the bottom squares by \cite[Prop.~II.1.4]{Neu13}. As $\calA'$ is Galois invariant \eqref{eqn:galoisInv} follows immediately.
\end{proof}

We recall that
\begin{equation} \label{eqn:2mod3}
	\Q_p^{\times 3} = \{ x \in \Q_p^\times : v_p(x) \equiv 0 \bmod 3 \}, \quad \text{ for } p \equiv 2 \bmod 3.
\end{equation}

\begin{lem}\label{surjlem} 
Let $p$ be a prime and $\pp$ a prime ideal of $\OO_K$ over $p$.
\begin{enumerate}
	\item For $p \equiv 1 \bmod 3$, the following map is surjective:
	$$(\ \cdot \ ,p)_{\omega,\pp}\colon \OO_\pp^\times\to \Z/3\Z, 
	\quad u \mapsto (u,p)_{\omega,\pp}.$$
	\item For $p \equiv 2 \bmod 3$, the following map is surjective:
	$$(\ \cdot \ ,p)_{\omega,\pp}\colon \{u \in \OO_\pp^\times: \Norm_{K_\pp/\Q_p}(u) = 1\}
	\to \Z/3\Z, \quad u \mapsto (u,p)_{\omega,\pp}.$$
\end{enumerate}
\end{lem}

\begin{proof}
    Taking $u \in \OO_\pp^\times$ to be a cube, we obtain the value $0$ by \eqref{eqn:cyclic_properties}.
    Taking $u \in \OO_\pp^\times$ to be a non-cube, we obtain a non-zero
    value by \eqref{eqn:nth_power}. But then $(u^2,p)_{\omega,\pp} = -(u,p)_{\omega,\pp}$
    by \eqref{eqn:cyclic_properties}.
    This shows that $(\ \cdot \ ,p)_{\omega,\pp}\colon \OO_\pp^\times \to \Z/3\Z$ is surjective
    in both cases. So let $p \equiv 2 \bmod 3$. For $u \in \OO_p^\times$ set
	$$n = \frac{u^4}{\Norm_{K_\pp/\Q_p}(u^2)}.$$
	Then clearly $\Norm_{K_\pp/\Q_p}(n) = 1$. As $p \equiv 2 \bmod 3$
	we have $\Norm_{K_\pp/\Q_p}(u^2) \in \Q_p^{\times3}$ by \eqref{eqn:2mod3}. Therefore 
	$(u,p)_{\omega,\pp} = (n,p)_{\omega,\pp},$ and the lemma follows by the above result.
\end{proof}

\subsubsection{The image of the local invariant map}
We use throughout the formula for the local invariant given in Lemma \ref{lem:cubic_invariant}.

\begin{lemma} \label{lem:trivial_2}
	let $p \neq 3$ be a prime such that $v_p(A_1) \equiv v_p(A_2)\equiv v_p(A_3) \bmod 3$
	and $X(\Q_p) \neq \emptyset$. Then $\inv_p \calA(X(\Q_p)) = 0$.
\end{lemma}
\begin{proof}
	Let $\pp\mid p$ be a prime ideal of $\calO_K$.  If $p \equiv 2 \bmod 3$,
	then since $v_p(A_2/A_3) \equiv 0 \bmod 3$ we have
	$A_2/A_3 \in \Q_p^{\times3}$ by \eqref{eqn:2mod3} and so clearly $A_2/A_3\in K_\pp^{\times3}$. Hence, the
	invariant is  trivial by \eqref{eqn:cyclic_properties}.
	Now assume that $p \equiv 1 \bmod 3$. 
	If $A_2/A_3$ is a cube in $K_\pp$, then the local invariant is clearly trivial.
	So assume that $A_2/A_3$ is not a cube in $K_\pp$. Here we will show that
	\begin{equation} \label{eqn:0mod3}
	v_\pp\left(\frac{w-\sqrt{A_1}x_1^3}{w+\sqrt{A_1}x_1^3}\right)
	\equiv 0 \bmod 3,
	\end{equation}
	which clearly implies the lemma (recall our implicit assumption
	that $w^2 \neq A_1x_1^6$).
	Rearranging the equation of $X$ gives
	$$w^2/A_1-x_1^6=A_2/A_1x_2^6 + A_3/A_1x_3^6.$$
	Since $A_2/A_3$ is not a cube in $K_\pp$, the right hand side has $\pp$-adic valuation divisible by $3$.
	If $v_\pp(w/\sqrt{A_1})\neq v_\pp(x_1^3)$, then $v_\pp(w/\sqrt{A_1}+x_1^3)=v_\pp(w/\sqrt{A_1}-x_1^3)=\min\{v_\pp(w/\sqrt{A_1}),v_\pp(x_1^3)\}$ which implies \eqref{eqn:0mod3}. Now assume $v_\pp(w/\sqrt{A_1})= v_\pp(x_1^3)$. Then either $v_\pp(w/\sqrt{A_1}+x_1^3)=v_\pp(x_1^3)\equiv0\bmod3$ or $v_\pp(w/\sqrt{A_1}-x_1^3)=v_\pp(x_1^3)\equiv0\bmod3$ since $p\neq2$. But since 
	\begin{equation*}
	v_\pp(w/\sqrt{A_1}-x_1^3)+v_\pp(w/\sqrt{A_1}+x_1^3)=v_\pp(w^2/A_1-x_1^6)\equiv0\bmod3,
	\end{equation*}
	it follows that $v_\pp(w/\sqrt{A_1}+x_1^3)\equiv v_\pp(w/\sqrt{A_1}-x_1^3)\equiv0\bmod3$, whence \eqref{eqn:0mod3}.
\end{proof}

\begin{lemma} \label{lem:almostsurj}
	Let $p \neq 3$ be a prime such that 
	\begin{enumerate}
		\item either $v_p(A_2/A_1)\not\equiv 0\bmod 3$
		or $v_p(A_3/A_1) \not\equiv 0\bmod 3$;
		\item $v_p(A_2/A_3)\equiv0\bmod 3$;
		\item $A_2/A_3$ is not a cube in $\Q_p$.
	\end{enumerate}
	Then $\{1/3,2/3\} \subseteq \inv_p \calA(X(\Q_p))$.
\end{lemma}

\begin{proof}
Hypotheses (2) and (3) imply that $p \equiv 1 \bmod 3$. Let $\pp\mid p$ be a prime ideal of $\calO_K$.
By (1) there exist $x_2,x_3 \in \Z_p$ such that
$$v_p(A_2/A_1x_2^6+A_3/A_1x_3^6)>0, \quad v_p(A_2/A_1x_2^6+A_3/A_1x_3^6) \not \equiv 0 \bmod 3.$$
For any $p$-adic unit $x_1$, we have $x_1^6+A_2/A_1x_2^6+A_3/A_1x_3^6\equiv x_1^6\bmod p$ which is a square, so $x_1^6+A_2/A_1x_2^6+A_3/A_1x_3^6 \in \Q_p^{\times2}$. Hence there exists $w\in\Q_p$ such that $w^2/A_1=x_1^6+A_2/A_1x_2^6+A_3/A_1x_3^6$, as $A_1 \in \Q_p^{\times2}$. Thus $v_p(w^2/A_1-x_1^6)=v_p(A_2/A_1x_2^6+A_3/A_1x_3^6)$, and one of $v_p(w/\sqrt{A_1}-x_1^3)$ or $v_p(w/\sqrt{A_1}+x_1^3)$ must be $0$ since $v_p(x_1)=0$. Therefore,
$$v_\pp\left(\frac{w-\sqrt{A_1}x_1^3}{w+\sqrt{A_1}x_1^3}\right)=\pm v_\pp(A_2/A_1x_2^6+A_3/A_1x_3^6)\not \equiv 0\bmod 3$$
Thus, in terms of the Hilbert symbol we have
$$\inv_\frakp\left(\frac{A_2}{A_3},\,\frac{w-\sqrt{A_1}x_1^3}{w+\sqrt{A_1}x_1^3}\right)_{\omega}=(A_2/A_3, p^i)_{\omega,\pp}$$
where $i=1$ or $2$. By assumptions (2) and (3) and \eqref{eqn:nth_power}, this is non-trivial. Now take a new solution by replacing $x_1$ by $-x_1$ to get
$$\inv_\pp\left(\frac{A_2}{A_3},\,\frac{w-\sqrt{A_1}x_1^3}{w+\sqrt{A_1}x_1^3}\right)_{\omega}=(A_2/A_3, p^{-i})_{\omega,\pp}=-(A_2/A_3, p^{i})_{\omega,\pp}.$$
Hence it follows that $\inv_\pp$ takes at least 2 non-trivial values. Thus, it follows that $\inv_p$ also takes at least 2 non-trivial values by Lemma \ref{lem:cubic_invariant}.
\end{proof}

\begin{proposition} \label{prop:BC}
	Let $p \neq 3$ be a prime such that 
	\begin{enumerate}
		\item $p \nmid A_1,$ $p \parallel A_2$ and $p \parallel  A_3$;
		\item $A_2/A_3$ is not a cube in $\Q_p$.
	\end{enumerate}
	Then $\{1/3,2/3\} = \inv_p \calA(X(\Q_p))$.
\end{proposition}
\begin{proof}
	Hypotheses (1) and (2) imply that $p \equiv 1 \bmod 3$.
	Let $(w,x_1,x_2,x_3) \in X(\Z_p)$. We may assume that this point is primitive,
	in the sense that $p$ does not simultaneously divide all $x_i$.
	Indeed, otherwise we would find that $p^3 \mid w$, so that we can 
	remove all these powers of $p$ in weighted homogeneous coordinates.
	
	By Lemma \ref{lem:almostsurj},
	it suffices to show that the invariant  does not take the value
	$0$. Let $\pp\mid p$ be a prime ideal of $\calO_K$. From hypotheses (1) and (2), we see that $v_p(A_2x_2^6+A_3x_3^6)\equiv1\bmod6$. From the equation of $X$, we have
	$$v_\pp(w-\sqrt{A_1} x_1^3) + v_\pp(w+\sqrt{A_1} x_1^3) \equiv 1 \bmod 6.$$
	For the local invariant to be $0$,
	by \eqref{eqn:nth_power} we must have
	$$v_\pp(w-\sqrt{A_1} x_1^3) - v_\pp(w+\sqrt{A_1} x_1^3) \equiv 0 \bmod 3.$$
	These congruences imply that
	$$w-\sqrt{A_1} x_1^3 \equiv w+\sqrt{A_1} x_1^3 \equiv 0 \bmod \pp^2.$$
	Rearranging and noting that $\Q_p = K_\pp$ gives  $x_1 \equiv w \equiv 0 \bmod p^2$. Then from the 
	equation of $X$, we deduce $x_2 \equiv x_3 \equiv 0 \bmod p$, but this
	contradicts primitivity.
\end{proof}

\begin{proposition}\label{prop:surjective_2a}
	Let $p > 13$ be a prime such that
	\begin{enumerate}
		\item $p \mid A_1$ and $p \nmid A_2A_3$;
		\item $A_2/A_3$ is not a cube in $\Q_p$.
	\end{enumerate}
	Then $\inv_p \calA(X(\Q_p)) = \Z/3\Z$.
\end{proposition}
\begin{proof}
Hypotheses (1) and (2) imply that $p \equiv 1 \bmod 3$.
The fact that the invariant map takes on two non-trivial values follows from Lemma \ref{lem:almostsurj}. Hence it remains to show that there exist a point $x\in X(\Q_p)$ such that $\inv_p \calA(x)=0$. It suffices to find a smooth $\F_p$-point on the curve
\begin{equation} \label{eqn:curve_2a}
    w^2=A_2x_2^6+A_3 x_3^6,
\end{equation}
since by Hensel's lemma this will lift to a $\Q_p$-point on $X$ with $x_1= 0$, so that
$$\left(\frac{A_2}{A_3},\,\frac{w-\sqrt{A_1}x_1^3}{w+\sqrt{A_1}x_1^3}\right)_{\omega,\pp}=\left(\frac{A_2}{A_3},\,1\right)_{\omega,\pp}=0.$$
The curve \eqref{eqn:curve_2a} is  smooth curve of genus $2$. Hence by the Hasse--Weil bounds, there is a smooth $\F_p$-point provided $p + 1 > 4\sqrt{p}$.
\end{proof}

Our last result on  local invariants is the following, which  will give us the power saving in the upper bound. The proof requires a more detailed analysis. In the statement, as in other similar statements in the paper, we are claiming that there exists a finite set of primes $S$ such that the proposition holds for all $A_i \in \Z$ as in Theorem \ref{thm:2}. The finite set of primes $S$ can be made explicit, as it just arises via an application of the Hasse--Weil estimates. But many finite sets of primes $S$ will appear in our paper, and to avoid delicate bookkeeping we will not make these explicit (we also do not need to explicitly know $S$ for our counting results).

\begin{proposition}\label{prop:surjective_2}
	There exists a finite set of primes $S$ with the following property.
	Let $p \notin S$ be a prime with 
	\begin{enumerate}
		\item $p \nmid A_1\gcd(A_2,A_3)$;
		\item $v_p(A_2/A_3) \not \equiv 0 \bmod 3$.
	\end{enumerate}
	Then $\inv_p \calA(X(\Q_p)) = \Z/3\Z$.
\end{proposition}

\begin{proof}
We may assume $3 \in S$.
Let $\varepsilon \in \Z/3\Z$ and let $\pp$ be a prime of $\OO_K$ above $p$. Let $n \in \OO_\pp^\times$ be such that $(n,p)_{\omega,\pp} = \varepsilon$ and $\Norm_{K_\pp/\Q_p}(n) = 1 $ if $p\equiv 2 \bmod 3$ (this exists by Lemma \ref{surjlem}).

To prove the result, it suffices to find $w,x_1,x_2,x_3 \in \Q_p$ and $u \in K_\pp^\times$ such that 
\begin{equation} \label{eqn:n_cube}
\frac{w-\sqrt{A_1}x_1^3}{w+\sqrt{A_1}x_1^3}=nu^3,\ \ w^2=A_1 x_1^6+A_2x_2^6 + A_3 x_3^6.
\end{equation}
Indeed, then by Lemma \ref{lem:cubic_invariant} we would have
\begin{align*}
    \inv_p \calA(x) = \beta \left(\frac{A_2}{A_3},\,\frac{w-\sqrt{A_1}x_1^3}{w+\sqrt{A_1}x_1^3}\right)_{\omega,\pp}
    &=\beta \left(p^i,\,n\right)_{\omega,\pp} = \begin{cases}
      -2i\varepsilon, & \text{if}\ p\equiv1\bmod3, \\
      -i\varepsilon, & \text{if}\ p\equiv2\bmod3,
    \end{cases}
\end{align*}
where $i = v_p(A_2/A_3)$. As $i \not \equiv 0 \bmod 3$ by assumption (2), we will obtain all possible values on changing $n$, as required.

Note that for the equation \eqref{eqn:n_cube}, the variables $w,x_i \in \Q_p$, but $n,u, \sqrt{A_1} \in K_\pp$. One could probably view this equation as a variety over $\Q$ defined by a suitable use of the Weil restriction. We shall take different approach which leads to simpler equations, by instead performing a Galois descent to show that this variety is the base change of a variety over $\Q$. This is achieved via the following lemma.

\begin{lem}\label{lem:res}
	Let $k$ be a field of characteristic not equal to $3$ 
	and $\alpha \in k(\mu_3)^{\times}\setminus\{\pm1\}$. If $\mu_3  \not \subset k$
	assume that $\alpha$ has norm $1$.
	Then 
	$$k(\mu_3)[u]/(u^3 - \alpha) \cong 
	k[t]/(t^3 - 3t - \alpha - \alpha^{-1}) \otimes_k k(\mu_3)$$
	as $k(\mu_3)$-algebras. 
\end{lem}
\begin{proof}
	Note that if $\mu_3 \not \subset k$ then $\alpha + \alpha^{-1} \in k$,
	as it is the trace of $\alpha$, thus $t^3 - 3t - \alpha - \alpha^{-1} \in k[t]$.
	For the proof, one easily checks that the $k(\mu_3)$-algebra
	morphisms
	\begin{equation} \label{eqn:t_u}
		u \mapsto \frac{-\alpha}{\alpha^2 - 1}\left(t^2 - \alpha t - 2\right),
		\quad t \mapsto u + u^{-1}
	\end{equation}
	are well-defined and inverse to each other, hence isomorphisms.
\end{proof}

We apply this result with $\alpha = n^{-1}\frac{w-\sqrt{A_1}x_1^3}{w+\sqrt{A_1}x_1^3}$ and $k = \Q_p$ (recall that $n$ has norm $1$ by construction for $p \equiv 2 \bmod 3$). Away from the locus $\alpha = \pm 1$, we find that the variety \eqref{eqn:n_cube} is the base-change of the variety 
\begin{equation} \label{eqn:descent}
t^3 - 3t = \frac{a(w^2+ A_1 x_1^6)+A_1bwx_1^3}{w^2 - A_1 x_1^6},
\,
w^2=A_1 x_1^6+A_2x_2^6 + A_3 x_3^6.
\end{equation}
where $a=n+n^{-1},b=2(n-n^{-1})/\sqrt{A_1}$. Note that $a,b\in \Z_p$.
Thus we have successfully performed a Galois descent, hence it suffices to find a $\Q_p$-point on the variety \eqref{eqn:descent} satisfying $\alpha \neq \pm 1$ and $t^2 - \alpha t - 2 \neq 0$ (cf.~\eqref{eqn:t_u}).
To do so, without loss of generality we have $p \mid A_2$. We consider the patch $x_3 = 1$, where modulo $p$ the equations become
$$t^3 - 3t - \frac{a(w^2 +A_1 x_1^6)+A_1bwx_1^3}{A_3} \equiv 0 \bmod p, \quad w^2\equiv A_1 x_1^6+A_3 \bmod p.$$
This is a cone over a curve. We claim that this curve is  geometrically integral of genus $4$. To prove this, we may pass to the algebraic closure. In particular, in the original equation \eqref{eqn:n_cube} we may assume that $A_1=A_3=n=1$ to obtain
$$t^3 - 3t - 2(w^2 +x_1^6) = 0, \quad w^2=x_1^6+1.$$
\texttt{Magma} now verifies that this is geometrically integral of genus $4$. Therefore by the Hasse--Weil estimates, providing $p$ is sufficiently large the curve has a smooth $\F_p$-point in the given affine patch. Hence we can lift this by Hensel's Lemma to find the required $\Q_p$-point on 
\eqref{eqn:descent}.
\end{proof}

\subsection{Consequences for the Brauer--Manin obstruction}

We finish by clarifying the conclusion of our investigations to the Brauer--Manin obstruction.

\begin{corollary} \label{cor:no_BM_3}
	Assume that there is a prime $p$ which satisfies Proposition \ref{prop:surjective_2a}
	or \ref{prop:surjective_2}, or there are at least
	two primes $p,q$ which satisfy Lemma \ref{lem:almostsurj}.
	
	Then there is no Brauer--Manin obstruction to the Hasse principle given by $\calA_1$.
\end{corollary}
\begin{proof}
	This is a standard argument, which we recall for completeness. 
	If $X(\Adele_\Q) = \emptyset$ then there is nothing to prove.
	So let $(x_v) \in X(\Adele_\Q)$ and $\varepsilon = \sum_{v \neq p} \inv_v \mathcal{A}_1(x_v) \in \Z/3\Z$.
	
	Assume that we are in the first case.
	Then there exists $y_p \in X(\Q_p)$ such that
	$\inv_p \mathcal{A}_1(y_p) = -\varepsilon$. Therefore, the adelic
	point given by replacing $x_p$ by $y_p$ lies in 
	$X(\Adele_\Q)^{\calA_1}$, thus $X(\Adele_\Q)^{\calA_1} \neq \emptyset$ 
	and hence there is no Brauer--Manin obstruction.
	
	In the last  case, we have $2$ primes $p,q$ where
	the local invariants take both values $1/3$ and $2/3$.
	A moment's thought reveals that no matter the sum of the local
	invariants away from $p$ and $q$,
	we can always choose local points at $p$ and $q$
	such that the sum of all local invariants is again trivial, 
	as required.
\end{proof}

\subsection{Upper bound}

We now prove the upper bound in Theorem \ref{thm:2} using the results from \S \ref{sec:counting}.

\begin{proposition} \label{prop:Upper_bound_2}
	There are at most $O(T^{3/2} \log \log T/(\log T)^{2/3})$
	non-zero integers
	$A_1,A_2,A_3$ such that $\max\{|A_i|\} \leq T$,
	$-3A_1$ is a square,  and there is a Brauer--Manin
	obstruction to the Hasse principle given by $\calA$.
\end{proposition}
\begin{proof}
	If $A_2/A_3$ is a cube, then $\calA$ is trivial
	so does not give a Brauer--Manin obstruction. Hence we assume
	that $A_2/A_3 \notin \Q^{\times3}$.
	We next factorise the $A_i$ to unravel
	the conditions given by Corollary~\ref{cor:no_BM_3}.
	Write $A_1 = -3a_1^2$.
	For $i \in \{2,3\}$ we write
	$$A_i = gg_iu_iz_i$$
	where
	$u_i = \prod_{p \mid \gcd(A_i,A_1)} p^{v_p(A_i)}$,
	$g = \gcd(A_2/u_2,A_3/u_3),$
	$g_i = \prod_{p \mid g}p^{v_p(A_i/u_ig)}$, $z_i = A_i/gg_iu_i$.
	Here $u_2$ consists of primes which appear
	in both $A_1$ and $A_2$, whilst $g_2$ consists of primes 
	which appear in both $A_2$ and $A_3$ (similarly with $u_3$ and $g_3$).
	Moreover by construction $\rad(u_2u_3) \mid A_1$ and $\rad(g_2g_3) \mid g$.
	We claim that there exists a finite set of primes $S$ such that 
	if there is a Brauer--Manin obstruction, then 
	$z_i \bmod \Q^{\times3} \in \langle p \in S \rangle \subset \Q^\times/\Q^{\times3}$,
	i.e.~there are only finitely many possibilities for $z_i$
	modulo cubes. Indeed, let $p \mid z_2$ (say) with $v_p(z_2) \not \equiv 0 \bmod 3$
	and note that $p \nmid A_1A_3$ by construction. Providing $p$ is sufficiently large,
	it follows from Proposition \ref{prop:surjective_2}
	that there is no Brauer--Manin obstruction, since 
	$v_p(A_2/A_3) = v_p(gg_2u_2z_2) = v_p(z_2)	\not \equiv 0 \bmod 3$. Thus $p \in S$, for some finite set of primes $S$.
	
	Enlarging $S$ if necessary, it follows from Proposition \ref{prop:surjective_2a}
	that for all primes $p \mid a_1$ with $p \notin S$ and $p \nmid u_2u_3$,
	we must have that $A_2/A_3$ is a cube in $\Q_p$. Similarly,
	it follows from Proposition \ref{prop:BC} that there is at most one prime 
	$q$ with $q\parallel  g, q \notin S$ and $q \nmid g_2g_3$ such that 
	$A_2/A_3$ is not a cube in $\Q_q$. 
	We now combine all these conditions together.
	Recall that there are only finitely many choices for the $z_i$ modulo cubes.
	For simplicity of exposition, we consider only the 
	case $z_i = a_i^3$ is a cube; the other cases being analogous.
	The set being counted is thus smaller than
	$$\# \left\{ a_i, g, g_i, u_i:
	\begin{array}{l} 
	|a_1| \leq \sqrt{T}, |gg_iu_ia_i^3| \leq T, r \notin \Q^{\times3}, \rad(u_2u_3)\mid 3a_1, \rad(g_2g_3) \mid g; \\
	p\mid a_1, p \nmid u_2u_3, p \notin S \implies r \in \Q_p^{\times3}; \\
	\mbox{there is at most one prime } q \mbox{ with } q \parallel g,
	q \notin S, q \nmid g_2g_3,
	r \not \in \Q_q^{\times3}	
	\end{array} \right\}$$
	where we use the shorthand $r = g_2u_2/g_3u_3 (= A_2/A_3 \text{ up to cubes})$.
	We will show that this is bounded by the quantity
	in the statement of the proposition.
	Let $S_r = S \cup \{p \mid g_2g_3u_2u_3\}$.
	Consider the indicator functions
	\begin{align*}
	&\varpi_{r}(n) = 1 \iff \mbox{if } p \mid n, p \notin S_r 
	\mbox{ then } r \in \Q_p^{\times3}, \\
	& \varpi'_{r}(n) = 1 \iff
	 \mbox{There is at most one prime } q \parallel  n \mbox{ with } q \notin S_r
	 \mbox{ such that } 
	r \not \in \Q_q^{\times3}.
	\end{align*}
	Summing over $a_2,a_3$, and recalling the notation
	$r= g_2u_2/g_3u_3$,	the quantity in question is bounded above by 	
	\begin{align}
	&\sum_{\substack{u_i,g_i\leq T \\ r \notin \Q^{\times3}}}
 \sum_{\substack{|a_1|\leq \sqrt{T} \\ \rad(u_2u_3) \mid 3a_1}} \varpi_{r}(a_1)
 \sum_{\substack{g\leq T \\ \rad(g_2g_3) \mid g }} \varpi'_{r}(g)
	\#\{a_2,a_3: |	gg_iu_ia_i^3| \leq T\} \nonumber  \\
	&\ll T^{2/3}\sum_{\substack{u_i,g_i\leq  T \\ r \notin \Q^{\times3}}}
	\frac{1}{(g_2g_3u_2u_3)^{1/3}}
 \sum_{\substack{|a_1|\leq \sqrt{T} \\ \rad(u_2u_3) \mid 3a_1}}\varpi_{r}(a_1)
 \sum_{\substack{g\leq T \\ \rad(g_2g_3) \mid g}}
 	\frac{\varpi'_{r}(g)}{g^{2/3}}. \label{eqn:eek}
	\end{align}
	We first show that the $u_i,g_i$ can be taken to be small. Namely,
	let $A > 0$ be sufficiently large. Then 
	the contribution from $u_2 > (\log T)^A$ is 
	\begin{align*}
	 &\ll T^{2/3}\sum_{\substack{u_i,g_i\leq T \\ u_2 > (\log T)^A}}
 \sum_{\substack{|a_1|\leq \sqrt{T} \\ \rad(u_2u_3) \mid 3a_1}} 
 \sum_{\substack{g\leq T \\ \rad(g_2g_3) \mid g }}  \frac{1}{(g^2g_2g_3u_2u_3)^{1/3}} \\
 & \ll T^{3/2} \sum_{\substack{u_i,g_i\leq T \\ u_2 > (\log T)^A}}
 \frac{1}{(g_2g_3u_2u_3)^{1/3}\rad(g_2g_3)\rad(u_2u_3)}.
	\end{align*}	
	The summation over $g_2,g_3$ is convergent; indeed this is
	$$ \sum_{g_2,g_3 \leq T}
 \frac{1}{(g_2g_3)^{1/3}\rad(g_2g_3)}
 \ll  \sum_{n \leq T^2}
 \frac{\tau(n)}{n^{1/3}\rad(n)}$$
 which converges by Lemma \ref{lem:rad} on applying the usual divisor bound $\tau(n) \ll_\varepsilon n^{\varepsilon}$. Similarly, the sums over $u_2,u_3$ are 
 	$$\ll  \sum_{(\log T)^A < n \leq T^2}
 \frac{\tau(n)}{n^{1/3}\rad(n)} \ll \frac{1}{(\log T)^{A/6}} \sum_{ n \leq T^2}
 \frac{\tau(n)}{n^{1/6}\rad(n)}$$
 and this  sum converges by Lemma \ref{lem:rad}.
	Altogether, we find that the contribution from 
	$u_2 > (\log T)^A$ is $O(T^{3/2}/(\log T)^{A/6})$. 
	Providing $A$ is  large, this is
	negligible for the statement. A similar argument applies to $u_3, g_2,g_3$, so we may assume
	that $u_i,g_i \leq (\log T)^A$. We are thus
	in a position to apply Proposition \ref{prop:cube}
	to the sums over
	$a_1$ and $g$. We therefore find that 	
	our quantity \eqref{eqn:eek} is bounded above by
	\begin{align*}
	&\ll T^{2/3}\sum_{\substack{u_i,g_i\leq (\log T)^A}}
	\frac{1}{(g_2g_3u_2u_3)^{1/3}}\cdot \frac{\sqrt{T}}{\rad(u_2u_3)(\log T)^{1/3}}
	\cdot \frac{T^{1/3} (\log \log T)}{\rad(g_2g_3)(\log T)^{1/3}} \\
	&\ll \frac{T^{3/2} (\log \log T)}{(\log T)^{2/3}}\sum_{\substack{u_i,g_i\leq (\log T)^A}}
	\frac{1}{(g_2g_3u_2u_3)^{1/3}\rad(g_2g_3)\rad(u_2u_3)}
	\end{align*}
	and this final sum is  convergent by
	Lemma \ref{lem:rad}.
	This completes the proof.
\end{proof}

\subsection{Lower bound}
We prove the lower bound in Theorem \ref{thm:2} by constructing an explicit family of counter-examples.

\begin{proposition} \label{prop:lower_bound_family_2}
There exists an integer $M$ with $3 \nmid M$ such that the following holds.
Consider the family of surfaces
$$w^2=c^2ax_1^6+bqx_2^6-c^4bqx_3^6.$$
Let $a<0$ and $b,c,q>0$ be integers satisfying the following conditions.
\begin{itemize}
\item $-3a$ is a square.
\item $c\equiv2\bmod 3$ is a prime such that $c$ is a cube in $\Q_3$.
\item $q \equiv 1 \bmod 3$ is a prime such that $c$ is not a cube in $\Q_q$.
\item $p\mid ab\implies c$ is a cube in $\Q_p$
\item $b\equiv q \equiv 1 \bmod c$.
\item $a \equiv b \equiv q \equiv 1 \bmod M$.
\end{itemize}
Then for such surfaces $\calA$ gives a Brauer--Manin obstruction to the Hasse principle.

\end{proposition}
\begin{proof}
Note that our conditions imply that $\gcd(ab,qc) = 1$ and $q \neq c$. The integer $M$ is only used to guarantee local solubility at small primes. Our proof of Theorem \ref{thm:2} will show that such choices of the coefficients actually exist. We use throughout the formula given in Lemma \ref{lem:cubic_invariant}.

\medskip
\noindent {\it Showing local solubility}.
We first show that these surfaces are everywhere locally soluble. There is an $\R$-point given by $(\sqrt{bq}:0:1:0)$. So let $p$ be a prime.

{\bf Case $p\mid ab$.} Since $c$ is a cube in $\Q_p$, we can take the point $(0:0:c^{2/3}:1)$.
(Note that this covers the case $p = 3$ since $3 \mid a$.)

{\bf Case $p=q$.} Since $q\equiv1\bmod 3$, $a$ is a square in $\Q_p$, so take  $(c\sqrt{a}:1:0:0)$.

{\bf Case $p=c$.} Since $bq$ is a square modulo $c$, take the point $(\sqrt{bq}:0:1:0)$.

{\bf Case $p \mid M$.} We choose $M$ to be a product of powers of sufficiently many small
primes including $2$ but excluding $3$. The surface is congruent modulo $p^{v_p(M)}$ to 
$$w^2 = c^2x_1^6 + x_2^6 - c^4x_3^6.$$
This has the smooth $\Q_p$-point $(1:0:1:0)$. Providing the $p$-adic valuation of $M$ is chosen sufficiently large, our original surface thus also has a $\Q_p$-point.

{\bf Case $p\nmid abcqM$.} Here the surface is smooth modulo $p$. Providing $M$ is divisible by sufficiently many small primes, the Lang--Weil estimates show that there is a smooth $\F_p$-point, which thus lifts to a $\Z_p$-point by Hensel's lemma.

\medskip
\noindent {\it Showing existence of a Brauer--Manin obstruction}.
Let $p$ be a prime. We consider the invariant maps $X(\Q_p)\to \Z/3\Z$ given by pairing with $\calA$. Note that $A_2/A_3 = -c^{-4}$ depends only on $c$. We will show the local invariants are trivial for $p\neq q$, but for $p=q$ the local invariants take exactly the values $1/3$ and $2/3$. From this, it is clear that there is a Brauer--Manin obstruction.

{\bf Case $p\nmid abcq$.} The invariant map is trivial by Lemma \ref{lem:trivial_2}.

{\bf Case $p\mid ab$.} The invariant map is trivial since $c$ is a cube in $\Q_p$. (This includes the case $p = 3$ since $3 \mid a$.)

{\bf Case $p=q$.} The invariant map takes on exactly the values $1/3$ and $2/3$ by Proposition  \ref{prop:BC}.

{\bf Case $p=c$.} This is the most difficult case.
Let $w,x_1,x_2,x_3$ be a solution over $\Z_p$. Moreover, as the invariant map is locally constant, we may assume that $w^2\neq c^2ax_1^6$. Let
$$n:=\frac{w-c\sqrt{a}x_1^3}{w+c\sqrt{a}x_1^3},$$
so that $n\in K_\pp=\Q_p(\omega)$ where $\pp=(p)$ is the prime ideal lying above $p$. Rearranging gives
$$w=\frac{1+n}{1-n}c\sqrt{a}x_1^3.$$
Plugging back into the original equation gives
$$c^2a\left(\frac{(1+n)^2}{(1-n)^2}-1\right)x_1^6=bqx_2^6-c^4bqx_3^6.$$
After simplifying, we get
$$n=(1-n)^2\frac{bqx_2^6-c^4bqx_3^6}{4c^2ax_1^6}.$$
Switching $x_1$ with $-x_1$, we can assume $v_\pp(n)\geq0$. First suppose $v_\pp(n)~=~0$. Then $v_\pp(w-c\sqrt{a}x_1^3)=v_\pp(w+c\sqrt{a}x_1^3)=:i$. Then we have one of the following.
\begin{enumerate}
\item $v_\pp(w)=i,$ \quad $v_\pp(cx_1^3)=i$.
\item $v_\pp(w)=i,$ \quad $v_\pp(cx_1^3)>i$.
\item $v_\pp(w)>i,$ \quad $v_\pp(cx_1^3)=i$.
\end{enumerate}
In cases (2) and (3), we have $n\equiv 1,-1\bmod p$ respectively, so the invariant is trivial as $n \in \OO_\pp^{\times3}$. Consider now case (1). In particular we must have $i\equiv1\bmod 3$. Hence $v_\pp(w^2-c^2ax_1^6)=2i\equiv 2\bmod 3$. However, $v_\pp(bqx_2^6-c^4bqx_3^6)\equiv 0,1\bmod 3$ which is a contradiction.

Now assume $v_\pp(n)>0$. Write $n=p^in'$ where $v_\pp(n')=0$. We have
$$n'\equiv p^{-i}\frac{bqx_2^6-c^4bqx_3^6}{4c^2ax_1^6}\bmod p.$$
Since the right hand side only involves terms from $\Q_p$, we have $n'\equiv m\bmod p$ for some $m\in \Z_p^\times$. Thus by \eqref{eqn:2mod3} and Hensel's lemma we have $n' \in \OO_\pp^{\times3}$. Hence
\[\left( n, p\right)_{\omega,\pp}=(p^i,p)_{\omega,\pp} + \left(n', p\right)_{\omega,\pp}=0, \]
where we use that $(p^i,p)_\omega = 0$ as a cyclic algebra (as $p$ is a norm from $K(\sqrt[3]{p})$).
\end{proof}

\begin{proposition}\label{prop:lower_bound_2}
	There are at least $\gg T^{3/2}\log\log T/(\log T)^{2/3}$ 
	non-zero integers $A_i$ such that $\max\{|A_i|\} \leq T$,
	$-3A_1$ is a square, 
	and there is a Brauer--Manin
	obstruction to the Hasse principle given by $\calA$.
\end{proposition}
\begin{proof}
It thus suffices to count the surfaces in Proposition \ref{prop:lower_bound_family_2}. We do this for $c = 17$, which is easily checked to be a cube in $\Q_3$ as is necessary for the proposition.
As we only seek a lower bound, we may assume that $q \leq T^{\varepsilon}$ for some small $\varepsilon > 0$. We let $\varpi_c(n)$ be the indicator function for those integers $n$ with $p \mid n \implies c \in \Q_p^{\times3}$. Noting that $\varpi_c$ is frobenian with $m(\varpi_c) = 2/3$, we may apply Proposition \ref{prop:lower_bound_frob} to see that the quantity in question is bounded below by
\begin{align*}
  \gg \sum_{\substack{q \leq T^\varepsilon \\ c \notin \Q_q^{\times3} \\ q \equiv 1 \bmod cM}}
\sum_{\substack{a \leq \sqrt{T} \\ a \equiv 1 \bmod M}} \varpi_c(a)
 \sum_{\substack{ bq \leq T \\ b \equiv 1 \bmod cM}} \varpi_c(b) 
&\gg \frac{T^{1/2}}{(\log T)^{1/3}} 
\sum_{\substack{q \leq T^\varepsilon \\ c \notin \Q_q^{\times3} \\ q \equiv 1 \bmod cM}}\frac{T}{q\log(T/q)^{1/3}}.
\end{align*}
The conditions $q \equiv 1 \bmod cM$ and $c \notin \Q_q^{\times 3}$ hold if $q$ is completely split in $\Q(\mu_{3cM})$ but not completely split in $\Q(c^{1/3})$. These are conditions are independent as these number fields are linearly disjoint. Thus we may sum over $q$ using the Chebotarev density theorem and partial summation to obtain
$$\sum_{\substack{q \leq T^\varepsilon \\ c \notin \Q_q^{\times3} \\ q \equiv 1 \bmod cM}}\frac{T}{q\log(T/q)^{1/3}}. \ll \frac{T^{3/2}\log\log T}{(\log T)^{2/3}}$$
which combined with the above bounds gives the result.
\end{proof}

Theorem \ref{thm:2} now follows from Propositions \ref{prop:Upper_bound_2} and \ref{prop:lower_bound_2}.  \qed

\begin{remark} \label{rem:Merten}
	The proof of Theorem \ref{thm:2} 
	makes clear how the special factor $\log \log T$
	arises. This is due to the rogue  prime $q$;
	this has the property that the local invariants take exactly
	the values $1/3$ and $2/3$ at $q$. We are only allowed one
	such prime, as otherwise there is no Brauer--Manin
	obstruction (cf.~Corollary \ref{cor:no_BM_3}).
	The prime $q$ divides one of the coefficients;  when we sum
	over this coefficient a term of the shape 
	$$\sum_{\substack{q \text{ prime} \\ q \leq T}} \frac{1}{q}$$
	appears. The sum over this has order $\log \log T$, by
	Mertens' theorem.
\end{remark}

\section{Brauer--Manin obstruction from a quaternion algebra}\label{sec:index3}

\subsection{Set-up}
Let $A_1,A_2,A_3$ be nonzero integers such that $A_2/A_3$ is a cube. We now play a similar game to the previous section, but using the quaternion algebra $\calB:=\calB_1$ from \S\ref{sec:Brauer_elements}.

\begin{theorem} \label{thm:3+6}
	There are $\asymp T^{3/2}/(\log T)^{3/8}$ non-zero integers
	$A_1,A_2,A_3$ such that $\max\{|A_i|\} \leq~T$, 
	$A_2/A_3$ is a cube, and there is a Brauer--Manin
	obstruction to the Hasse principle given by $\calB$.
\end{theorem}

The proof of Theorem \ref{thm:3+6} has a similar structure to that of Theorem \ref{thm:2}. There are some simplifications however as we are working with a quaternion algebra, rather than a corestriction of a cyclic algebra. So the local invariant map is either constant or surjective, and there are no rogue primes appearing in the proof.

\subsection{Local invariants}
We let $c=\sqrt[3]{A_3/A_2}$, so that the equation becomes
$$w^2 = A_1x_1^6 + A_2(x_2^6 + c^3x_3^6).$$
For a place $v$ of $\Q$, the local invariant is given by the usual quadratic Hilbert symbol
$$\inv_v \calB(x) =\left( A_1,\, \frac{x_2^2 + cx_3^2}{x_1^2} \right)_v,$$
where by our convention this symbol takes values in $0,1/2$. As before, this formula is only well-defined providing $(x_2^2 + cx_3^2)x_1^2  \neq 0$, which we implicitly assume throughout.

\begin{lemma}\label{lem:surjective2tor}
	There exists a finite set of primes $S$ with the following 	
	property. Let $p \notin S$ be a prime with 
	\begin{enumerate}
		\item $p \mid A_1$ and $p \nmid A_2A_3$;
		\item $v_p(A_1) \equiv 1 \bmod 2$.
	\end{enumerate}
	Then $\inv_p \calB(X(\Q_p)) = \Z/2\Z$.
\end{lemma}
\begin{proof}
	Let $\varepsilon\in\Z/2\Z$ and choose $n \in \Z_p^\times$ so that $(n,p)_p=\varepsilon$. It suffices to find a $\Q_p$-point
	on the surface $X$ with
	$$x_2^2 + cx_3^2 = nx_1^2,\quad x_1\neq0,$$
	since, as $v_p(A_1)$ is odd, standard formulae
	for Hilbert symbols would show that
	$$\left( A_1,\, \frac{x_2^2 + cx_3^2}{x_1^2} \right)_{p}
	= (p,n)_p=(n,p)_p=\varepsilon.$$
	To do this, we find a smooth $\F_p$-solution to the equations
	$$x_2^2 + cx_3^2 = nx_1^2,\quad w^2 =A_2x_2^6 + A_2c^3x_3^6,$$
	which could then be lifted to a desired $\Q_p$-point on $X$ via Hensel's
	lemma.
	
	The above equations describe a curve in weighted 
	projective space. We look at the patch given by
	$x_3 = 1$, and pass to the algebraic closure,
	where the curve is given by
	$$
	x_2^2 + 1 = x_1^2, \quad w^2 = x_2^6 + 1.
	$$
	\texttt{Magma} verifies that this is geometrically integral of genus $3$. Therefore the Hasse--Weil estimates complete the proof.
\end{proof}

\begin{lemma}\label{lem:surjective2tor3}
	Let $p > 2$ be a prime with
	\begin{enumerate}
		\item $v_p(A_1)$ even;
		\item $A_1 \notin \Q_p^{\times2}$;
		\item $v_p(A_2)$ odd;
		\item $v_p(A_2) = v_p(A_3)$.
	\end{enumerate}
	If $X(\Q_p) \neq \emptyset$ then
	$$-A_2/A_3 \in \Q_p^{\times2}.$$
	Moreover in this case, there exists a finite set $S$ of primes such that
	$$\inv_p \calB(X(\Q_p))=
	\begin{cases}
	\mathbb{Z}/2\Z, & \text{if}\ p\equiv1\bmod3,\,\,
	p \notin S,
	\\
      1/2, & \text{if}\ p\equiv2\bmod3.
      \end{cases}$$
\end{lemma}

\begin{proof}
    We can multiply the equation of $X$ by $p^{2v_p(A_1)}$ so that $v_p(A_1)\equiv0\bmod6$. Then any factor of $p$ in $A_1$ can be absorbed into $x_1$, and similarly for $w$. Conditions (1)--(4) and the class of $\calB$ in $\Br X$ are unchanged while now we have $v_p(A_1)=0$. Hence for the remainder of the proof we assume $v_p(A_1)=0$.
    
	Assume $X$ is locally soluble at $p$ and let $w,x_i\in \Q_p$ be a solution to the equation
	$$w^2=A_1 x_1^6+A_2x_2^6+A_2c^3x_3^6.$$
	We can assume $w,x_i\neq0$ by smoothness. Since $A_1 \notin \Q_p^{\times2}$, we have that $v_p(w^2- A_1x_1^6)$ is always even, so $v_p(A_2x_2^6+A_2c^3x_3^6)$ is even as well. Since $v_p(A_2)$ is odd, this means $v_p(x_2^6+c^3x_3^6)$ is also odd. 
	In particular, this implies  $-c$ is a square in $\Q_p$. Thus $-A_2/A_3 = (-c)^{-3}$ is a square in $\Q_p$.

	Next, suppose $p\equiv 2\bmod 3$. Observe that $v_p(x_2^4-cx_2^2x_3^2+c^2x_3^4)$ is even since the polynomial $t^2-t+1$ is irreducible in $\F_p[t]$. As the valuation
	$v_p(x_2^6+c^3x_3^6) = v_p((x_2^4-cx_2^2x_3^2+c^2x_3^4)(x_2^2+cx_3^2))$ is odd,
	this implies that $v_p(x_2^2+cx_3^2)$ is odd. Hence
	(1) and (2) give
	$$\inv_p \calB(x_p)=(A_1, x_2^2+cx_3^2)_{p}=(A_1, p)_{p}=1/2.$$
	
	Now suppose $p\equiv 1\bmod 3$. Let $\beta$ be a third
	root of unity in $\Q_p$ and $i \in \N$ with $i\equiv -v_p(A_2)\bmod 6$. Since $-c$ is a square in $\Z_p$,
	Hensel's lemma shows that there exists $\alpha\in\Q_p$ 
	such that $\alpha^2=-\beta c+p^i$. 
	Set $x_2=\alpha x_3$, then
	\begin{equation}\label{eqn:surjective2tor3}
	    w^2=A_1 x_1^6+A_2(x_2^6+c^3x_3^6)
	\end{equation}
	becomes
	$$w^2=A_1 x_1^6+A_2(3\beta^2c^2p^i-3\beta cp^{2i}+p^{3i})x_3^6.$$
	Since $v_p(A_2(3\beta^2c^2p^i-3\beta cp^{2i}+p^{3i}))\equiv0\bmod6$ by the choice of $i$ and (4), we can make a change of variables on $x_3$ to get
	$$w^2=A_1 x_1^6+A_2'x_3'^6$$
	where $v_p(A_2')=0$. As $v_p(A_1) = 0$, the reduction modulo $p$ is a smooth curve of genus 2. Hence for $p$ large enough, this has $\F_p$-solutions which can be lifted to a solution $w,x_i\in\Q_p$ to \eqref{eqn:surjective2tor3}.
	If $x_p\in X(\Q_p)$ is such a solution, then
	$$\inv_p \calB(x_p)=(A_1, x_2^2+cx_3^2)_{p}=(A_1, \alpha^2+c)_{p}
	= \begin{cases}
	    0, & \mbox{if } \beta \neq 1,\\
	    1/2, & \mbox{if }\beta = 1,
	\end{cases}$$
	by (1)--(4).
\end{proof}

\begin{lemma}\label{lem:trivial2tor1}
	Let $p> 2$ be a prime with
	\begin{enumerate}
		\item $v_p(A_1)$ even; 
		\item $A_1 \in \Q_p^{\times2}$ or $v_p(A_2)$ even;
		\item $v_p(A_2) =v_p(A_3)$;
		\item $X(\Q_p) \neq \emptyset$.
	\end{enumerate}
	Then $\inv_p \calB(X(\Q_p)) = 0$.
\end{lemma}

\begin{proof}
As in Lemma \ref{lem:surjective2tor3}, we can assume that $v_p(A_1)=0$. If $A_1 \in \Q_p^{\times2}$  then the local invariant is clearly trivial.
So assume that $v_p(A_2)$ is even. If $v_p(x_2^2+cx_3^2)$ is even then the invariant is clearly trivial. So assume that $v_p(x_2^2+cx_3^2)$ is odd, which implies that $-c \in \Q_p^{\times2}$. Since $v_p(x_2^2+cx_3^2)$ is odd it follows that $v_p(x_2^4-cx_2^2x_3^2+c^2x_3^4)=v_p((x_2^2+\omega cx_3^2)(x_2^2+\omega^2cx_3^2))$ is even, hence $v_p(x_2^6+c^3x_3^6)$ is odd as well. 

Hence $v_p(A_2x_2^6+A_2c^3x_3^6)$ is also odd. Thus, $v_p(w^2-A_1x_1^6)$ is odd whence $A_1 \in \Q_p^{\times2}$. The result follows.
\end{proof}

\subsection{Upper bounds}
We now begin the counting part of the proof of Theorem \ref{thm:3+6}.

\begin{proposition}\label{prop:ub_index3}
	There are $O(T^{3/2}/(\log T)^{3/8})$ non-zero integers
	$A_1,A_2,A_3$ such that $\max\{|A_i|\} \leq~T$, 
	$A_2/A_3$ is a cube, and there is a Brauer--Manin
	obstruction to the Hasse principle given by $\calB$.
\end{proposition}

\begin{proof}
We begin by eliminating some cases in which $\calB$ does not obstruct the Hasse principle.  If $A_1$ is a square then $\calB$ is trivial hence gives no Brauer--Manin obstruction. If $-A_2/A_3$ is a square, then it would be a sixth power as it is a cube by assumption;
in this case there is the rational point $(0:0:1:\sqrt[6]{-A_2/A_3})$.
If $-3A_1$ is a square then the algebra is given by
$$\calB
= \left(-3,\frac{x_2^2+cx_3^2}{x_1^2}\right)
= \left(-3,\frac{A_2(x_2^4 - cx_2^2x_3^2 + c^2x_3^4)}{x_1^4}\right)
= \left(-3,A_2\right) \in \Br \Q$$
where the last equality comes from the fact
that the quartic polynomial in the numerator
is a norm from $\Q(\omega)$. Thus being a constant
algebra, it does not obstruct the Hasse principle.
Finally, if $3A_1A_2A_3$ is a square, then
\begin{align*}\calB
= \left(3A_2A_3,\frac{x_2^2+cx_3^2}{x_1^2}\right)
&= \left(-3,\frac{x_2^2+cx_3^2}{x_1^2}\right)
+ \left(-A_3/A_2,\frac{x_2^2+cx_3^2}{x_1^2}\right) \\
&= \left(-3,A_2\right) +
3\left(-c,\frac{x_2^2+cx_3^2}{x_1^2}\right)\in \Br \Q,
\end{align*}
as the last algebra is clearly trivial.
In conclusion, we may assume that
\begin{equation} \label{eqn:not_squares}
    A_1, -A_2/A_3, -3A_1, 3A_1A_2A_3 \notin \Q^{\times 2}.
\end{equation}
We now begin the counting. If there is some prime for which the local invariant map is surjective, then $\calB$ does not give a Brauer--Manin obstruction to the Hasse principle.
Thus Lemmas \ref{lem:surjective2tor} and \ref{lem:surjective2tor3}, show that if $X_{\mathbf{A}}$ has a Brauer--Manin obstruction from $\calB$, then $(A_1,A_2,A_3)$ lies in the following set

$$ \left\{ A_i\in \Z :
	\begin{array}{l} 
	|A_i|\leq T, A_2/A_3\in\Q^{\times3}, \eqref{eqn:not_squares}; \\
	p \notin S, p\mid A_1, p\nmid A_2A_3
	\implies v_p(A_1) \equiv 0 \bmod 2;\\
	p \notin S, v_p(A_1)\text{ even}, v_p(A_2)= v_p(A_3)\text{ odd},  A_1 \notin \Q_p^{\times2}\\ \implies
	(\frac{-A_2/A_3}{p})=1, p\equiv2\bmod3.
	\end{array} \right\}$$
for some finite set of primes $S$ with $2,3 \in S$.
	For $i \in \{1,2,3\}$ we write
	$$u_i = \prod_{p \mid \gcd(A_1,A_2A_3)} p^{v_p(A_i)}.$$
	Let $g = \gcd(A_2/u_2,A_3/u_3)$ and  
	$g_i = \prod_{p \mid g}p^{v_p(A_i/u_ig)}$
	for $i =2,3$. Then we can write
	$$A_1 = u_1z_1, \qquad A_i = gg_iu_iz_i, \quad  i\in \{2,3\}.$$
	By construction $\rad(u_1) = \rad(u_2u_3)$
	and $\rad(g_2g_3) \mid g$.
	By the above $z_1$ takes finitely many values modulo squares.
	For simplicity of notation we just consider the case where
	$z_1 = a_1^2$ is a square. Similarly, by assumption
	$A_2/A_3$ is a cube; since $\gcd(z_i,g_ig_ju_iu_jz_j) = 1$ by construction for $\{i,j\} = \{2,3\}$,
	we have $z_i = a_i^3$ for some $a_i$. Combining
	all our conditions together, we find that our quantity is majorised by 
$$\# \left\{a_i,u_i,g,g_i \in \Z :
	\begin{array}{l} 
	|u_1a_1^2|,|gg_iu_ia_i^3| \leq T, i=2,3,  g_2u_2/g_3u_3 \in \Q^{\times3}, \eqref{eqn:not_squares};\\
	\rad(u_1) = \rad(u_2u_3), \rad(g_2g_3) \mid g; \\
	p\notin S, p\nmid  u_1r_2, p \parallel  g, (\frac{u_1}{p}) = -1
	\implies (\frac{r_2}{p})=1, p\equiv2\bmod3.
	\end{array} \right\}$$
	where we use the shorthand
	$r_2 = -g_2g_3u_2u_3a_2a_3 \equiv -A_2/A_3 \bmod \Q^{\times2}$.
	
	The proof now proceeds in a similar manner to the proof of Proposition \ref{prop:Upper_bound_2}.
	We first show that the variables $a_2,a_3,u_i,g_i$ can
	be taken to be small. We sum over $a_1$
	and $g$, ignoring the conditions on the prime divisors
	of $g$, to obtain
	\begin{align*}
	\ll & \sum_{u_1}
	\sum_{\substack{g_i,u_2,u_3,a_2,a_3 \\ \rad(u_1) = \rad(u_2u_3) \\ g_2u_2/g_3u_3 \in \Q^{\times3}}}
	\frac{T^{1/2}}{\sqrt{u_1}} \cdot \frac{T}{\rad(g_2g_3)\max\{g_2u_2a_2^3,
	g_3u_3a_3^3\}} \\
	\ll& T^{3/2}\sum_{u_1}\frac{1}{\sqrt{u_1}}
	\sum_{\substack{g_i,u_2,u_3,a_2,a_3 \\ \rad(u_1) = \rad(u_2u_3) \\ g_2u_2/g_3u_3 \in \Q^{\times3} \\ 
	g_2u_2a_2^3 \leq g_3u_3a_3^3}}
	\frac{1}{\rad(g_2g_3)g_3u_3a_3^3} \\
    \ll &  T^{3/2}\sum_{u_1}\frac{1}{\sqrt{u_1}}
	\sum_{\substack{g_i,u_2,u_3,a_3\\ \rad(u_1) = \rad(u_2u_3) \\ g_2u_2/g_3u_3 \in \Q^{\times3}}}
	\frac{1}{\rad(g_2g_3)(g_2u_2)^{1/3}(g_3u_3)^{2/3}a_3^{2}},
	\end{align*}
    whence the sums over $a_2$ and $a_3$ are convergent.
	Similarly, the sums over the  $g_i$ are
	convergent by Lemma \ref{lem:rad}. For the sum over the $u_i$, by construction
	$\gcd(u_i,g_i) = 1$, hence the condition 
	$g_2u_2/g_3u_3 \in \Q^{\times3}$ implies that
	$u_2/u_3 \in \Q^{\times3}$. Letting $u = \gcd(u_2,u_3)$,
	we have $u_i = u v_i^3$ for $i =2,3$. Hence
	the sum over the $u_i$ is
	$$\sum_{u_1}\frac{1}{\sqrt{u_1}}
	\sum_{\substack{u_2,u_3\\ \rad(u_1) = \rad(u_2u_3) \\ u_2/u_3 \in \Q^{\times3}}}
	\frac{1}{(u_2u_3^2)^{1/3}} \ll
	\sum_{u_1}\frac{1}{\sqrt{u_1}}
	\sum_{\substack{u,v_2,v_3\\ \rad(u_1) = 
	\rad(uv_2v_3) }}
	\frac{1}{uv_2v_3},$$
	which is convergent by Lemma \ref{lem:rad2}
	(with $n_1 = u_1$ and $n_2 = uv_2v_3$,
	and applying the usual divisor bound $\tau(n) \ll n^{\varepsilon}$).
	Therefore, up to an acceptable error,
	we may assume that
	$a_2,a_3,u_i,g_i \ll (\log T)^A$
	for some large $A > 0$.
    We now sum over $a_1$ first then sum over $g$ 
    using Lemma \ref{lem:square} with 
	$$
	r_1 = u_1, \quad r_2 = -g_2g_3u_2u_3a_2a_3,
	\quad r_3 = -3.
    $$
    Note that modulo squares, in terms of the $A_i$ we have
    \begin{align*}
        &(r_1,r_2,r_3,r_1r_2,r_1r_3,r_2r_3,r_1r_2r_3) \\
        &\equiv
    (A_1,-A_2A_3,-3,-A_1A_2A_3, -3A_1,3A_2A_3, 3A_1A_2A_3) \bmod \Q^{\times2}.
    \end{align*}
    By Lemma \ref{lem:square} and \eqref{eqn:not_squares}, we obtain the bound $O(T^{3/2}/(\log T)^{1/2})$ if $r_1r_2$ is a square,
    and $O(T^{3/2}/(\log T)^{3/8})$ otherwise, which are both satisfactory.
\end{proof}

\subsection{Lower bounds}

\begin{proposition}\label{prop:lower_bound_family3} Consider the family of surfaces

$$w^2=4\cdot 7a^2x_1^6+2bx_2^6+2\cdot 7^3bx_3^6$$
where $a,b\in\N$ satisfy
\begin{itemize}
\item $\gcd(ab,2\cdot3\cdot7) = 1$. 
\item If $p\mid b$ and $(\frac{7}{p})=-1$, then $p\equiv2\bmod3$ and $p\equiv 3\bmod 4$.
\item $b\equiv 1\bmod 7\cdot8$.
\end{itemize}
Then for such surfaces $\calB$ gives a Brauer--Manin obstruction to the Hasse principle.
\end{proposition}

\begin{proof}
{\it Showing local solubility.}  It is clear there are real points. Let $p$ be a prime.

{\bf Case $p\nmid 2\cdot7b$}. If $(\frac{7}{p})=1$, then $(2a\sqrt{7}:1:0:0)$ is a $\Q_{p}$-point. Otherwise $(\frac{7}{p})=-1$ and either $2b$ or $2\cdot7^3b$ is a square modulo $p$, so either $(\sqrt{2b}:0:1:0)$ or $(\sqrt{2\cdot 7^3b}:0:0:1)$ is a $\Q_p$-point.

{\bf Case $p\mid b$}. If $(\frac{7}{p})=1$, then $(2a\sqrt{7}:1:0:0)$ is a $\Q_{p}$-point. Otherwise, by hypothesis  $p\equiv 3\bmod 4$, hence $(\frac{-7}{p})=1$ and so $(0:0:\sqrt{-7}:1)$ is a $\Q_{p}$-point.

{\bf Case $p=7$}. Since $b$ is a square modulo $7$, we have $(\sqrt{2b}:0:1:0)$ is a $\Q_{7}$-point.

{\bf Case $p=2$}. Here $-7 \in \Q_2^{\times2}$, hence  $(0:0:\sqrt{-7}:1)$ is a $\Q_2$-point.

\medskip
{\noindent \it Showing existence of Brauer--Manin obstruction}. We will show that
\begin{enumerate}
\item[(i)] For $p$ with $v_p(b)$ odd and $(\frac{7}{p})=-1$, $\inv_p \calB(x_p)=1/2$ for all $x_p\in X(\Q_p)$.
\item[(ii)] For $p=2$, $\inv_p \calB(x_p)=1/2$ for all $x_p\in X(\Q_p)$.
\item[(iii)] For all other primes $p$, $\inv_p \calB(x_p)=0$ for all $x_p\in X(\Q_p)$.
\end{enumerate}
The condition $b\equiv1 \bmod 7$ implies that there is an even number of primes $p$ such that $v_p(b)$ is odd and $(\frac{p}{7})=-1$. The condition $b\equiv1\bmod 4$ implies that there is an even number of primes $p$ such that $v_p(b)$ is odd and $p\equiv 3\bmod 4$. Together, these imply that there is an even number of primes that satisfy condition (i).
Hence $\calB$ gives an obstruction to the Hasse principle, as  for any $\{x_p\}\in X(\Adele_\Q)$ we have
$$\sum_p\inv_p \calB(x_p)=1/2.$$
Claim (i) follows from Lemma \ref{lem:surjective2tor3}.

For (ii), we copy the proof in \cite[Thm.~4.1]{cn17}. Let $(w:x_1:x_2:x_3)\in X(\Q_2)$ be a point. Assume the coordinates are scaled so that they are all in $\Z_2$ and not all even. Note that $w$ must be even by the equation. Assume now that both $x_2,x_3$ are also even, which implies $(w/2)^2\equiv 7a^2x_1^6\bmod 8$. If $x_1$ is odd, then the right side is congruent to $7\bmod 8$, which is not a square, so $x_1$ must be even. If $x_1$ is even then all coordinates are even, which is a contradiction. Hence, one of $x_2,x_3$ must be odd. Lastly, $2bx_2^6+2\cdot7^3bx_3^6=w^2-4\cdot7a^2x_1^6\equiv 0\bmod 4$ and since $b$ is odd, it follows $x_2,x_3$ must both be odd.

To compute the invariant, we use the fact that $2b(x_2^2+7x_3^2)(x_2^4-7x_2^2x_3^2+7^2x_3^4)=w^2-4\cdot 7a^2x_1^6$. Then the Hilbert symbol becomes
$$\left(A_1, x_2^2+cx_3\right)_2=\left(7,x_2^2+7x_3^2\right)_2=(7,2b)_2+(7,x_2^4-7x_2^2x_3^2+7^2x_3^4)_2$$
since $(7,w^2-4\cdot7a^2x_1^6)_2=0$ as the right hand side is a norm. But since $x_2,x_3$ are odd and $b\equiv1\bmod 8$, we have
$$(7,2b)_2+(7,x_2^4-7x_2^2x_3^2+7^2x_3^4)_2=(7,2)_2+(7,-5)_2=1/2,$$
which proves (ii).

We now prove (iii). First observe that we can assume our $\Q_p$-point satisfies $x_1(x_2^2+7x_3^2)\neq0$ since the invariant map is continuous on the $p$-adic topology. Modulo squares, the Hilbert symbol becomes
$$(A_1,x_2^2+cx_3^2)_p=(7,x_2^2+7x_3^2)_p=(-1,x_2^2+7x_3^2)_p.$$

{\bf Case $p=\infty$.} Note that $x_2^2+7x_3^2\geq0$ always so the invariant map is trivial.

{\bf Case $p\nmid 14b$.} The only way the invariant can be non-trivial is if $(\frac{-1}{p})=-1$ and $v_p(x_2^2+7x_3^2)$ is odd. The last condition implies $(\frac{-7}{p})=1$, so $(\frac{7}{p})=-1$. But then from the equation $v_p(w^2-4\cdot7a^2x_1^6)$ is also odd, which implies $(\frac{7}{p})=1$, a contradiction. Hence the invariant is trivial.

{\bf Case $p\mid b$ and $v_p(b)$ even or $(\frac{7}{p})=1$.} This follows from Lemma \ref{lem:trivial2tor1}. 

{\bf Case $p=7$.} Let $w,x_1,x_2,x_3\in \Z_7$ be a solution. If $v_7(x_2^2+7x_3^2)$ is odd, then $v_7(2bx_2^6+2\cdot7^3bx_3^6)$ is also odd. Moreover, this implies that $v_7(2bx_2^6+2\cdot7^3bx_3^6)\equiv 3\bmod 6$. However $v_7(w^2-4\cdot7a^2x_1^6)\equiv0,1,2,4\bmod 6$ always, a contradiction. Thus $v_7(x_2^2+7x_3^2)$ is even, in which case $(-1,x_2^2+7x_3^2)_p=0$.
\end{proof}

\begin{proposition}\label{prop:lower_bound_3}
	There are $\gg T^{3/2}/(\log T)^{3/8}$ non-zero integers
	$A_1,A_2,A_3$ such that $\max\{|A_i|\} \leq~T$, 
	$A_2/A_3$ is a cube, and there is a Brauer--Manin
	obstruction to the Hasse principle given by $\calB$.
\end{proposition}
\begin{proof}
We count the number of surfaces in Proposition \ref{prop:lower_bound_family3}.
Let $\varpi(b)$ be the indicator function for those integers $b$ with 
$$p\mid b \mbox{ and } \left(\frac{7}{p}\right)=-1 \implies  p\equiv2\bmod3 ,p\equiv 3\bmod 4.$$ 
This is easily seen to be a frobenian multiplicative function of mean $5/8$ (cf.~the proof of Lemma \ref{lem:square}).
Applying Proposition \ref{prop:lower_bound_frob}, we see that the quantity in question is bounded below by
\begin{align*}
  \gg \sum_{\substack{a \ll \sqrt{T} \\ a \equiv 1 \bmod 2 \cdot 3 \cdot 7}}
  \sum_{\substack{b \leq T \\
  b \equiv 1 \bmod 7 \cdot 8}}
  \varpi(b)
    &\gg T^{1/2}
    \cdot T/(\log T)^{3/8}
    \gg T^{3/2}/(\log T)^{3/8}. \qedhere
\end{align*}
\end{proof}

\begin{remark}
    There is an important difference
    between how the powers of $\log T$ arise
    in Theorems \ref{thm:2} and \ref{thm:3+6}.
    In Theorem \ref{thm:2} a positive
    proportion of the surfaces are everywhere
    locally soluble (this follows from  \cite[Thm.~1.4]{BBL16}); all the logarithmic
    savings come from forcing the local
    invariants of $\calA$ to be almost
    always constant.
    However in Theorem \ref{thm:3+6}, almost
    all the surfaces are \emph{not} everywhere
    locally soluble: it follows from Lemma \ref{lem:surjective2tor3} that for certain primes $p$ the existence of $\Q_p$-point implies that
    $-A_2/A_3 \in \Q_p^{\times 2}$. This gives an extra saving
    which is nothing to do with  $\calB$.
\end{remark}

\section{Brauer--Manin obstruction from multiple algebras}
In this last section, we study the Brauer--Manin obstruction associated to various combinations of the algebras $\calA_i, \calB_i, \calC_i$ from \S\ref{sec:Brauer_elements}. We will show using similar methods to the previous sections that these are all negligible for Theorem \ref{thm:main2}.

\begin{theorem}\label{thm:multalg}
    Let $A_1,A_2,A_3$ be non-zero integers with $\max\{|A_1|,|A_2|,|A_3|\} \leq T$ which satisfy one of the conditions in Table \ref{table:multiple algebras} and let $B_i$ be the stated
    subset of $\Br X_{\AA}$ corresponding to label $(i)$.
    Then Table \ref{table:multiple algebras} gives an upper bound for the number of such surfaces with a Brauer--Manin obstruction to the Hasse principle given by $B_i$.
\end{theorem}
\begin{table}[!htb]
\begin{center}
    \begin{tabular}{ | l | l | l | l | }
    \hline
    Label & Conditions & Brauer elements $B_i$  & Upper bound\\ \hline
    (1) & $-3A_1,-3A_2\in\Q^{\times2}$ & $\calA_1,\calA_2$ & $T^{4/3+\varepsilon}$\\ \hline
    (2) & $\begin{array}{@{}l} -3A_1\in\Q^{\times2}, \\ A_1/A_2\in\Q^{\times3}\end{array}$ & $\calA_1,\calB_3$ & $T^{1+\varepsilon}$\\ \hline
    (3) & $27A_1/A_2\in\Q^{\times6}$ & $\calB_3,\calC_3$ & $T^{3/2}\log\log T/(\log T)^{1/2}$\\ \hline
    (4) & $\begin{array}{@{}l} -3A_1,-3A_2\in\Q^{\times2}, \\ A_1/A_2\in\Q^{\times3}\end{array}$ & $\calA_1,\calA_2,\calB_3$ & $T^{1+\varepsilon}$\\ \hline
    (5) & $\begin{array}{@{}l} -3A_1,-A_2\in\Q^{\times2},\\ A_1/A_2\in\Q^{\times3}\end{array}$ & $\calA_1,\calB_3,\calC_3$ & $T^{1+\varepsilon}$\\ \hline
	\end{tabular}
\end{center}
\caption{}
\label{table:multiple algebras}
\end{table}

\subsection{Local invariants}
We begin with an analogue of Proposition~\ref{prop:surjective_2} and Lemma \ref{lem:surjective2tor} from the single algebra case. For any element $b\in \Br X$, let $\ord(b)$ denote its order. Recall that $\ord(\calA_i)=3, \ord(\calB_i)=\ord(\calC_i)=2$.

\begin{lemma}\label{lem:surjmultalg}
For each label $1\leq i\leq 5$ on Table \ref{table:multiple algebras}, let $B_i$ be the corresponding set of Brauer elements listed. There exists a finite set of primes $S$ with the following property. Let $p\notin S$ be a prime with
\begin{enumerate}
\item $p\nmid A_1A_2$;
\item $\ord(b)\nmid v_p(A_3)$ for every $b\in B_i$.
\end{enumerate}
Then the local invariant map
$$X(\Q_p)\to \prod_{b\in B_i}\Z/(\ord(b))\Z,\quad x_p\to \prod_{b\in B_i} \inv_p b(x_p)$$
is surjective.
\end{lemma}

\begin{proof}
We only prove for the case $i=1$, as all other cases are small modifications of the same argument. The proof uses a similar approach to the proof of  Proposition~\ref{prop:surjective_2}.  We  highlight the main differences which are the equations of the curves. Let $K=\Q(\omega)$.

We may assume that $3 \in S$.
Let $\boldsymbol{\varepsilon} \in (\Z/3\Z)^2$ and let $\frakp\in\Z[\omega]$ be a prime above $p$. Let $n_1,n_2\in\calO^\times_\frakp$ be as in Lemma \ref{surjlem} such that $((n_1,p)_{\omega, \frakp}, (n_2,p)_{\omega, \frakp}) = \boldsymbol{\varepsilon}$. Thus to represent $\boldsymbol{\varepsilon}$, it suffices to find $w,x_1,x_2,x_3\in\Q_p$ and $u_1,u_2\in K_\frakp^\times$ such that
\begin{equation} \label{eqn:n_cube2}
\frac{w-\sqrt{A_1}x_1^3}{w+\sqrt{A_1}x_1^3}=n_1u_1^3,\ \ \frac{w-\sqrt{A_2}x_2^3}{w+\sqrt{A_2}x_2^3}=n_2u_2^3,\ \ w^2=A_1x_1^6+A_2x_2^6+A_3x_3^6.
\end{equation}
By Lemma \ref{lem:res}, the variety given by \eqref{eqn:n_cube2} is isomorphic to the base change to $K_\frakp$ of the following variety
\begin{gather}\label{eqn:n_cube2des}
\begin{split}
&t_1^3 - 3t_1 = \frac{a_1(w^2 +A_1x_1^6)+A_1b_1wx_1^3}{A_2x_2^6+A_3x_3^6},\\ 
&t_2^3 - 3t_2 = \frac{a_2(w^2 +A_2x_2^6)+A_2b_2wx_2^3}{A_1x_1^6+A_3x_3^6},\quad 
w^2=A_1x_1^6+A_2x_2^6+A_3x_3^6.
\end{split}
\end{gather}
where $a_i=n_i+n_i^{-1}, b_i=2(n_i-n_i^{-1})/\sqrt{A_1}$ (cf.~\eqref{eqn:descent}). Hence it suffices to find a $\Q_p$-point on \eqref{eqn:n_cube2des}.
We reduce modulo $p$ and pass to the algebraic closure to study the curve. In particular, we may assume that $A_1=A_2=n_1=n_2=1$, so that the equation modulo $p$ becomes
$$t_1^3 - 3t_1 = 2\frac{w^2 +x_1^6}{x_2^6},\quad t_2^3 - 3t_2 = 2\frac{w^2 +x_2^6}{x_1^6}, \quad w^2=x_1^6+x_2^6.$$
Using \texttt{Magma}, this curve is geometrically integral. Hence for $p$ large enough, there exists a smooth $\F_p$-point which can be lifted to give a $\Q_p$-point on \eqref{eqn:n_cube2des}.
\end{proof}

For the case $27A_1/A_2$ is a sixth power (label (3)), we need an extra lemma.

\begin{lemma}\label{lem:surjective2x2tor2}
    Suppose that $27A_1/A_2$ is a sixth power. There exists a finite set of primes $S$ with the following property. Let $p_1,p_2 \notin S$ be distinct primes such that for both $i=1,2$ the following holds.
	\begin{enumerate}
		\item $v_{p_i}(A_3)$ is even;
		\item $A_3 \notin \Q_{p_i}^{\times2}$;
		\item $v_{p_i}(A_1) = v_{p_i}(A_2)$ is odd;
	\end{enumerate}
	If $X(\Q_{p_1})\times X(\Q_{p_2}) \neq \emptyset$ then
	$$X(\Q_{p_1})\times X(\Q_{p_2})\to (\Z/2\Z)^2,\quad (x_{p_1},x_{p_2})\mapsto \left(\sum_{i=1}^2\inv_{p_i} \calB_3(x_{p_i}),\sum_{i=1}^2\inv_{p_i}\calC_3(x_{p_i})\right)$$
	is surjective.
\end{lemma}

\begin{proof}
As in Lemma \ref{lem:surjective2tor3}, we can assume that $v_{p_i}(A_3)=0$. Let $c=\sqrt[3]{A_2/A_1}$. If $X(\Q_{p_i})\neq\emptyset$, then Lemma \ref{lem:surjective2tor3} implies $-c\in\Q_{p_i}^{\times2}$, which in this case means that $-3\in\Q_{p_i}^{\times2}$ since $3A_2/A_1 \in \Q^{\times2}$. Hence $p_i\equiv1\bmod 3$. Next let $\beta_i\in\Q_{p_i}$ be a sixth root of unity. Let $e_i$ be a positive integer such that $e_i\equiv -v_{p_i}(A_1)\bmod 6$. Since $-c$ is a square in $\Z_{p_i}$, choose a square root $\gamma_i\in\Z_{p_i}$. Let $\alpha_i=\beta_i\gamma_i+{p_i}^{e_i} \in \Q_{p_i}$. Then,
$$A_1\alpha_i^6+A_2=A_1(-c^3+p_i^{e_i}u)+A_2=A_1p_i^{e_i}u$$
for some $u\in \Z_{p_i}^\times$. Hence, by our choice of $e_i$, we have that $v_{p_i}(A_1\alpha_i^6+A_2)\equiv0\bmod 6$. If $p_i$ is large enough, the following equation
$$w^2=A_1(\alpha_i^6+c^3)x_2^6+A_3x_3^6$$
then has a $\Q_{p_i}$-solution with $(w,x_2,x_3) \neq (0,0,0)$. We then set $x_{p_i}:=(w:\alpha_ix_2:x_2:x_3)\in X(\Q_{p_i})$. Using hypothesis (2), for $p_i\neq13$ we find that
\begin{align*}
    \sum_{i=1}^{2}\inv_{p_i} \calB_3(x_{p_i})&=(A_3, \alpha_1^2+c)_{p_1}+(A_3, \alpha_2^2+c)_{p_2}\\
&= \begin{cases}
    1/2, & \mbox{if } \beta_1=\pm1 \text{ or }\beta_2=\pm1,\text{ but not both},\\
    0, & \text{otherwise},\\
    \end{cases}\\
    \sum_{i=1}^{2}\inv_{p_i} \calC_3(x_{p_i})&=(A_3, \alpha_1^2+\alpha_1\sqrt{3c}+c)_{p_1}+(A_3, \alpha_2^2+\alpha_2\sqrt{3c}+c)_{p_2}\\
&= \begin{cases}
    1/2, & \mbox{if } \beta_1=\frac{\pm1-\sqrt{-3}}{2} \text{ or }\beta_2=\frac{\pm1-\sqrt{-3}}{2}\text{ but not both},\\
    0, & \text{otherwise},\\
\end{cases}
\end{align*}
A moment's thought reveals that we can get every possible value of $(\Z/2\Z)^2$ on choosing $\beta_1,\beta_2$ appropriately. We briefly explain the reason for taking $p_i\neq13$ above. We have $$(A_3,\alpha_1^2+\alpha_1\sqrt{3c}+c)_{p_1}=(A_3,-c(\beta_1^2+\beta_1\sqrt{-3}-1)+2\beta_1\gamma_1 p_1^{e_1}+p_1^{2e_1}+p_1^{e_1}\sqrt{3c})_{p_1}.$$
If $\beta_1=\frac{\pm1-\sqrt{-3}}{2}$, then this is $(A_3,2\beta_1\gamma_1 p_1^{e_1}+p_1^{2e_1}+p_1^{e_1}\sqrt{3c})_{p_1}=(A_3, p_1^{e_1}(2\beta_1\gamma_1+\sqrt{3c})+p_1^{2e_1})_{p_1}$. Depending on the root $\gamma_1$ chosen, it is possible that $2\beta_1\gamma_1+\sqrt{3c}=13$. Then the invariant is actually trivial for $p_1=13$.
\end{proof}

\subsection{Proof of Theorem \ref{thm:multalg}}
We now prove the upper bounds based on the above two lemmas.

\begin{proposition}\label{prop:index4}
There are at most $O(T^{4/3+\varepsilon})$ non-zero integers $A_1,A_2,A_3$ such that $\max\{|A_i|\}\leq T$, where $-3A_1,-3A_2$ are squares, and there is a Brauer--Manin obstruction to the Hasse principle given by $\calA_1,\calA_2$.
\end{proposition}

\begin{proof}
Write $A_1=-3a_1^2$, $A_2=-3a_2^2$, and $A_3=gz_3$ where $\rad(g)\mid 3a_1a_2$ and $\gcd(z_3,3a_1a_2)=1$. By Lemma \ref{lem:surjmultalg}, we see that $z_3$ takes a finite set of values modulo cubes. Since the proof is similar, we assume $z_3= a_3^3$ is a cube. In conclusion, we have reduced to finding an upper bound for
$$\#\left\{a_i,g \in \Z :
	|3a_1^2|,|3a_2^2|, |ga_3^3|\leq T,\ \rad(g) \mid 3a_1a_2 \right\}.$$
Setting $a = a_1a_2$, this is
$$
\ll 
\sum_{g\leq T} \sum_{\substack{|a|\leq T \\ \rad(g) \mid 3a}}\sum_{|ga_3^3|\leq T}\tau(a) 
\ll \sum_{g\leq T} \frac{T^{1+\varepsilon}}{\rad(g)} \cdot \frac{T^{1/3}}{g^{1/3}} \ll T^{4/3+\varepsilon} \sum_{g\leq T} \frac{1}{g^{1/3}\rad(g)},
$$
where we used the bound $\tau(a) \ll a^{\varepsilon}$ for the divisor function. The sum over $g$ is convergent by Lemma \ref{lem:rad}.
\end{proof}

\begin{proposition}\label{prop:upper_bound_6}
There are at most $O(T^{1+\varepsilon})$ non-zero integers $A_1,A_2,A_3$ such that $\max\{|A_i|\}\leq T$, where $-3A_1$ is a square, $A_1/A_2$ is a cube, and there is a Brauer--Manin obstruction to the Hasse principle given by $\calA_1,\calB_3$.
\end{proposition}
\begin{proof}
Our conditions imply that  $-3^3A_1A_2^2$ is a sixth power, so write $a^6 = -3^3A_1A_2^2$.
Let $A_3=gz_3$ where $\rad(g)\mid\rad(3A_1A_2)$ and $\gcd(z_3,3A_1A_2)=1$. By Lemma \ref{lem:surjmultalg}, there is a finite set of primes $S$ such that $z_3=y_3a_3$ where $y_3$ is a product of primes in $S$ and $a_3$ is a square-full integer. Since there are only finitely many choices for $y_3$ modulo squares, we assume $y_3=1$ for simplicity. We have reduced to counting
$$\#\left\{a, A_1,A_2,g,a_3\in\Z : \begin{array}{l}
|a|\leq \sqrt{T}, a= -3^3A_1A_2^2, \\
|ga_3|\leq T,\ \rad(g) \mid a, a_3\text{ square-full}\end{array}\right\}$$
The sum over $A_1$ and $A_2$ is $O(T^{\varepsilon})$ by the usual bound for the divisor function. For the remaining variables we have
\begin{align*}
\sum_{g\leq T} 
\sum_{\substack{a \leq \sqrt{T} \\ \rad(g) \mid a}}
\sum_{\substack{|ga_3|\leq T \\ a_3 \text{ square-full}}}1 \ll 
\sum_{g\leq T} \frac{\sqrt{T}}{\rad(g)} 
\cdot \frac{\sqrt{T}}{\sqrt{g}}
\ll T\sum_{g\leq T} \frac{1}{\sqrt{g}\rad(g)}
\end{align*}
and the sum over $g$ is convergent by Lemma \ref{lem:rad}.
\end{proof}

\begin{proposition}\label{prop:upperbound_2x2torsion}
There are at most $O(T^{3/2}(\log \log T)/(\log T)^{1/2})$ non-zero integers $A_1,A_2,A_3$ such that $\max\{|A_i|\}\leq T$, $27A_1/A_2$ is a sixth power, and there is a Brauer--Manin obstruction to the Hasse principle given by $\calB_3,\calC_3$.
\end{proposition}

\begin{proof}
    The proof proceeds in a similar manner to the proof of Proposition \ref{prop:ub_index3}. Firstly, as in  \eqref{eqn:not_squares},
	we may assume that
	\begin{equation} \label{eqn:not_squares_2}
    A_3, -A_1/A_2 \notin \Q^{\times 2},
    \end{equation}
    since in the first case both $\calB_3,\calC_3$ are trivial, and in the
    second case there is a rational point.
    By Lemmas \ref{lem:surjmultalg} and \ref{lem:surjective2x2tor2}, there exists a finite set of primes $S$ such that if $X_{\mathbf{A}}$ has a Brauer--Manin obstruction, then $(A_1,A_2,A_3)$ lie in the following set:
	$$ \left\{ A_i\in \Z :
	\begin{array}{l}
	|A_i|\leq T, \ 27A_1/A_2\in\Q^{\times6}, \eqref{eqn:not_squares_2};\\
	p\notin S,\ p\mid A_3,\ p\nmid A_1A_2\implies v_p(A_3)\text{ even};\\
	\text{there is at most one prime }
	q\notin S \text{ with }v_q(A_3) \text{ even}\\
	\text{such that }v_q(A_1) = v_q(A_2)\text{ is odd}, 
	\text{ and }A_3 \notin \Q_q^{\times 2}
	\end{array} \right\}.$$
	We first consider the case where there is \emph{no} such 
	prime $q$. As in the proof of Proposition \ref{prop:ub_index3},
	for $i \in \{1,2,3\}$ we write
	$$u_i = \sign(A_i)\prod_{p \mid \gcd(A_3,A_1A_2)} p^{v_p(A_i)}.$$
	Let $g = \gcd(A_1/u_1,A_2/u_2)$ and  
	$g_i = \prod_{p \mid 3g}p^{v_p(A_i/u_ig)}$
	for $i = 1,2$. Then we can write
	$$A_3 = u_3z_3, \qquad A_i = gg_iu_iz_i, \quad  i\in \{1,2\}.$$
	By construction $\rad(u_3) = \rad(u_1u_2)$
	and $\rad(g_1g_2) \mid 3g$.
	By the above (Lemma \ref{lem:surjmultalg})  $z_3$ takes finitely many values modulo squares.
	For simplicity we just consider the case where
	$z_3 = a_3^2$ is a square. Similarly, by assumption
	$27A_1/A_2$ is a sixth power; since $\gcd(z_i,3g_ig_ju_iu_jz_j) = 1$ by construction for $\{i,j\} = \{1,2\}$,
	we have $z_i = a_i^6$ for some $a_i$. Combining
	all our conditions together, our quantity is majorised by
$$\# \left\{a_i,u_i,g,g_i \in \Z :
	\begin{array}{l} 
	|u_3a_3^2|,|gg_iu_ia_i^6| \leq T, i=1,2, 27g_1u_1/g_2u_2 \in \Q^{\times 6},
	\eqref{eqn:not_squares_2},\\
	\rad(u_3) = \rad(u_1u_2), \rad(g_1g_2) \mid 3g, \\
	p\notin S, p\nmid g_1g_2u_3, p \parallel  g \implies (\frac{u_3}{p}) = 1.
	\end{array} \right\}$$
	As in the proof of Proposition \ref{prop:ub_index3},
    after summing over $a_3$ and $g$ the sum
    over the remaining variables is convergent.
    Hence we may assume that 
	$a_1,a_2,u_i,g_i \leq (\log T)^A$
	for some large $A > 0$.
	The sum over $a_3$ contributes $O(\sqrt{T})$.
	We sum over $g$ using Lemma \ref{lem:square} 
	with
	$$
	r_1 = u_3 \equiv A_3 \mod \Q^{\times 2}, \quad r_2 = -g_1g_2u_1u_2 \equiv -A_1/A_2 \bmod \Q^{\times 2},
	\quad r_3 = -3.
    $$
	Here $r_2r_3 \equiv 3A_1A_2 \bmod \Q^{\times2}$
	is a square since $27A_1/A_2 \in \Q^{\times2}$,
	thus the sum over $g$ is $O(T/(\log T)^{1/2})$ by Lemma \ref{lem:square}
	and \eqref{eqn:not_squares_2}.
	The sum over the remaining variables is convergent,
	which gives  $O(T^{3/2}/(\log T)^{1/2})$
	under the assumption that there is no such prime $q$.
	
	For the general case, a minor variant of the proof of part (2)
	of Proposition \ref{prop:cube} shows that taking
	$q$ into account gives an extra factor $\log \log T$,
	as required.
\end{proof}

\begin{proposition}\label{prop:upper_bound_12}
There are at most $O(T^{1+\varepsilon})$ non-zero integers $A_1,A_2,A_3$ such that $\max\{|A_i|\}\leq T$, where $-3A_1,-3A_2$ are squares and $A_1/A_2$ is a cube, and there is a Brauer--Manin obstruction to the Hasse principle given by $\calA_1,\calA_2,\calB_3$.
\end{proposition}

\begin{proof} 
Given Lemma \ref{lem:surjmultalg} and ignoring the condition $-3A_2 \in \Q^{\times 2}$, we obtain the exact set of coefficients counted in the proof Proposition \ref{prop:upper_bound_6}, whence the result.
\end{proof}

\begin{proposition}\label{prop:upper_bound_12b}
There are at most $O(T^{1+\varepsilon})$ non-zero integers $A_1,A_2,A_3$ such that $\max\{|A_i|\}\leq T$, where $-3A_1,-A_2$ are squares, and $A_1/A_2$ is a cube, and there is a Brauer--Manin obstruction to the Hasse principle given by $\calA_1,\calB_3,\calC_3$.
\end{proposition}

\begin{proof} 
Again, by  Lemma \ref{lem:surjmultalg} we obtain a subset of coefficients counted in the proof Proposition \ref{prop:upper_bound_6}, whence the result.
\end{proof}

\begin{proof}[Proof of Theorem \ref{thm:multalg}]
    The upper bound for label (1) is Proposition \ref{prop:index4}, (2) is Proposition \ref{prop:upper_bound_6}, (3) is Proposition \ref{prop:upperbound_2x2torsion}, (4) is Proposition \ref{prop:upper_bound_12}, (5) is Proposition~\ref{prop:upper_bound_12b}.
\end{proof}

\part{Bounding the transcendental Brauer group}\label{part:transcendental}
Recall that the transcendental Brauer group of $X_\AA$ is defined as the quotient $\Br(X_\AA)/\Br_1(X_\AA)$. The aim of this part is to calculate $\Br(X_\AA)/\Br_1(X_\AA)$ as an abstract abelian group. The main results carried over to Part 3 are Propositions~\ref{invariants}, \ref{prop:3transcendental} and \ref{prop:2transcendental}, which together imply that $X_\AA$ with non-trivial transcendental Brauer classes are so rare that their count is dominated by the count of $X_\AA$ with an algebraic Brauer-Manin obstruction to the Hasse principle. As a consequence, we never have to compute Brauer--Manin obstructions arising from the transcendental part (which we do not know how to do in general).

We remark that the surface $X_\AA$ is rather special amongst K3 surfaces in the sense that its transcendental lattice is highly structured. More precisely, the rational Hodge endomorphisms of the transcendental lattice of $X_\AA$ are isomorphic to the CM field $\BQ(\sqrt{-3})$ so that $X_\AA$ exhibits complex multiplication in the sense of Pjateckiĭ-Šapiro and Šafarevič \cite{PSS}. It is because of this property (shared amongst all K3 surfaces of Picard number $20$) that we are able to control the transcendental cycles.

For an extension of the results which fully classifies the transcendental Brauer group of $X_\AA$, see \cite[\S III.12.4]{gvirtz}.

\section{Cohomology of weighted diagonal hypersurfaces}\label{sec:cohomology}
In this section, we develop a full description of the middle cohomology of smooth weighted diagonal hypersurfaces of dimension $n$, which will be applied to the surface $X_\AA$ in \Cref{sextic}. By ``full'', we mean an explicit understanding of the integral singular cohomology including the cup product, the Hodge cohomology, the Galois action on $\ell$-adic cohomology (up to an undetermined action of cyclotomic Galois automorphisms -- cf. \Cref{rem:cyclotomic}) and the comparison isomorphisms between those. We will build on previous work by Pham \cite{pham}, Looijenga \cite{looijenga}, Weil \cite{weil} and Ulmer \cite{ulmer}. The approach closely follows a framework developed by Gvirtz--Skorobogatov \cite{qua}.

Let $d_0,\dots,d_{n+1}\in\NN$. Set $d=\lcm(d_0,\dots,d_{n+1})$, $q_i=d/d_i$, $q^*=\prod_{i=0}^{n+1}q_i$ and $\bm q=(q_0,\dots,q_{n+1})$.
\begin{assumption}\label{ass:qi-coprime} \hfill
\begin{enumerate}
 \item[(a)] The integers $q_1,\dots,q_{n+1}$ are pairwise coprime.
 \item[(b)] $q_{n+1}=1$.
\end{enumerate}
\end{assumption}
\Cref{ass:qi-coprime}(a) will ensure smoothness (cf.\ \Cref{rem:smooth}).  \Cref{ass:qi-coprime}(b) is used only for the proof of \Cref{thm:complete-cohomology}(1) and (2). Presumably, the results hold without \Cref{ass:qi-coprime}(b) but its inclusion simplifies the argument considerably while still covering the intended scope of this article.

For $i=0,\dots,n+1$, we write $\mu_{d_i}=\langle u_i\rangle$ for the cyclic group of order $d_i$ with generator $u_i=e^{2i\pi/d_i}$. The following group plays a central role:
\[G=G_{d_0,\dots,d_{n+1}}=(\mu_{d_0}\times\dots\times\mu_{d_{n+1}})/\langle u_0\dots u_{n+1}\rangle.\]
There is an isomorphism \[G\cong(\mu_{d_0}\times\dots\times\wh{\mu_{d_i}}\times\dots\times\mu_{d_{n+1}})/\langle \prod_{j\neq i}u_j^{d_i}\rangle\]
 that identifies $u_i$ with $\prod_{j\neq i}u_j^{-1}$. We will freely switch between representing elements of $G$ in the symmetric or asymmetric notation, depending on which is more convenient.

We now describe the cohomology as promised before giving a section-by-section proof.
\begin{theorem}\label{thm:complete-cohomology}
 Let $k$ be a number field with a fixed embedding $k\subset\BC$.
 
 Let $\bm A = (A_1,\dots,A_{n+1})\in (k^\times)^{n+1}$ and let
 \[F:\quad f\coloneqq x_0^{d_0}+A_1x_1^{d_1}\dots+A_{n+1}x_{n+1}^{d_{n+1}}=0\]
 be a weighted projective diagonal hypersurface over $k$ satisfying \Cref{ass:qi-coprime}. Set $n'=\lfloor n/2\rfloor$ and $\zeta_m=e^{2i\pi/m}$.
 \begin{enumerate}
  \item \textbf{Primitive cohomology:}\\
  Let $l_i$, for $0\leq i\leq n+1$, be the class of the hyperplane $\{x_i=0\}$ in $\RH^2(F_\BC,\BZ(1))$. Write $H=\RH^n(F_\BC,\BZ(n'))$ and $P\subseteq H$ for the kernel of the cup pairing with $l_i$ (which is independent of $i$). Let $I\subset\BZ[G]$ be the saturation with respect to $d$ of the ideal \[\langle 1+u_i+u_i^2+\dots+u_i^{d_i-1}: i=0,\dots,{n+1}\rangle\subset\BZ[G].\] Then $P$ as a $\BZ$-module is isomorphic to $\BZ[G]/I$.
  \item \textbf{Cup product:}\\
  For two monomials $u$ and $u'$ in the variables $u_0,\dots,u_{n+1}$, the cup product $(u,u')$ under the identification of $(1)$ is equal to the coefficient of $1$ in
  \[(-1)^{n(n+1)/2}(1-u_0)(1-u_1)\dots(1-u_{n+1})uu'^{-1}\in\BZ[G].\]
  \item \textbf{Full cohomology:}\\
  If $n$ is odd, $P$ equals $H$. If $n$ is even, let $L$ be a generator of $\langle l_i^{n/2}\rangle\in H$. Then $L$ has self-intersection $d_{\bm q}=d/q^*$ and $H$ is generated by $P$ and
  \[\frac{1}{d_{\bm q}}L+\frac{(-1)^{n(n+1)/2}}{d}\rho(u_0,u_1)\rho(u_2,u_3)\dots\rho(u_n, u_{n+1})\]
  where $\rho(x,y)=\sum_{0\leq l\leq m\leq  d-2}y^l x^m$.
  \item \textbf{Hodge structure:}\\
  Set \[S=\left\{(a_0,\dots,a_{n+1})\in(\BZ/d)^{n+2}:a_i\neq 0, q_i|a_i,\sum_{i=0}^{n+1}a_i=0\right\}.\] For $\chi=(a_0,\dots,a_{n+1})\in S$, we define
  \[ q(\chi)=\frac{\sum_{i=0}^{n+1}a_i}{d}-1-n' \]
 and elements
  \[\alpha_\chi=\alpha_{a_1}(u_1)\dots\alpha_{a_{n+1}}(u_{n+1})\in P\otimes_\BZ \BC, \quad \mbox{where }
\alpha_{i}(u)=\frac{1}{d}\sum_{j=0}^{d-1}{\zeta_d}^{-ij}u^j.\]
    Then the Hodge structure on $P\otimes_\BZ \BC$ is given by
    \[P^{n-q,q}=\bigoplus_{\substack{\chi\in S\\q(\chi)=q}}\langle\alpha_\chi\rangle.\]
  \item \textbf{Galois action and comparison with $l$-adic cohomology:}\\
  Let $\ell$ be a prime number and $\lambda$ a place of $E=\BQ(\zeta_d)$ lying over $\ell$. Let $\pp\nmid d\ell$ be a finite place of the compositum $K=Ek$. Define the multiplicative character $\psi:\BF_\mathfrak p^\times\to\mu_d$ by the condition \[\psi(x)\bmod \mathfrak p=x^{(\Norm(\mathfrak p)-1)/d},\]
  and for all $\chi=(a_0,\dots,a_{n+1})\in S$, the Jacobi sum \[J_\mathfrak{p}(\chi)=\sum_{x_1+\dots+x_{n+1}=1}\psi(x_1)^{a_1}\dots\psi(x_{n+1})^{a_{n+1}}.\]
  
  Then the induced action of $\Frob_\pp\in\Gal(\ov K/K)$ on $\langle\alpha_\chi\rangle\subset H\otimes_\BZ E_\lambda$ under the canonical comparison isomorphism
  \[\Het^n(F\times_k \ov K, \BZ_\ell(n'))\otimes_{\BZ_\ell} E_\lambda \simeq H\otimes_\BZ E_\lambda\]
  is given by multiplication with
  \[(-1)^n\psi^{a_0}(-1)\Norm(\mathfrak p)^{-n'}J_\mathfrak{p}(\chi)/{\prod_{i=1}^{n+1}\psi(A_i)^{a_i}},\]
  
  \item \textbf{Action of complex conjugation:}\\
  Assume that $k\subset \BR$. Choose roots $A_i^{1/d_i}\in\BC$, $i=1,\dots,n+1$. Under the isomorphism in $(1)$, the action of complex conjugation $\tau$ on $P$ is the $\BZ$-linear extension of the map
  \[u\mapsto (-1)^{n'+1}u_{\bm A}u^{-1}\]
  on monomials $u$ in the variables $u_0,\dots,u_{n+1}$. Here $u_{\bm A}$ is the element
  \[(A_1^{1/d_1}/\tau(A_1^{1/d_1}),\dots,A_{n+1}^{1/d_{n+1}}/\tau(A_{n+1}^{1/d_{n+1}}))\in G.\]
 \end{enumerate}
\end{theorem}

\subsection{Background on weighted diagonal hypersurfaces}\label{sub:background}

We define the group $\mu_{\bm q}=\mu_{q_0}\times\dots\times\mu_{q_{n+1}}$ where $\mu_{q_i}=\langle t_i\rangle$ is the cyclic group of order $q_i$ with generator $t_i=\zeta_{q_i}$.

\begin{definition}
 Let $k[x_0,\dots,x_{n+1}]$ be a polynomial ring equipped with the grading $\deg(x_i)=q_i$ for all $i=0,\dots,n+1$. The \emph{($\bm q$-)weighted projective space} $\BP^{n+1}_k({\bm q})$ is the $(n+1)$-di\-men\-sio\-nal projective scheme $\Proj k[x_0,\dots,x_{n+1}]$.
\end{definition}

\begin{remark}
 Alternatively, $\BP^{n+1}_k({\bm q})$ can be constructed as follows: Let $\mu_{\bm q}$ act on $\BP_k^{n+1}$ such that the action of $t_i$ multiplies the $i$-th coordinate of $\BP_k^{n+1}$ with $\zeta_{q_i}$ for $0\leq i\leq n+1$. Then the quotient by this action is
\[\pi_{\bm q}:\BP_k^{n+1}\to\BP^{n+1}_k({\bm q}),(y_0:\dots:y_{n+1})\mapsto (y_0^{q_0}:\dots:y_{n+1}^{q_{n+1}}).\]
It is easy to see that every weighted projective space is isomorphic to a normalised one satisfying $\gcd({\bm q})=1$. By \Cref{ass:qi-coprime}(a), this holds and we write shorthand $\BP=\BP^{n+1}_k({\bm q})$.
\end{remark}

\begin{definition}
The \emph{weighted diagonal hypersurface of multidegree $(d_0,\dots,d_{n+1})$ with coefficients $A_1,\dots,A_{n+1}$} is the weighted projective hypersurface 
\[F=F_{(d_0,\dots,d_{n+1})}\subset \BP:\quad x_0^{d_0}+A_1x_1^{d_1}+\dots+A_{n+1}x_{n+1}^{d_{n+1}}=0.\] 
\end{definition}
There is a natural quotient map
$\pi_{\bm q}:F_d\to F$
from the Fermat hypersurface
\[F_d\subset \BP_k^{n+1}:y_0^d+A_1y_1^d+\dots+A_{n+1}y_{n+1}^{d}=0\]
of degree $d$ and dimension $n$ to the weighted quotient.

The group $\mu_{d_0}\times\dots\times\mu_{d_{n+1}}$ acts on $F_\BC$. Namely, $u_i$ multiplies $x_i$ with $\zeta_{d_i}$. This action restricts to a trivial action of $\mu_d$ where $\mu_d$ acts via
${\zeta_d}\mapsto({\zeta_d}^{q_0},\dots,{\zeta_d}^{q_{n+1}})$, yielding an action of $G$ on $F_\BC$.

Similarly, the group $G_d=G_{d,\dots,d}=\mu_d^{n+2}/\mu_d$ acts on the Fermat hypersurface $(F_d)_\BC$ such that the generators $(v_i)_{i=0,\dots,n+1}$ for the factors of $\mu_d^{n+2}$ multiply the $i$-th coordinate with $\zeta_d$. There is a surjective morphism $G_d\to G$ sending $v_i$ to $u_i$, which is compatible with $\pi_{\bm q}$ and the group actions on $(F_d)_\BC$ and $F_\BC$.

\begin{remark}\label{rem:smooth}
The singularities of $F$ have been analysed by {Y.~Goto}. He proves that $F$ is smooth if and only if $\gcd(q_i,q_j)=1$ for all $i\neq j$ between $0$ and $n+1$ \cite[Prop.~2.1]{goto}. This explains \Cref{ass:qi-coprime}(1) although we expect that our approach can be extended to the singular case by analysing the Hirzebruch resolution of the appearing cyclic quotient singularities.
\end{remark}

\begin{example}
By \cite[Thm.~3.2.4 and 3.3.4]{dolgachev}, $\RH^i(F,\CO_F)=0$ for $0<i<n$ and the dualising sheaf of $F$ is $\omega_F=\CO_F(d-q_0-q_1-\dots-q_{n+1})$. In the case of $n=2$, it is easy to check that this implies a finite list of weighted diagonal surfaces over $\BC$ whose minimal resolution is K3, of which precisely two cases are smooth (see the last two entries in Table 7 of \cite[Prop.~8.1]{goto}). These are the Fermat surface of degrees $(4,4,4,4)$, i.e.\ $(d,\bm q)=(4,(1,1,1,1))$, and the diagonal degree $2$ K3 surface of degrees $(2,6,6,6)$, i.e.\ $(d,\bm q)=(6,(3,1,1,1))$, which this article considers. 
\end{example}

The analogues of the Lefschetz hyperplane theorem
\[\RH^i(F_\BC,\BZ)\cong\RH^i(\BP_\BC,\BZ)\cong \begin{cases}\BZ,& 2\mid i\\ 0,& 2\nmid i\end{cases},\quad i<n\]
and Poincaré duality hold and the integral cohomology of $F_\BC$ is torsion-free \cite[B19, B22 and B32]{dimca}. (Readers should take note there is an incorrect shift of the cohomological degree in the weighted Lefschetz hyperplane theorem as stated in \cite[p.~65]{dolgachev}.) 
For this reason, our interest lies in the middle cohomology and, as far as the transcendental Brauer group is concerned, in surfaces.

\subsection{Primitive cohomology and cup product -- Proof of (1) and (2)}
\begin{lemma}\label{lem:hyperplane}
 The generator $l$ of $\langle l_i \rangle\in\RH^2(F_\BC,\BZ(1))$ is independent of the choice of $i\in\{0,\dots,n-1\}$. If $n$ is even, then $L=l^{n/2}$ is a generator of $\langle l_i^{n/2}\rangle\subset H$ and the self-intersection of $L$ is $d_{\bm q}=d/q^*$.
\end{lemma}
\begin{proof}
 We have the following commutative diagram for $0\leq j\leq n$ where the bottom map is an isomorphism by the Cartan--Leray spectral sequence:
 \[\begin{tikzcd}
  \RH^{2j}(F_\BC,\BZ(j)) \arrow[r, "\pi_{\bm q}^*"]\arrow[d,"\otimes \BQ"] & \RH^{2j}((F_d)_\BC,\BZ(j))^{\mu_{\bm q}}\arrow[d,"\otimes \BQ"]\\
  \RH^{2j}(F_\BC,\BQ(j)) \arrow[r,"\pi_{\bm q}^*","\sim" swap] & \RH^{2j}((F_d)_\BC,\BQ(j))^{\mu_{\bm q}}.
 \end{tikzcd}\]
 The left map is injective and therefore so is the top map. The class $l_i$ pulls back to $q_il'$ where $l'$ is the hyperplane class of $F_d$. Because $\gcd(\bm q)=1$, there exists a linear combination $l$ of $(l_i)_{i=0,\dots,n+1}$ with $\pi_{\bm q}^*(l)=l'$. Since $l'$ is not the multiple of any cohomology class, it follows that $l$ is a generator of $\langle l_i\rangle$, and if $n$ is even, $l^{n/2}$ is a generator of $\langle l_i^{n/2}\rangle$. The self-intersection of $L$ is $\langle l'^{n/2},l'^{n/2}\rangle/\#\mu_{\bm q}=d/q^*$.
\end{proof}

\begin{definition}
For a free abelian group $M$ with a bilinear form $Q:M\times M\to\BZ$ and an action of a group $G$ on $M$ preserving $Q$, i.e.\ $Q(x,y)=Q(gx,gy)$ for all $x,y\in M, g\in G$, define the \emph{sesquilinear extension}
\begin{eqnarray*}
M\times M\to\BZ[G],\quad
x,y\mapsto x*y:=\sum_{g\in G}Q(x,gy)g\in\BZ[G]. 
\end{eqnarray*}
Here, sesquilinearity means that $g(x*y)=gx*y=x*g^{-1}y$.
\end{definition}

We recall a topological description of the singular middle homology of the affine diagonal hypersurface
\[U: x_0^{d_0}+A_1x_1^{d_1}+\dots+A_{n}x_{n}^{d_{n}}=-A_{n+1}\subset \BA_k^{n+1}\]
due to Pham \cite{pham}. Note that $G\cong\mu_{d_0}\times\dots\times\mu_{d_{n}}$ acts on $U_\BC$.

\begin{proposition}
 Let
\[\Delta^n=\{z\in\BR^{n+1}:z_0+\dots+z_{n}=1,z_i\geq 0,\forall i=0,\dots, n\}\]
be the standard $n$-simplex. Fix roots $B_0=A_{n+1}^{1/d_0}, B_1=(A_{n+1}/A_1)^{1/d_1},\dots, B_n=(A_{n+1}/A_n)^{1/d_n}\in\BC$ and set
\[\bm e':\Delta^n \to U(\BC),\quad
(z_0,\dots,z_{n}) \mapsto (\zeta_{2d_0}B_0z_0^{1/d_0},\dots,\zeta_{2d_{n}}B_nz_n^{1/d_{n}})
\]
where the roots of the $z_i$ are chosen to be positive real numbers.

\begin{enumerate}\label{prop:fermat-homology}
 \item The sum $e'=(1-u_0^{-1})\dots(1-u_{n}^{-1})_*\bm e'$ is a cycle belonging to the homology group $\RH_n(U_\BC,\BZ)$.
 \item The cycle $e'$ generates $\RH_n(U_\BC,\BZ)$ as a $\BZ[G]$-module. More precisely, 
 as $\BZ[G]$-modules we have $\RH_n(U_\BC,\BZ)\cong R'$, where \[R'= \BZ[G]/\langle1+u_i+u_i^2+\dots+u_i^{d_i-1}: i=0,\dots,{n}\rangle.\]
 \item The sesquilinear extension $*$ of the intersection product on $\RH_n(U_\BC,\BZ)$ is characterised by \[e'*e'=(-1)^{n(n+1)/2}(1-u_0)\dots(1-u_{n})(1-(u_0\dots u_{n})^{-1}).\]
 (This value determines $*$ completely by sesquilinearity.)
\end{enumerate}
\end{proposition}

Our definition of $e'$ differs from {Pham's} cycle, but only by the element $\prod_{i=0}^{n}(-u_i)^{-1}$ which is invertible in $\BZ[G]$. In particular, the intersection product, which is preserved by the $G$-action, is as described by Pham.

\begin{proposition}\label{primitive}
 Let
\[\Delta^n=\{z\in\BR^{n+1}:z_0+\dots+z_{n}=1,z_i\geq 0,\forall i=0,\dots, n\}\]
be the standard $n$-simplex. Fix roots $B_0=A_{n+1}^{1/d_0}, B_1=(A_{n+1}/A_1)^{1/d_1},\dots, B_n=(A_{n+1}/A_n)^{1/d_n}\in\BC$ and set
\[
\bm e:\Delta^n \to F(\BC),\quad
(z_0,\dots,z_{n}) \mapsto (\zeta_{2d_0}B_0z_0^{1/d_0}:\dots:\zeta_{2d_{n}}B_nz_n^{1/d_{n}}:1)
\]
where the roots of the $z_i$ are chosen to be positive real numbers.

\begin{enumerate}
 \item The sum $e=(1-u_0^{-1})\dots(1-u_{n}^{-1})_*\bm e$ is a cycle belonging to the primitive homology (i.e. the orthogonal complement of any hyperplane with respect to the intersection pairing) $\RP_n(F_\BC,\BZ)$.
 \item The cycle $e$ generates $\RP_n(F_\BC,\BZ)$ as a $\BZ[G]$-module. More precisely, let $I\subset\BZ[G]$ be the saturation with respect to $d$ of the ideal \[J=\langle 1+u_i+u_i^2+\dots+u_i^{d_i-1}: i=0,\dots,{n+1}\rangle\subset{\BZ[G]}.\] Then $\RP_n(F_\BC,\BZ)$ as a $\BZ[G]$-module is isomorphic to $\BZ[G]/I$.
 \item The sesquilinear extension $*$ of the intersection product on $\RP_n(F_\BC,\BZ)$ is characterised by \[e*e=(-1)^{n(n+1)/2}(1-u_0)(1-u_1)\dots(1-u_{n+1}).\]
 (This value determines $*$ completely by sesquilinearity.)
\end{enumerate}
\end{proposition}
\begin{proof}
 We assume that $q_{n+1}=1$ by \Cref{ass:qi-coprime}(b).
 
 Let $\iota:Z\hookrightarrow F$ be the hyperplane section given by $x_{n+1}=0$. Its complement is the open subvariety $j:U\hookrightarrow F$. The Gysin sequence in homology of the smooth pair $(F,Z)$ yields an exact sequence
\[
 \RH_{n}(U_\BC,\BZ)\xrightarrow{j_*}\RH_{n}(F_\BC,\BZ)\xrightarrow{\iota^*}\RH_{n-2}(Z_\BC,\BZ)\to 0.
\]
where $\iota^*$ is Poincaré dual to pullback in cohomology. If $n$ is even, the kernel of $\iota^*$ is thus dual to $\langle l_{n+1}^{n/2}\rangle$, while if $n$ is odd $\RH_{n-2}(Z_\BC,\BZ)=0$. Thus, we have a surjective morphism
\[\RH_{n}(U_\BC,\BZ)\xrightarrow{j_*}\RP_{n}(F_\BC,\BZ).\]
The map $j_*$ is compatible with the intersection products, hence $\ker(j_*)$ is contained in the kernel of the intersection product on $\RH_{n}(U_\BC,\BZ)$. On the other hand, the intersection product on $\RP_{n}(F_\BC,\BZ)$ is non-degenerate. We thus identify $\ker(j_*)$ with the kernel of the intersection product on $\RH_{n}(U_\BC,\BZ)$.

It remains to show that the kernel of the intersection product is equal to the saturation of \[J'\coloneqq JR'=\langle 1+u_{n+1}+u_{n+1}^2+\dots+u_{n+1}^{d_{n+1}-1}\rangle\subset{R'}\] with respect to $d$. 

By \Cref{prop:fermat-homology}(3), the kernel of the intersection product is isomorphic to $\Ann_{R'}(1-u_0)(1-u_1)\dots(1-u_{n+1})$. Now in $R'\otimes \BQ$, the elements $(1-u_i)$, $i=0,\dots,n$, are invertible and so
\begin{align*}
\ker(j_*)\otimes\BQ&=\Ann_{R'\otimes\BQ}((1-u_0)(1-u_1)\dots(1-u_{n+1}))\\
&=\Ann_{R'\otimes\BQ}(1-u_{n+1})=J'\otimes\BQ.
\end{align*}
Therefore, $\ker(j_*)=(J'\otimes\BQ)\cap R'$. By definition, the quotient
$\ker(j_*)/J'$ equals the Tate cohomology $\wh\RH^0(\mu_d,R')$ where $\mu_d$ acts on $R'$ by multiplication with $u_{n+1}$. This cohomology group is killed by $d$.
\end{proof}

This completes the proof of \Cref{thm:complete-cohomology}(1) and (2) by applying Poincaré duality for $F$.

\begin{remark}
 Our proof follows the same idea as \cite[Cor.~2.2]{looijenga}. However, the proof given there is incomplete and the statement of Cor.~2.2 is incorrect since in general, the inclusion $J\subsetneq I$ is strict, or equivalently, the Tate cohomology in the proof of \Cref{primitive} is non-trivial. We conjecture that $I/J$ is isomorphic to $\BZ/d_{\bm q}$ when $n$ is odd and $0$ when $n$ is even. For any given $d$ and $n$ (as in \Cref{sextic}), this can be verified explicitly. 
\end{remark}

\subsection{Full cohomology -- Proof of (3)}
Because the primitive cohomology is equal to the full cohomology if $n$ is odd, we restrict to the case of even $n$ for this section only.
\begin{corollary}\label{cor:dq}
 Assume that $n$ is even. Then $H/(P\oplus\BZ L)\cong P^*/P \cong (\BZ L)^*/\BZ L \cong \BZ/d_{\bm q}$ (where for a $\BZ$-lattice $M$, $M^*=\Hom(M,\BZ)$ denotes its dual).
\end{corollary}
\begin{proof}
  The last isomorphism follows from \Cref{lem:hyperplane}. The other isomorphisms exist because $P$ is the orthogonal complement of $L$ in the unimodular lattice $H$.
\end{proof}

The group $(\BZ L)^*/\BZ L$ is generated by the class of the linear map $\langle\frac{1}{d_{\bm q}}L,\cdot\rangle$. We deduce that $H/(P\oplus\BZ L)$ is generated by the image of $\frac{1}{d_{\bm q}}(L+\xi)$ for some $\xi\in P$. The integrality of the cup product requires that $\langle \xi,P\rangle\subset d_{\bm q}\BZ$. Note that $\xi$ is only uniquely determined in $P/d_{\bm q}P$. Our aim is to show that
\[\xi=\frac{1}{q^*}(-1)^{n(n+1)/2}c\in P\] 
with $c=\rho(u_0,u_1)\rho(u_2,u_3)\dots\rho(u_n, u_{n+1})$
as in \Cref{thm:complete-cohomology}(3) is one of the many possible lifts to $P$.

Define the polynomial function $\phi(x)=\sum_{i=0}^{d-1}x^i$. We need the following easy identities.
\begin{lemma}\label{identities} \hfill
\begin{enumerate}
 \item $(1-y)\phi(xy)=(1-x)(1-y)\rho(x,y)$ inside the ring $\BZ[x,y]/\langle x^d-1,y^d-1\rangle$.
 \item $(1-x)\rho(1,x)=d$ inside the ring $\BZ[x]/\langle\phi(x)\rangle$, in particular $1-x$ is invertible in $\BQ[x]/\langle\phi(x)\rangle$.
 \item $\phi(xy)=(1-x)\rho(x,y)$ inside the ring $\BZ[x,y]/\langle\phi(x),\phi(y)\rangle$.
\end{enumerate}
\end{lemma}
Let $\Lambda\in\RH_n((F_d)_\BC,\BZ)$ be the homology class of the linear subspace given by
\[\zeta_{2d}y_0=A_1^{1/d}y_1,\zeta_{2d}A_2^{1/d}y_2=A_3^{1/d}y_3,\dots, \zeta_{2d}A_n^{1/d}y_{n}=A_{n+1}^{1/d}y_{n+1}.\]
Because the intersection number of $\Lambda$ with a hyperplane section $L_d$ of $F_d$ is $1$, it follows that $\Lambda$ generates $\RH_n((F_d)_\BC,\BZ)$ modulo primitive homology.

\begin{proposition}
 Let \begin{eqnarray*}
c_d&=&(1-v_0)^{-1}\phi(v_0v_1)(1-v_2)^{-1}\phi(v_2v_3)\dots(1-v_n)^{-1}\phi(v_n v_{n+1})\\
&=&\rho(v_0,v_1)\rho(v_2,v_3)\dots\rho(v_n,v_{n+1})\in \RP_n((F_d)_\BC,\BZ).      
     \end{eqnarray*}
 Then  $\frac{1}{d}(L_d+(-1)^{n(n+1)/2}c_d)=\Lambda$. 
\end{proposition}

\begin{proof}
 The intersection product $(\cdot,\cdot)$ on $\RH_n((F_d)_\BC,\BZ)$ is non-degenerate, hence we only need to show that the images of $\frac{1}{d}(L_d+(-1)^{n(n+1)/2}c_d)$ and $\Lambda$ in $\RH_n((F_d)_\BC,\BZ)^*$ are equal. It is clear that
 \[( \Lambda,L_d)=1=\frac{1}{d}(L_d,L_d)=( \frac{1}{d}(L_d+(-1)^{n(n+1)/2}c_d),L_d)\]
 and $( \frac{1}{d}(L_d+(-1)^{n(n+1)/2}c_d),x)=\frac{1}{d}((-1)^{n(n+1)/2}c_d,x)$ for all $x\in \RP_n((F_d)_\BC,\BZ)$.
 
 {Degtyarev} and {Shimada} have computed in \cite[p.\ 989, Proof of Part (a) of Thm.~1.1]{degtyarev2} that the image of $\Lambda$ under the map
 \[\ev: \RH_n((F_d)_\BC,\BZ)\to\BZ[G_d], x\to\sum_{g\in G_d}(x,g)g\]
 is given by $\psi:=(1-v_1)(1-v_3)\dots(1-v_{n+1})\phi(v_2v_3)\dots\phi(v_nv_{n+1})$. So it remains to show that $\ev(c_d)=(-1)^{n(n+1)/2}d\psi$.
 
 Using the $G_d$-invariance of the intersection product on $F_d$, we get
 \[\ev(h)=\sum_{g\in G_d}(h,g)g=\sum_{g\in G_d}(1,gh^{-1})g=\sum_{g\in G_d}(1,g)gh=\ev(1)h\]
 for all $h\in G_d$ and by bilinearity of the intersection product, the same equation holds for $h\in \RP_n((F_d)_\BC,\BZ)$.
 
 Recall that by \Cref{thm:complete-cohomology}(2),
 $\ev(1)=(-1)^{n(n+1)/2}(1-v_0)\dots(1-v_{n+1})$.
 Thus $(-1)^{n(n+1)/2}\ev(c_d)$ equals
 \begin{eqnarray*}
  &&(1-v_0)(1-v_1)\dots(1-v_{n+1})\rho(v_0,v_1)\rho(v_2,v_3)\dots\rho(v_n,v_{n+1})\\
  &=&(1-v_1)(1-v_3)\dots(1-v_{n+1})\phi(v_0v_1)\phi(v_2v_3)\dots\phi(v_nv_{n+1})\\
  &=&(1-v_1)(1-v_3)\dots(1-v_{n+1})\phi(v_2\dots v_{n+1})\phi(v_2v_3)\dots\phi(v_nv_{n+1})\\
  &=&(1-v_1)(1-v_3)\dots(1-v_{n+1})d\phi(v_2v_3)\dots\phi(v_nv_{n+1})=d\psi. 
 \end{eqnarray*} 
\end{proof}

The image of $\frac{1}{d}(L_d+(-1)^{n(n+1)/2}c_d)$ under the pushforward map \[(\pi_{\bm q})_*:\RH_n((F_d)_\BC,\BZ)\to\RH_n(F_\BC,\BZ)\] is \[\frac{1}{d_{\bm q}}L+\frac{1}{d}(-1)^{n(n+1)/2}c.\]
Here we use that $(\pi_{\bm q})_*(\pi_{\bm q})^*$ equals $\deg\pi_{\bm q}=q^*$. As a consequence, we infer that $\xi=\frac{1}{q^*}(-1)^{n(n+1)/2}c\in P$ is a possible choice such that $H/(P\oplus\BZ L)$ is generated by $\frac{1}{d_{\bm q}}(L+\xi)$. This finishes the proof of \Cref{thm:complete-cohomology}(3).

\subsection{Hodge structure -- Proof of (4)}
We write \[\wh G=\Hom(G,\BC^\times)=\left\{a\in(q_1\BZ/d\times\dots\times q_{n+1}\BZ/d): q_0\mid\sum_{i=1}^{n+1}a_i\right\}\] for the group of complex characters of $G$. In the symmetric notation,
\[\widehat G\cong\left\{a\in (q_0\BZ/d\times\dots\times q_{n+1}\BZ/d):\sum_{i=0}^{n+1}a_i=0\in\BZ/d\right\}.\]
A tuple $(a_1,\dots,a_{n+1})\in q_1\BZ/d\times\dots\times q_{n+1}\BZ/d$ corresponds to the character defined by
\[\chi(u_1^{l_1}\dots u_{n+1}^{l_{n+1}})={\zeta_d}^{a_1l_1+\dots+a_{n+1}l_{n+1}}.\]
Attached to $\chi$ is an element
\[\alpha_\chi=\alpha_{a_1}(u_1)\dots\alpha_{a_{n+1}}(u_{n+1})\in E[G]\subset \BC[G], \quad \mbox{where }
\alpha_{i}(u)=\frac{1}{d}\sum_{j=0}^{d-1}{\zeta_d}^{-ij}u^j.\]

The family $(\alpha_\chi)_{\chi\in\wh G}$ is a basis of idempotent eigenvectors: it satisfies $\alpha_\chi\alpha_\rho=\delta_{\chi,\rho}$ where $\delta$ is the Kronecker delta. One easily checks that $g\alpha_\chi=\chi(g)\alpha_\chi$ for all $g\in G,\chi\in\wh G$.

\begin{proposition}\label{prop:decomp}
 We have an equality of $E[G]$-modules $P\otimes E=\bigoplus_{\chi\in S} V_\chi$ where $V_\chi=\langle \alpha_\chi\rangle$.
\end{proposition}

\begin{proof}
The classical representation theory of finite groups gives that after extending the base to $E$, the $\BZ[G]$-module $\BZ[G]$ decomposes into a sum of the $1$-dimensional eigenspaces $V_\chi$:
\[E[G]=\bigoplus_{\chi\in\wh G}V_\chi.\]
By \Cref{thm:complete-cohomology}(1), $P\otimes E$ is the quotient of $E[G]$ by the ideal
$\langle\phi_i(u_i):i=0,\dots,n+1\rangle$. We find that
\[\phi_i(u_i)E[G]=\bigoplus_{\chi\in S_i} V_\chi\]
for $i=0,\dots,n+1$, where $S_i$ is the set of all characters $\chi\in\wh G$ restricting trivially to the factor $\mu_{d_i}$. Thus the $E[G]$-module $P\otimes E$ decomposes into a sum of eigenspaces over all characters in $S=\widehat G\setminus\bigcup_{i=0}^{n+1} S_i$.
\end{proof}

It will be useful in \Cref{sub:T} to express the cup product in terms of idempotents:
\begin{lemma}\label{idemcup}
 Set $\Xi=\mathrm{Re}$ (the real part) for even $n$ and $\Xi=i\cdot\mathrm{Im}$ ($i$ times the imaginary part) for odd $n$.
 For all $\chi,\rho\in S$,
 \[\langle\alpha_\rho,\alpha_\chi\rangle=(-1)^{n(n+1)/2}q^*\frac{2}{d^{n+1}}\Xi((1-{\zeta_d}^{-a_1})\dots(1-{\zeta_d}^{-a_{n+1}}))\delta_{\rho^{-1},\chi}\]
 where $\chi$ corresponds to $(a_1,\dots,a_{n+1})$.
\end{lemma}
\begin{proof}
 Using the bilinearity of the cup product, we find that
 $\langle\alpha_\rho,\alpha_\chi\rangle = \langle 1,\alpha_{\rho^{-1}}\alpha_\chi\rangle$.
 From the idempotency property, it follows that $\langle\alpha_\rho,\alpha_\chi\rangle=0$ if $\chi\neq\rho^{-1}$. If $\chi=\rho^{-1}$, then
 $\langle\alpha_{\chi^{-1}},\alpha_\chi\rangle =\langle1, \alpha_\chi^2 \rangle= \langle1, \alpha_\chi \rangle$
is the coefficient of $1$ in the expression
 \[(-1)^{n(n+1)/2}\alpha_{\chi^{-1}}(1-u_0)\dots(1-u_{n+1})\in E[G].\]
 The coefficient of $1$ in $\alpha_{\chi^{-1}}$ receives a contribution from all terms of the form $u_i^j$ for all $i=1,\dots,n+1$, $j=0,d_i,\dots,(q_i-1)d_i$, as well as $(u_1\dots u_{n+1})^j$, $j=0,d_0,\dots,(q_0-1)d_0$, hence is equal to $q^*/d^{n+1}$.
 Thus, using the eigenvector property of $\alpha_{\chi^{-1}}$, the previous expression evaluates to
 \begin{align*}
 &\frac{(-1)^{n(n+1)/2}q^*}{d^{n+1}} \prod_{i=0}^{n+1}(1-{\zeta_d}^{-a_i})
 =  \frac{(-1)^{n(n+1)/2}q^*}{d^{n+1}} (1-\zeta_d^{a_1+\dots+a_{n+1}})\prod_{i=1}^{n+1}(1-{\zeta_d}^{-a_i})\\
 = & \frac{(-1)^{n(n+1)/2}q^*}{d^{n+1}} \left(\prod_{i=1}^{n+1}(1-{\zeta_d}^{-a_i})+(-1)^n\prod_{i=1}^{n+1}(1-{\zeta_d}^{a_i})\right). \qedhere
 \end{align*}
\end{proof}

In \cite{griffiths}, Griffiths describes the primitive Hodge structure of a smooth hypersurface inside a projective variety $X$. The following convenient formulation is due to Voisin \cite[Thm.~6.1]{voisinII} (see also Proposition~6.2 ibid.\ and \cite[Thm.~4.3.2]{dolgachev} which treats the special case $X=\BP$) apart from the missing twist by $n'$, which simply shifts the Hodge weights. Let \[\Omega=x_0\dots x_n\sum_{i=0}^{n+1}dx_0/x_0\wedge\dots\wedge \wh{dx_i/x_i}\wedge\dots\wedge dx_{n+1}/x_{n+1}\] be a generator of $\RH^0(\BP_\BC,K_\BP(\sum_{i=0}^{n+1} q_i))$ (this is one-dimensional as $K_\BP=\CO_\BP(-\sum_{i=0}^{n+1}q_i)$ \cite[2.1.5]{dolgachev}).

\begin{theorem}\label{griffiths}
 There is a surjective residue map
 \[\BC[x_0,\dots,x_n]^{\deg=(q+1)d-\sum_{i=0}^{n+1}q_i}\cong\RH^0(\BP_\BC,K_\BP((q+1)d))\to P^{n-q-n',q-n'}\]
 sending $\eta$ to the residue $\Res_F\eta\Omega/f^{q+1}$. The kernel of this map is the Jacobian ideal $J_f=\langle \del f/\del x_i: i=0,\dots,n+1\rangle$.
\end{theorem}

Concretely, $J_f=\langle x_i^{d_i-1}:i=0,\dots,n+1\rangle$. It follows, absorbing the product $x_0\dots x_n$ from $\Omega$ into $\eta$, that a basis of $P^{n-q,q}$ is given by the differential forms
\[\omega_{a'_0,\dots,a'_{n+1}}=\Res_F \prod_{i=0}^{n+1}x_i^{a'_i}\frac{\sum_{i=0}^{n+1}dx_0/x_0\wedge\dots\wedge \wh{dx_i/x_i}\wedge\dots\wedge dx_{n+1}/x_{n+1}}{f^{q+1}}\]
where the $a'_i$ run through $\{1,\dots,d_i-1\}$ such that \[\sum_{i=0}^{n+1}(a'_i-1)q_i=(q+1+n')d-\sum_{i=0}^{n+1}q_i.\] Setting $a_i=a'_iq_i$, this is equivalent to \[q((a_0,\dots,a_{n+1}))=\frac{\sum_{i=0}^{n+1}a_i}{d}-1-n'=q.\]

\begin{lemma}
 The group $G$ acts on $\langle\omega_{a'_0,\dots,a'_{n+1}}\rangle$ via $\chi=(a_0,\dots,a_{n+1})$.
\end{lemma}
\begin{proof}
 The element $u_i$ acts on $x_i^{a'_i}$ by multiplication with $\zeta_{d_i}^{a'_i}=\zeta_d^{a_i}=\chi(u_i)$ and trivially on the other terms of $\omega_{a'_0,\dots,a'_{n+1}}$.
\end{proof}

This finishes the proof of \Cref{thm:complete-cohomology}(4) since we have identified $P^{n-q,q}$ as the sum of those eigenspaces under the $G$-action for which $q(\chi)=q$.

\subsection{Galois action and comparison with $l$-adic cohomology -- Proof of (5)}\label{sub:gal}
Let $\ov F=F\times_k \ov K$ be the base change to an algebraic closure $\ov K$ of $K$. The action of the absolute Galois group $\Gamma=\Gal(\ov K/K)$ on $\Het^n(\ov F,\BZ_\ell(n'))$ preserves the hyperplane class, so $\Gamma$ acts on the primitive $\ell$-adic cohomology
\[P_\ell\coloneqq P\otimes_\BZ \BZ_\ell\simeq \RP^n_\mathrm{\acute{e}t}(\ov F,\BZ_\ell(n')).\]
From \Cref{prop:decomp}, we have that
\[P_\ell\otimes_{\BZ_\ell} E_\lambda=\bigoplus_{\chi\in S} V_\chi\otimes_E E_\lambda\]
and because the action of $G$ commutes with the action of ${\Gamma}$, this decomposition is preserved by ${\Gamma}$.

By the Chebotarev density theorem, to determine the action of ${\Gamma}$ on $P_\ell$, it suffices to determine the action of $\Frob_\mathfrak p\in{\Gamma}$  for all prime ideals $\mathfrak p$ of $K$ such that $\mathfrak p\nmid d\ell$.

The following lemma reduces this task to the ``untwisted'' hypersurface $\wt F$ with $\bm A=(1,\dots,1)$.
\begin{lemma}\label{twist}
 Let $\chi=(a_1,\dots,a_{n+1})\in S$. Let $h(\chi)$ be the eigenvalue by which $\Frob_\mathfrak p$ acts on $V_\chi\otimes_E E_\lambda$ for the hypersurface $\wt F$. Then the eigenvalue of $\Frob_\mathfrak p$ on $V_\chi\otimes_E E_\lambda\subset P_\ell$ is given by ${h(\chi)}/{\prod_{i=1}^{n+1}\psi(A_i)^{a_i}}$.
\end{lemma}
\begin{proof}
    The hypersurface $F$ with coefficients $(A_1,\dots,A_{n+1})$ is obtained from $\wt F$ by twisting (in the sense of \cite[Thm.~4.5.2]{Poo17}) with the $1$-cocycle which is the image of $(A_1,\dots,A_{n+1})$ under the composition of the natural maps
\[(k^\times)^{n+1}\to\prod_{i=1}^{n+1}(k^\times/{k^\times}^{d_i})\simeq\RH^1(\Gal(\ov K/k),G)\to\RH^1(\Gal(\ov K/k),\Aut_{\ov k}(\ov F)).\]
 The induced Galois representation on cohomology has to be twisted by the same cocycle restricted to $\Gamma$.
\end{proof}
We can thus restrict to $F=\wt F$ for the rest of \Cref{sub:gal}.

Fix a primitive $p$-th root of unity $\zeta$.
\begin{definition}
 Let $r\in\BZ/d$. The \emph{Gauss sum} $g_\pp(r)\in\BQ({\zeta_d},\zeta)$ is the element
 \[g_\pp(r)=\sum_{x\in\BF_\mathfrak p^\times}\psi(x)^r\zeta^{\Tr_{\BF_\mathfrak p / \BF_p}(x)}.\]
\end{definition} 

\begin{lemma}
 Let $\chi\in S$ correspond to $(a_0,\dots,a_{n+1})$. Then
 \begin{eqnarray*}
  J_\mathfrak{p}(\chi)=\frac{g_\pp(a_1)\dots g_\pp(a_{n+1})}{g_\pp(a_1+\dots+a_{n+1})}=\Norm(\mathfrak p)^{-1}\psi^{a_0}(-1)g_\pp(a_0)\dots g_\pp(a_{n+1}).
 \end{eqnarray*}
\end{lemma}
\begin{proof} The equalities follow from \cite[Ch.~8, Thm.~3 and Cor.~1]{ireland}.
\end{proof}

In \cite{weil}, A.\ Weil essentially computed the eigenvalues of $\Frob_\mathfrak{p}$ acting on $P_\ell$.
We have to match these to the known eigenspace decomposition under the action of $G$. In the classical projective Fermat case, this was done by {D.~Ulmer} \cite[7.6]{ulmer} but the statement goes back to {Shioda}.

It is however possible to give a short and simple proof using the Fourier transform on $G$. The inspiration comes from the equivariant Lefschetz trace formula by {Deligne} and {Lusztig} \cite[p.\ 119]{dl}.

\begin{proposition}\label{match}
 Let $\lambda$ be a prime of $E=\BQ(\mu_d)$ lying above $\ell$. Let $\mathfrak p$ be a prime of $K$ not dividing $d\ell$. Then for all $\chi\in S$, the action of $\Frob_\mathfrak p$ on $V_\chi\otimes_E E_\lambda\subset P_\ell\otimes_E E_\lambda$ (of $\wt F$) is multiplication by
 \[(-1)^n\psi^{a_0}(-1)\Norm(\mathfrak p)^{-n'}J_\mathfrak{p}(\chi).\]
\end{proposition}
\begin{proof}
For ease of notation, we write $q=\Norm(\mathfrak p)$. We define two functions $h_1,h_2:\wh G\to E_\lambda$ and show that their Fourier transforms agree.  If $\chi\in \wh G\setminus S$, set $h_1(\chi)=h_2(\chi)=0$. Otherwise,
 let $h_1(\chi)=(-1)^n\psi^{a_0}(-1)q^{-n'}J_\mathfrak{p}(\chi)$ and
  let $h_2(\chi)$ be the eigenvalue by which $\Frob_\mathfrak{p}$ acts on $V_\chi\otimes E_\lambda$. For arbitrary $c=(c_1,\dots,c_{n+1})\in G$, we choose preimages $\wt c_i\in \BF_q^\times$ under the multiplicative character $\psi$.
  
  For $*=1,2$, denote the Fourier transform of $h_{*}$ by
  \[\wh h_{*}(c)= \sum_{\chi\in\wh G}h_{*}(\chi)/(\psi(\wt c_1)^{a_1}\dots\psi(\wt c_{n+1})^{a_{n+1}})=\sum_{\chi\in\wh G}\chi^{-1}(c)h_{*}(\chi).\]
 
 Consider the twisted diagonal hypersurface
 \[F':\quad x_0^{d_0}+\wt c_1x_1^{d_1}\dots+\wt c_{n+1}x_{n+1}^{d_{n+1}}=0\subset \BP_{\BF_q}.\]
 By \Cref{twist} the trace of $\Frob_\mathfrak p$ on $\Het^n(F'\times_{\BF_q} \ov {\BF_q}, \BZ_\ell)$  equals $q^{n'}(1+\wh h_{2}(c))$ if $n$ is even, and $q^{n'}\wh h_{2}(c)$ if $n$ is odd.
 
 The Lefschetz trace formula gives
 \[\#F'(\BF_q) = 1+q+\dots+q^n+(-1)^n q^{n'}\wh h_2(c).\]
 According to Weil's work on exponential sums (as presented in \cite[Chapter 8, Theorem~5 and Proposition 8.5.1(c)]{ireland}), one also has
 \begin{align*}
\#F'(\BF_q) =& \frac{q^{n+1}-1}{q-1}+\sum_{\chi\in S}\psi^{a_0}(-1)J_\mathfrak{p}(\chi)/(\psi(\wt c_1)^{a_1}\dots\psi(\wt c_{n+1})^{a_{n+1}})\\
=& 1+q+\dots+q^n+(-1)^nq^{n'}\wh h_1(c).  
 \end{align*}
 It follows that $\wh h_1=\wh h_2$ and for their inverse Fourier transforms,
 \[h_1(\chi)=\frac{1}{\# G}\sum_{c\in G} \chi(c)\wh h_{1/2}(c)=h_2(\chi). \qedhere\]
\end{proof}

The combination of \Cref{twist} and \Cref{match} finishes the proof of \Cref{thm:complete-cohomology}(5).

The explicit determination of Gauss and Jacobi sums including their sign is in general a difficult subject. The case $d=4$ was treated by Pinch and Swinnerton-Dyer in \cite[\S 3]{pinch-sd} and \S \ref{sextic} will treat the case $d=6$. The following property of Gauss sums will be helpful.
\begin{lemma}\label{gaussproduct}
 We have $g_\pp(r)g_\pp(-r)=\psi(-1)^r\Norm(\mathfrak p)$.
\end{lemma}
\begin{proof}
 See for example \cite[Exercise 10.22(d)]{ireland}.
\end{proof}

\subsection{Action of complex conjugation -- Proof of (6)}
\begin{remark}
 Before giving the proof, we would like to clarify the relation between three different ``complex conjugations'' on $\RH^n(F_\BC,\BC)$ as described in \cite{deligne}. The ``complex conjugation'' $\tau$ (or $F_\infty$ in the notation of \cite[0.2.5]{deligne}) is induced by the involution on the points $F(\BC)$. The second ``complex conjugation'' is induced by the action of complex conjugation on the coefficients. Each of these actions swaps the Hodge spaces $\RH^{p,q}(F_\BC)$ and $\RH^{q,p}(F_\BC)$ (see \cite[I.2.4]{silhol} for the latter) and their composition is the ``complex conjugation'' induced by the comparison isomorphism $\RH^n_{\mathrm{dR}}(F_\BC)\otimes_\BR\BC\cong \RH^n(F_\BC,\BC)$, which hence preserves the Hodge spaces $\RH^{p,q}(F_\BC)$ \cite[Prop.~1.4, Cor.~1.6]{deligne}.
\end{remark}

We remark that $\tau g=g^{-1}\tau$ for all $g\in G$. In particular, in the notation of \Cref{primitive}, it suffices to compute $\tau((2i\pi)^{n'}e)$. From the explicit form of $\bm e$, we find that
\[\tau(\bm e)=u_{\bm A}(u_0\dots u_{n})^{-1} \bm e.\]
It follows that $\tau(e)=(-1)^{n+1}u_{\bm A} e$ and $\tau((2i\pi)^{n'}e)=(-1)^{n'+n+1}u_{\bm A} (2i\pi)^{n'}e$. 

The Poincaré duality isomorphism $\RH_n(F_\BC,\BZ)\cong\RH^n(F_\BC,\BZ)$ is induced by the cap product with the fundamental class $[F]\in \RH_{2n}(F_\BC,\BZ)$.
Since the action of $G$ preserves $[F]$ and $\tau$ sends $[F]$ to $(-1)^n[F]$, this is an isomorphism of $\BZ[G]$-modules which (anti)commutes with $\tau$. It follows that the action of $\tau$ on a $\BZ[G]$-generator of cohomology differs by the action of $\tau$ on the generator $e$ of homology by $(-1)^n$. This finishes the proof of \Cref{thm:complete-cohomology}(6).

\begin{remark}\label{rem:cyclotomic}
Assume $k=\BQ$. For general $d$, we do not know how to account for the action of Galois automorphisms whose restriction to $E$ is neither trivial nor complex conjugation on $E$. One immediate obstacle is that the short exact sequence
 \[0\to\Gal(\ov E/E)\to\Gal(\ov \BQ/\BQ)\to\Gal(E/\BQ)\to 0\]
 does not split when $\deg(E/\BQ)>2$.
\end{remark}

\section{Transcendental Brauer group of diagonal degree 2 K3 surfaces}
\label{sextic}
We now use the results from the \Cref{sec:cohomology} to study the transcendental Brauer group of diagonal degree $2$ K3 surfaces. The main results are \Cref{invariants,invariantsQ} for the Galois invariants of the geometric Brauer group and \Cref{prop:3transcendental,prop:2transcendental} for the transcendental part.

In the notation of \Cref{sec:cohomology} we  restrict to $n=2$, $d=6$ and ${\bm q}=(3,1,1,1)$. Recall that $\omega$ denotes a primitive third root of unity. We work over the Eisenstein numbers \[k=E=\BQ(\zeta_6)=\BQ(\omega)=\BQ(\sqrt{-3})\] and write $\CO=\CO_E$ for their ring of integers. Note that $\CO$ is a principal ideal domain.
Explicitly, we consider the surface
\[X=X_{\bm A}:x_0^2-A_1x_1^6-A_2x_2^6-A_3x_3^6=0\]
in $\BP_\BQ^{n+1}(3,1,1,1)$. (The notation $X_\AA$ differs from the notation for the coefficients used in \Cref{thm:complete-cohomology} by multiplying $A_1$, $A_2$ and $A_3$ with $-1$.)

It can be checked explicitly (e.g.\ with \texttt{Magma}) from \Cref{thm:complete-cohomology}(1) that \[\RP^n((X_\AA)_\BC,\BZ)=\BZ[u_1,u_2,u_3]/\langle1+(u_1u_2u_3),1+u_i+\dots+u_i^5,i=1,\dots,3\rangle.\]
The set $S$ from \Cref{thm:complete-cohomology}(4) has cardinality $21$ and contains the elements $(1,1,1)$, $(5,5,5)$, $(3,3,3)$, three permutations of $(2,2,5)$, three permutations of $(4,4,1)$, six permutations of $(1,3,5)$ and six permutations of $(2,3,4)$ (where we are writing elements of $S\subset\wh G$ in the asymmetric notation that omits $a_0=3$).

\subsection{Transcendental lattice}\label{sub:T}
\begin{definition}\label{def:T}
 The transcendental lattice $T(X_\BC)$ is the smallest primitive sublattice of $P$ such that $P^{-1,1}\subset T(X_\BC)\otimes_\BZ \BC$.
\end{definition}
The group $\Gal(k/\BQ)=(\BZ/6\BZ)^\times$ acts on $P\otimes k$ via the second factor so that an element $t\in (\BZ/6\BZ)^\times$ sends $\alpha_\chi$ to $\alpha_{\chi^t}$.
\begin{definition}\label{def:S_T}
 Set $S_T\subset S$ to be the subset of those characters $\chi\in S$ whose $\Gal(k/\BQ)$-orbit contains a $\chi'$ with $q(\chi')=1$.
\end{definition}
\begin{lemma}\label{T}
 Let $V_\chi$ be as in \Cref{prop:decomp}. We have \[T(X_\BC)=P\cap\bigoplus_{\chi\in S_T} V_\chi.\]
\end{lemma}
\begin{proof}
 By \Cref{def:T} and \Cref{thm:complete-cohomology}(4), $V_{\chi'}\subset T(X_\BC)\otimes k$ for all $\chi'$ with $q(\chi')=1$. Hence,  $\bigoplus_{\chi\in S_T}V_\chi$ is the smallest $\Gal(k/\BQ)$-stable $k$-subspace $U$ of $T(X_\BC)\otimes k$ such that $P^{-1,1}\subset U\otimes_k \BC$. There is a one-to-one correspondence between primitive sublattices $W\subset P$ and $\Gal(k/\BQ)$-stable $k$-subspaces $U\subset P\otimes k$, given by $W\mapsto W\otimes k$ and $U\mapsto U\cap P$. The claim follows. 
\end{proof}

Explicating \Cref{def:S_T}, the set $S_T$ equals $\{(1,1,1),(5,5,5)\}$. By \Cref{T}, it follows that
\[T(X_\BC)=P\cap (V_{(1,1,1)}\oplus V_{(5,5,5)}).\]

\begin{lemma}\label{Tsextic}
 We have $T(X_\BC)=\BZ w_1\oplus\BZ w_2$, where
 \[w_1=12\sqrt{-3}({\zeta_6}\alpha_{(1,1,1)}+{\zeta_6}^2\alpha_{(5,5,5)}),\quad w_2=12\sqrt{-3}({\zeta_6}^2\alpha_{(1,1,1)}+{\zeta_6}\alpha_{(5,5,5)}).\]
 Furthermore,
 \[\langle w_1,w_1\rangle=\langle w_2,w_2\rangle=24,\quad\langle w_1,w_2\rangle=\langle w_2,w_1\rangle=12.\]
\end{lemma}
\begin{proof}
 Clearly, $T(X_\BC)\otimes \BQ=\BQ w_1\oplus\BQ w_2$.
  We calculate from \Cref{idemcup}
 \begin{eqnarray*}
\langle \alpha_{(1,1,1)},u_1^\alpha u_2^\beta u_3^\gamma\rangle&=&\langle u_1^{-\alpha} u_2^{-\beta} u_3^{-\gamma}\alpha_{(1,1,1)},1\rangle
=\langle {\zeta_6}^{-(\alpha+\beta+\gamma)}\alpha_{(1,1,1)},1\rangle=\frac{1}{36}{\zeta_6}^{-(\alpha+\beta+\gamma)}
 \end{eqnarray*}
 and similarly
 $\langle \alpha_{(5,5,5)},u_1^\alpha u_2^\beta u_3^\gamma\rangle=\frac{1}{36}{\zeta_6}^{\alpha+\beta+\gamma}$.
 Thus, \begin{eqnarray*}
        \langle w_1,u_1^\alpha u_2^\beta u_3^\gamma\rangle&=&\frac{12\sqrt{-3}}{36}({\zeta_6}^{-(\alpha+\beta+\gamma-1)}+{\zeta_6}^{\alpha+\beta+\gamma+2})
        =\frac{-2}{\sqrt{3}}\mathrm{Im}({\zeta_6}^{-(\alpha+\beta+\gamma-1)})
       \end{eqnarray*}
 and similarly $\langle w_2,u_1^\alpha u_2^\beta u_3^\gamma\rangle=\frac{-2}{\sqrt{3}}\mathrm{Im}({\zeta_6}^{-(\alpha+\beta+\gamma-2)})$.
 Therefore, if $sw_1+tw_2\in H^*$ for some $s,t\in\BQ$, then $s,t\in\BZ$. However, $H$ was the direct sum of $P$ and an algebraic class ${\pi_{\bm q}}_*\Lambda$. Because $\langle w_i,P\rangle=\BZ$ and $\langle w_i,{\pi_{\bm q}}_*\Lambda\rangle=0$ for $i=1,2$, we get that $\langle w_i,\cdot\rangle\in H^*$. By unimodularity of $H$, this means $w_i\in H$, hence
 \[\BZ w_1\oplus\BZ w_2 = H\cap (T(X_\BC)\otimes\BQ)=T(X_\BC).\]
 The formula for the cup product follows directly from \Cref{idemcup}.
\end{proof}

The group $\mu_6$ acts on $T(X_\BC)$ and $T(X_\BC)\otimes E$ via multiplication by $u_1$, $u_2$, or $u_3$ and because $(1,1,1)$ and $(5,5,5)$ are invariant under permutations, it does not matter which of the three variables we pick. Denote the action of $x\in\mu_6$ by $[x]$. We have that
$[x]\alpha_{(1,1,1)}=x\alpha_{(1,1,1)}$ and $[x]\alpha_{(5,5,5)}=x^{-1}\alpha_{(5,5,5)}$, hence
\[[{\zeta_6}]w_1=w_2,\quad [{\zeta_6}]w_2=w_2-w_1.\]

The free $\CO$-module $T(X_\BC)$ is thus (non-uniquely) isomorphic to $\CO$ itself sending $w_1$ to $1$ and $w_2$ to ${\zeta_6}$. The intersection product under this identification is given by $\langle x,y\rangle=12\Tr_{E/\BQ}(x\ov y)$.
Because \[12\Tr_{E/\BQ}(1/(12\sqrt{-3}))=0\text{ and }12\Tr_{E/\BQ}({\zeta_6}/(12\sqrt{-3}))=-1,\] it follows that the dual lattice of $\CO$ is $\frac{1}{12\sqrt{-3}}\CO$.

\begin{definition}
 The \emph{discriminant} of $X$ is the cokernel of the exact sequence
\begin{equation}\label{eq:disc}
0\to T(X_\BC)\to T(X_\BC)^*\to\Delta\to 0, 
\end{equation}
where the second map is the natural embedding using the cup product.
\end{definition}
The above sequence~\ref{eq:disc} becomes
\[0\to \CO\xrightarrow{12\sqrt{-3}}\CO\to\CO/12\sqrt{-3}\to0\]
We recover the fact that $\Delta=\CO/12\sqrt{-3}=\BZ/12\BZ\times\BZ/36\BZ$, as shown in \cite[\S 2.1]{cn17} with explicit divisors.

\subsection{Galois action in the untwisted case}\label{sub:GaloisT}
For \Cref{sub:GaloisT} only, we restrict to the case $A_1=A_2=A_3=-1$. Let $\ell$ be a prime number and let $\lambda$ be a place of $k$ lying above $\ell$. Let $\pi\in \CO$ be a prime element not dividing $6\ell$. The multiplicative character $\psi$ from \Cref{thm:complete-cohomology}(5) becomes the sextic residue character $(\cdot/\pi)_6$. By multiplying with sixth powers, this character can be extended to elements $x\in k$ such that $x\CO$ is coprime to $\pi$.

We will require a very particular notion of primary generators of prime ideals due to Eisenstein \cite[7.3]{lemmermeyer}. This notion is a strengthening of the cubic notion of primary primes so that we can apply sextic reciprocity. 

\begin{definition}
 We call $x=a+b\omega\in\CO$ \emph{E-primary}, if $3\mid b$ and
 \[\begin{cases}
    a+b\equiv 1 \bmod 4, &\text{if } 2\mid b\\
    b\equiv 1 \bmod 4, &\text{if } 2\mid a\\
    a\equiv 3 \bmod 4, &\text{if } 2\nmid ab.\\
   \end{cases}
\]
\end{definition}
Every prime ideal in $\CO$ not dividing $6$ has an E-primary generator. Furthermore, if $x\in\CO$ is an $E$-primary prime element, then $x$ is primary in the usual cubic sense (i.e.\ $\equiv \pm1\bmod 3$). Henceforth, we will assume that $\pi$ is E-primary.

\begin{theorem}\label{sexticreciprocity}
 Let $x,y\in\CO$ be E-primary and relatively prime. Then
 \[\left(\frac{x}{y}\right)_6=(-1)^{\frac{\Norm x-1}{2}\frac{\Norm y-1}{2}}\left(\frac{y}{x}\right)_6.\]
\end{theorem}
\begin{proof}
  This is a combination of cubic reciprocity and a quadratic reciprocity law for $\CO$, see \cite[Thm.~7.10]{lemmermeyer}.
\end{proof}

\begin{proposition}\label{jacobi}
 Let $\mathfrak p\subset\CO$ be a prime ideal generated by an E-primary element $\pi$. Let $\lambda\subset\CO$ be a prime ideal over the rational prime $\ell$. Assume $\mathfrak p$ does not divide $6\ell$.
 
 Set $\zeta_\pi=(-4/\pi)_6$. Then for $\chi\in S$, the eigenvalue $\mu$ of $\Frob_\mathfrak{p}\in\Gamma$ on $V_\chi\otimes_k k_\lambda$ is as follows:\\
   If $\chi$ is $(3,3,3)$ or a permutation of $(1,3,5)$, then $\mu=1$.\\
   If $\chi$ is a permutation of $(2,3,4)$, then $\mu=\zeta_\pi^3$.\\
   If $\chi$ is a permutation of $(2,2,5)$, then $\mu=\zeta_\pi$.\\
   If $\chi$ is a permutation of $(4,4,1)$, then $\mu=\ov{\zeta_\pi}$.\\
   If $\chi=(1,1,1)$, then $\mu=\pi/\ov\pi$.\\
   If $\chi=(5,5,5)$, then $\mu=\ov\pi/\pi$.
\end{proposition}
\begin{proof}
  We treat each item individually via \Cref{match} which states that $\mu$ is given by $\Norm(\mathfrak p)^{-1}\left(\frac{-1}{\pi}\right)_6 J_\mathfrak p(\chi)=\Norm(\mathfrak p)^{-2}\prod_{i=0}^3g_\pp(a_i)$.
  
 If $\chi$ is $(3,3,3)$ or a permutation of $(1,3,5)$, we find by \Cref{gaussproduct} that
  \[g_\pp(3)^2=\left(\frac{-1}{\pi}\right)_6^3 \Norm(\mathfrak p)=\left(\frac{-1}{\pi}\right)_6 \Norm(\mathfrak p)=g_\pp(1)g_\pp(5)\]
  from which it follows that $g_\pp(3)^4=g_\pp(3)^2g_\pp(1)g_\pp(5)=\Norm(\mathfrak p)^2$.
  If $\chi$ is a permutation of $(2,3,4)$, we find by \Cref{gaussproduct} that \[g_\pp(2)g_\pp(4)=(-1/\pi)_6^2 \Norm(\mathfrak p)=\Norm(\mathfrak p)\] from which it follows that \[\Norm(\mathfrak p)^{-2}g_\pp(3)^2g_\pp(2)g_\pp(4)=\left(\frac{-1}{\pi}\right)_6^3=\left(\frac{(-4)^3}{\pi}\right)_6=\left(\frac{-4}{\pi}\right)_6^3.\]
  If $\chi$ is a permutation of $(2,2,5)$, we find by \cite[(2.1.2)]{evans} (see also  \cite[Thm.~3.1.1]{evans}) that 
  \begin{equation}\label{eq:evans}
   \frac{g_\pp(1)g_\pp(2)}{g_\pp(3)}=\left(\frac{4^2}{\pi}\right)_6\frac{g_\pp(2)^2}{g_\pp(4)}=\left(\frac{-4}{\pi}\right)_6\frac{g_\pp(1)^2}{g_\pp(2)}.
  \end{equation}

  Now $\Norm(\mathfrak p)^2\left(\frac{-1}{\pi}\right)_6(g_\pp(4)g_\pp(5))^{-1}=g_\pp(1)g_\pp(2)$, hence
  \[\Norm(\mathfrak p)^2\left(\frac{-4}{\pi}\right)_6=g_\pp(3)g_\pp(4)g_\pp(5)\frac{g_\pp(2)^2}{g_\pp(4)}=g_\pp(2)^2g_\pp(5)g_\pp(3).\]
  If $\chi$ is a permutation of $(4,4,1)$, this is the conjugate case to $(2,2,5)$.
  
  If $\chi=(1,1,1)$, we find by applying \Cref{eq:evans}
  \[J_\pp(\chi)=\frac{g_\pp(1)g_\pp(1)}{g_\pp(2)}\frac{g_\pp(1)g_\pp(2)}{g_\pp(3)}=\left(\frac{-1}{\pi}\right)_6\left(\frac{g_\pp(2)g_\pp(2)}{g_\pp(4)}\right)^2.\]
  But by \cite[\S9.4 Lemma~1]{ireland},   \[\frac{g_\pp(2)g_\pp(2)}{g_\pp(4)}=\pm \pi.\]
  Hence, $\Norm(\mathfrak p)^{-1}\left(\frac{-1}{\pi}\right)_6 J_\mathfrak p(\chi)=\Norm(\mathfrak p)^{-1}\pi^2=\pi/\ov\pi$.
  
  If $\chi=(5,5,5)$, this is the conjugate case to $(1,1,1)$.
\end{proof}

\begin{corollary}\label{Taction}
 Let $\mathfrak p$ and $\pi$ be as in \Cref{jacobi}. The element $\Frob_\mathfrak{p}\in\Gamma$ acts on $T(X_\BC)\otimes\BZ_\ell$ as multiplication by $\pi/\ov\pi$. Complex conjugation (the generator of $\Gal(\ov\BQ/\ov\BQ\cap\BR)$) acts on $T(X_\BC)\otimes\BZ_\ell$ by swapping $w_1$ and $w_2$.
\end{corollary}
\begin{proof}
 In \Cref{Tsextic}, $T(F_\BC)=\BZ w_1\oplus\BZ w_2$ was identified with $\CO$ as an $\CO$-module such that $[{\zeta_6}]w_1=w_2$. By \Cref{jacobi}, $\Frob_\pi$ acts with eigenvalue $\pi/\ov \pi$ on $\alpha_{(1,1,1)}\CO_\lambda$ and with eigenvalue $\ov \pi/\pi$ on $\alpha_{(5,5,5)}\CO_\lambda$. Now $[\zeta_6]$ acts on $\alpha_{(1,1,1)}$ with eigenvalue $\zeta_6$ and on $\alpha_{(5,5,5)}$ with eigenvalue $\ov \zeta_6$. Therefore, the matrix representing the action of $\Frob_\pi$ in the basis $(w_1,w_2)$ is given by multiplication with $\pi/\ov\pi$.
 
 The complex conjugation $\tau$ swaps $\alpha_{(1,1,1)}$ and $\alpha_{(5,5,5)}$ by \Cref{thm:complete-cohomology}(6), hence also $w_1$ and $w_2$.
\end{proof}

\subsection{Galois invariant part of the Brauer group}
We now begin our study of transcendental Brauer groups of the surfaces $X=X_\AA$ from \eqref{eqn:X_A} with $A_1,A_2,A_3\in\BZ$.

Let $k=\BQ(\omega)$, $\Gamma=\Gal(\ov k/k)$ and $X_k=X\times_\BQ k$. We note that bounding $\Br(X_k)/\Br_1(X_k)$ is enough for our purposes of finding an upper bound on the transcendental part $\Br (X)/\Br_1(X)$ since the latter injects into the former.

The transcendental lattice is related to the Brauer group in the following way.
\begin{lemma}[{{\cite[Eq. (8.7) and (8.9)]{grothendieck}}}]\label{prop:grothendieck} We have an isomorphism of Galois modules
  \[\Br(\ov X)\cong T(X_\BC)^*\otimes \BQ/\BZ\cong \bigoplus_\ell T_\ell(\ov X)^*\otimes_{\BZ_\ell} \BQ_\ell/\BZ_\ell,\]
  where the sum runs over all primes $\ell$, the $\ell$-adic transcendental lattice $T_\ell(\ov X)$ is the orthogonal complement of $\Pic(\ov X)$ inside $\Het^2(\ov X,\BZ_\ell(1))$, and the Galois action on $T(X_\BC)^*\otimes \BQ/\BZ$ is carried over from $\ell$-adic cohomology via the second isomorphism.
\end{lemma}
Using \Cref{Taction}, we can therefore bound and compute the Galois invariant part of the geometric Brauer group of $X$. The following lemma will be applied with $M=T(X_\BC)^*\otimes \BQ/\BZ$ and $G$ the image of $\Gamma$ in $\Aut_\BZ(M)$.

\begin{lemma}\label{coinvariants}
 Let $G$ be a finite soluble group and let $M$ be a $\BZ[G]$-module. Then the natural morphism
 \[\Hom(M,\BQ/\BZ)_G\to \Hom(M^G,\BQ/\BZ),\quad [\phi]\mapsto\phi\restriction_{M^G}\]
 is an isomorphism from the coinvariants of the $\BQ/\BZ$-dual of $M$ to the $\BQ/\BZ$-dual of the invariants of $M$.
\end{lemma}
\begin{proof}
 It suffices to prove the claim for cyclic $G=\langle g\rangle$. Dualising the short exact sequence
 \[0\to M^G\to M\to (g-1)M \to 0\]
 yields, together with $\Ext^1((g-1)M,\BQ/\BZ)=0$, the short exact sequence
 \[0\to (g-1)\Hom(M,\BQ/\BZ) \to\Hom(M,\BQ/\BZ)\to\Hom(M^G,\BQ/\BZ)\to 0.\]
 But the cokernel of the above sequence is by definition $\Hom(M,\BQ/\BZ)_G$.
\end{proof}

\begin{proposition}\label{exponents}
 The exponent of $\Br(\ov X)[\ell^\infty]^\Gamma$ is at most $4$ if $\ell=2$, at most $9$ if $\ell=3$, at most $\ell$ if $\ell\in\{5,7\}$, and $1$ if $\ell\geq 11$. More precisely, $\Br(\ov X)[3^\infty]^\Gamma$ is a quotient of $\CO/3\sqrt{-3}\cong\BZ/3\BZ\times\BZ/9\BZ$ as an abelian group.
\end{proposition}
\begin{proof}
 For $\pi\nmid\ell$, we know from \Cref{twist} and \Cref{Taction} that $\Frob_\pi$ acts on $T(X_\BC)\otimes\BZ_\ell$ as multiplication by $x\pi/\ov \pi$ where $x\in\mu_6$.
 
 Set $\pi=3{\zeta_6}-1$ so that $\Norm(\pi)=7$. Then a calculation shows that the set of maximal prime powers that divide $(x\pi-\ov\pi)$ are $\{4,3\sqrt{-3},5\}$. Doing the same for $\pi=-3{\zeta_6}+4$, so that $\Norm(\pi)=13$, yields $\{4,3\sqrt{-3},5,7\}$. The result follows from \Cref{coinvariants}.
\end{proof}
One may compare \Cref{exponents} with the bounds obtained in \cite[Example 11.2]{valloni}.

\begin{proposition}\label{invariants}
 The group $\Br(\ov X)[\ell^\infty]^\Gamma$ equals
 \begin{itemize}
  \item $\ell=2:\begin{cases}
                \CO/4, & \text{if } A_1A_2A_3/16\in {k^\times}^6,\\
                \CO/2, & \text{if } A_1A_2A_3/16\in {k^\times}^3\setminus{k^\times}^6,\\
                0, & \text{otherwise.}
               \end{cases}$
  \item $\ell=3:\begin{cases}
                \CO/3\sqrt{-3}, & \text{if } -A_1A_2A_3\in {k^\times}^6,\\
                \BZ/3\BZ, & \text{if } -A_1A_2A_3\in {k^\times}^2\setminus{k^\times}^6,\\
                0, & \text{otherwise.}
               \end{cases}$
  \item $\ell=5:\begin{cases}
                \CO/5, & \text{if } -5A_1A_2A_3\in {k^\times}^6,\\
                0, & \text{otherwise.}
               \end{cases}$
  \item $\ell=7:\begin{cases}
                \CO/7, & \text{if } A_1A_2A_3/7\in {k^\times}^6,\\
                0, & \text{otherwise.}
               \end{cases}$
  \item $\ell>7:\quad 0.$
 \end{itemize}
\end{proposition}
\begin{proof}
 By \Cref{Taction} and \Cref{twist}, $\Frob_\pi$ acts on $T(X_\BC)\otimes\BQ/\BZ\cong\CO\otimes\BQ/\BZ$ as multiplication by $$\frac{\pi}{\ov\pi} \left(\frac{-1/A_1A_2A_3}{\pi}\right)_6.$$
 By \Cref{coinvariants}, it is enough to compute the coinvariants of $T(X_\BC)\otimes\BQ/\BZ$.
 
 We use the exponent bounds from \Cref{exponents} and the Sextic Reciprocity Theorem~\ref{sexticreciprocity} throughout (and Cubic Reciprocity for $\ell=2$) in order to express the action modulo a power of $\lambda|\ell$ as Dirichlet characters. A similar calculation in the case of diagonal quartics was first done in \cite[Lem.~4.2]{IS15}. 
 
  $\ell=2$: 
  We have
  \begin{eqnarray*}\left(\frac{-16}{\pi}\right)_6 &=& \left(\frac{-1}{\pi}\right)_2\left(\frac{16}{\pi}\right)_2\left(\frac{16}{\pi}\right)_3^{-1}=(-1)^{(\Norm(\pi)-1)/2}\left(\frac{2}{\pi}\right)_3^2\\
  &\equiv& (-1)^{(\Norm(\pi)-1)/2} \pi^2 \equiv \Norm(\pi)^{-1}\pi^2\equiv \pi/\ov\pi \bmod 4.
  \end{eqnarray*}
  So for E-primary $\pi$, $\Frob_\pi$ acts on $T(X_\BC)\otimes\frac{1}{4}\BZ/\BZ\cong\CO/4$ via $\left(\frac{16/A_1A_2A_3}{\pi}\right)_6$. Thus if $A_1A_2A_3/16$ is a sixth power, then $\left(\frac{16/A_1A_2A_3}{\pi}\right)_6\equiv 1\bmod 4$, so all of $\CO/4$ is invariant. If $A_1A_2A_3/16$ is a third but not a sixth power, then the sextic residue symbol assumes the value $-1$ for some $\pi$ by Chebotarev density, so the coinvariants are $\CO/2$. In all other cases, the sextic residue symbol assumes the value $\omega$ for some $\pi$. The coinvariants of $\CO/4$ under multiplication by $\omega$ are trivial.
  
  $\ell=3$: We have \[\pi/\ov\pi\equiv1\bmod 3\sqrt{-3}.\] So for E-primary $\pi$, $\Frob_\pi$ acts on \[T_3(\ov X)/(3\sqrt{-3}\cdot T_3(\ov X))\cong\CO/3\sqrt{-3}\] via $\left(\frac{-1/A_1A_2A_3}{\pi}\right)_6$. Thus if $-A_1A_2A_3$ is a sixth power, then $\left(\frac{-1/A_1A_2A_3}{\pi}\right)_6\equiv 1\bmod 3\sqrt{-3}$, so all of $\CO/3\sqrt{-3}$ is invariant. If $-A_1A_2A_3$ is a square but not a sixth power, then the sextic residue symbol assumes all values in $\mu_3$ infinitely often by Chebotarev density, so the coinvariants are $\CO/(1-\omega)$. In all other cases, the sextic residue symbol assumes the value $-1$ for some $\pi$. The coinvariants of $\CO/3\sqrt{-3}$ under multiplication by $-1$ are trivial.

  $\ell=5$: We have \[\left(\frac{5}{\pi}\right)_6^{-1}=\left(\frac{\pi}{5}\right)_6^{-1}\equiv \pi^{2}\equiv\pi/\ov\pi\bmod 5.\]
  So for E-primary $\pi$, $\Frob_\pi$ acts on $T(X_\BC)\otimes\frac{1}{5}\BZ/\BZ\cong\CO/5$ via $\left(\frac{-1/(5A_1A_2A_3)}{\pi}\right)_6$. Thus if $-5A_1A_2A_3$ is a sixth power, then $\left(\frac{-1/(5A_1A_2A_3)}{\pi}\right)_6\equiv 1\bmod 5$, so all of $\CO/5$ is invariant. In all other cases, the sextic residue symbol assumes a nontrivial value in $\mu_6$ for some $\pi$. The coinvariants of $\CO/5$ under multiplication by nontrivial $x\in\mu_6$ are trivial.

  $\ell=7$: We have $7=\theta\ov\theta$ where $\theta=1+2{\zeta_6}$. Furthermore $-7$ is primary and $\Norm(-7)-1=48$, hence
  \begin{eqnarray*}
   \left(\frac{-7}{\pi}\right)_6&=&\left(\frac{\pi}{-7}\right)_6=\left(\frac{\pi}{7}\right)_6\\
   &=&\left(\frac{\pi}{\theta}\right)_6\left(\frac{\pi}{\ov\theta}\right)_6=\left(\frac{\pi}{\theta}\right)_6\left(\frac{\ov\pi}{\theta}\right)_6^{-1}\equiv(\pi/\ov\pi)\bmod 7.
  \end{eqnarray*}
So for E-primary $\pi$, $\Frob_\pi$ acts on $T(X_\BC)\otimes\frac{1}{7}\BZ/\BZ\cong\CO/7$ via $\left(\frac{7/A_1A_2A_3}{\pi}\right)_6$. Thus if $A_1A_2A_3/7$ is a sixth power, then $\left(\frac{7/A_1A_2A_3}{\pi}\right)_6\equiv 1\bmod 7$, so all of $\CO/7$ is invariant. In all other cases, the sextic residue symbol assumes a nontrivial value in $\mu_6$ for some $\pi$. The coinvariants of $\CO/7$ under multiplication by nontrivial $x\in\mu_6$ are trivial.
\end{proof}

Incorporating the action of complex conjugation as in \Cref{Taction} is an easy exercise so that one gets:
\begin{proposition}\label{invariantsQ}
   The group $\Br(\ov X)[\ell^\infty]^{\Gal(\ov\BQ/\BQ)}$ equals
  \begin{itemize}
  \item $\ell=2:\begin{cases}
                \BZ/4\BZ, & \text{if } A_1A_2A_3/16\in {\BQ^\times}^6\cup (-27){\BQ^\times}^6,\\
                \BZ/2\BZ, & \text{if } A_1A_2A_3/16\in {\BQ^\times}^3\setminus({\BQ^\times}^6\cup (-27){\BQ^\times}^6),\\
                0, & \text{otherwise.}
               \end{cases}$
  \item $\ell=3:\begin{cases}
                \BZ/9\BZ, & \text{if } -A_1A_2A_3\in (-27){\BQ^\times}^6,\\
                \BZ/3\BZ, & \text{if } -A_1A_2A_3\in {\BQ^\times}^6\cup(-3){\BQ^\times}^2\setminus(-27){\BQ^\times}^6,\\
                0, & \text{otherwise.}
               \end{cases}$
  \item $\ell=5:\begin{cases}
                \BZ/5\BZ, & \text{if } -5A_1A_2A_3\in {\BQ^\times}^6\cup(-27){\BQ^\times}^6,\\
                0, & \text{otherwise.}
               \end{cases}$
  \item $\ell=7:\begin{cases}
                \BZ/7\BZ, & \text{if } A_1A_2A_3/7\in {\BQ^\times}^6\cup(-27){\BQ^\times}^6,\\
                0, & \text{otherwise.}
               \end{cases}$
  \item $\ell>7:\quad 0.$
 \end{itemize}
\end{proposition}
\begin{proof}
 Note that $\BQ^\times\cap {k^\times}^2={\BQ^\times}^2\cup(-3){\BQ^\times}^2$ and $\BQ^\times\cap {k^\times}^6={\BQ^\times}^6 \cup (-27){\BQ^\times}^6$.
 If $A_1A_2A_3<0$, then complex conjugation on $T(X_\BC)$ swaps $w_1$ and $w_2$. If $A_1A_2A_3>0$, complex conjugation on $T(X_\BC)$ sends $w_1$ to $-w_2$ and $w_2$ to $-w_1$. (This description depends upon the choice of sixth roots of $-A_1,-A_2,-A_3$ in the generator $e$ from \Cref{primitive} but other choices for $e$ yield the same results.)
 
 Computing the coinvariants under complex conjugation of the groups $\Br(\ov X)[\ell^\infty]^\Gamma$ appearing in \Cref{invariants} is then an easy calculation. For example, if $A_1A_2A_3<0$, the coinvariants of $\CO/3\sqrt{-3}$ are equal to its quotient as an abelian group by $1-\zeta_6$, which is isomorphic to $\BZ/3\BZ$. If $A_1A_2A_3>0$, the coinvariants of $\CO/3\sqrt{-3}$ are equal to its quotient as an abelian group by $1+\zeta_6$, which is isomorphic to $\BZ/9\BZ$. The other cases follow similarly.
\end{proof}

\subsection{Image of $\Br(X)\to\Br(\ov X)^\Gamma$}

In general, the natural morphism \[\Br(X)\to\Br(\ov X)^\Gamma\] is not surjective. In order to decide which elements lie in the image, we make use of the cohomological machinery developed in \cite{cts,qua}, which in general applies to any surface over a number field with torsion-free geometric Picard group.

There are two short exact sequences
\[0\to \Pic(\ov X) \to \Pic(\ov X)^* \to \Delta\to 0,\quad
0\to T(X_\BC) \to T(X_\BC)^* \to \Delta\to 0,\]
where $\Delta$ denotes the discriminant of $X$. The injection in both sequences is the natural embedding via the cup product and the unimodularity of the lattice $\RH^2(X_\BC,\BZ)$ ensures that the cokernels of both sequences are isomorphic as abstract groups.

After dualising the latter sequence, we obtain an exact sequence of Galois modules
\[0\to \Delta\to T(X_\BC)\otimes\BQ/\BZ \xrightarrow{p} T(X_\BC)^*\otimes\BQ/\BZ \to 0.\]
Recall that by \Cref{prop:grothendieck} the cokernel of this sequence is naturally isomorphic to $\Br(\ov X)$.
For a positive integer $n$, we define $M_n=p^{-1}(\Br(\ov X)[n])$. Then the action of $\Gamma$ on $\Pic(\ov X)$ and $M_n$ factors through a finite quotient $G_n=\Gal(k_n/k)$.
\begin{proposition}\label{prop:qua}\cite[Cor.~1.6]{qua}
 The image of $(\Br(X_k)/\Br_1(X_k))[n]\to\Br(\ov X)^\Gamma$ equals the kernel of the composition of the two connecting maps
    \[\Br(\ov X)[n]^{G_n}\to\RH^1(G_n,\Delta),\quad \RH^1(G_n,\Delta)\to\RH^2(G_n,\Pic(\ov X)).\]
\end{proposition}
Since $\#\Delta=2^4 3^3$, the map $\Br(X_k)[\ell^\infty]\to\Br(\ov X)[\ell^\infty]^\Gamma$ is in fact surjective for $\ell\neq 2,3$ (and the same is true for $\Br(X)[\ell^\infty]\to\Br(\ov X)[\ell^\infty]^{\Gal(\ov\BQ/\BQ)}$).

The exact sequences
\begin{equation*}
 0\to \Delta\to M_2 \xrightarrow{p} \Br(\ov X)[2] \to 0,\quad
0\to \Delta\to M_3 \xrightarrow{p} \Br(\ov X)[3] \to 0
\end{equation*}
become
\begin{align}
 0&\to\CO/12\sqrt{-3}\xrightarrow{\cdot 2}\CO/24\sqrt{-3}\to\CO/2 \to0,\label{eq:2transcendental}\\
0&\to\CO/12\sqrt{-3}\xrightarrow{\cdot 3}\CO/36\sqrt{-3}\to\CO/3 \to0.\label{eq:3transcendental}
\end{align}

\begin{lemma}\label{lem:ringclass}
 In \Cref{prop:qua}, we can set \[k_2=k(\sqrt{3},\sqrt[6]{2},\sqrt[6]{A_1},\sqrt[6]{A_2},\sqrt[6]{A_3})\text{ and }
 k_3=k(\sqrt[3]{2},\sqrt[6]{3},\sqrt[6]{A_1},\sqrt[6]{A_2},\sqrt[6]{A_3}).\]
\end{lemma}
\begin{proof}
    Adjoining sixth roots of the coefficients, it remains to show the statement for the surface $X_{-1,-1,-1}$ on which the Galois action was described by multiplication with elements of the form $\pi/\ov\pi$.
    
    By \cite[\S3.1]{cn17} (see also \S\ref{sec:Pic}), the absolute Galois group of $k(\sqrt[3]{2},\sqrt{3})$ acts trivially on $\Pic(\ov X_{-1,-1,-1})$.
    
    We have that $\pi/\ov \pi$ acts trivially on $\CO/24\sqrt{-3}$ if and only if $\pi \equiv \ov \pi \bmod 24\sqrt{-3}$ if and only if $\pi\in\BZ+12\sqrt{-3}\BZ$. Thus the Galois group $\Gal(\ov k/k(\BZ+12\sqrt{-3}\BZ))$ corresponding to the ring class field $k(\BZ+12\sqrt{-3}\BZ)$ of the non-maximal order $\BZ+12\sqrt{-3}\BZ\subset\CO$ acts trivially on $\CO/24\sqrt{-3}$. It can be checked that \[k(\BZ+12\sqrt{-3}\BZ)=k(\sqrt{3},\sqrt[6]{2}).\]
    The statement for $k_3=k(\BZ+18\sqrt{-3}\BZ)$ follows analogously.
\end{proof}

Using the previously derived description of the middle cohomology for degree $2$ K3 surfaces and \Cref{prop:qua}, one can now completely classify the transcendental part of the Brauer group of $X_\AA$. For our purposes, we will only require much weaker statements.

\begin{proposition} \label{prop:3transcendental}
 Assume $-A_1A_2A_3\in {k^\times}^2\setminus{k^\times}^6$. Then \[(\Br(X_k)/\Br_1(X_k))[3]=0\quad(\text{hence also } \Br(X)/\Br_1(X))[3]=0).\]
\end{proposition}
\begin{proof}
It follows from \Cref{lem:ringclass} that the group $G_3$ is a subgroup of the generic Galois group
\[\Gal(k(\sqrt[3]{2},\sqrt[6]{3},\sqrt[6]{t_1},\sqrt[6]{t_2},\sqrt[6]{t_3})/k(t_1,t_2,t_3))\]
(compare this with the smaller generic Galois group acting on $\Pic(\ov X)$ in \S\ref{sec:Pic}). We can thus iterate over all such subgroups (up to conjugacy and permutation of the variables $(t_1,t_2,t_3)$) and obtain all possible cases for the action of $\Gamma$ on \eqref{eq:3transcendental}.

The condition on the coefficients corresponds to the cases where $\Br(\ov X)[3]^\Gamma=\BZ/3$ by \Cref{invariants}. It can be checked with \texttt{Magma} that the kernel of
\[\Br(\ov X)[n]^{G_3}\to\RH^1(G_3,\Delta)\to\RH^2(G_3,\Pic(\ov X))\]
is trivial in all these cases.
\end{proof}

\begin{proposition}\label{prop:2transcendental}
 There exists a finite set $C\subset\BQ^\times$ such that if $(\Br(X)/\Br_1(X))[2]$ is non-trivial, then $\alpha A_1A_2A_3\in \BQ^{\times6}$ for some $\alpha\in C$.
\end{proposition}
\begin{proof}
 As in \Cref{prop:3transcendental}, it can be checked that in all cases of subgroups
 \[G_2\subset\Gal(k(\sqrt{3},\sqrt[6]{2},\sqrt[6]{t_1},\sqrt[6]{t_2},\sqrt[6]{t_3})/k(t_1,t_2,t_3))\]
 which yield $(\Br(X_k)/\Br_1(X_k))[2]\neq0$ by \Cref{prop:qua}, the group $G_2$ corresponds to sixth power conditions of the form $\alpha A_1A_2A_3\in (\BQ^\times)^6$.
\end{proof}

\part{Proof of the main theorems}
In this part we bring everything together to prove the main theorems from the introduction. This requires  calculating the possible algebraic Brauer groups.

\section{The Picard group and algebraic Brauer group} \label{sec:subgroupsOfG}

\subsection{Galois action on the Picard group} \label{sec:Pic}

For the proof of our main theorems, we will need to understand quite well the Galois action on the geometric Picard group of the $X_{\AA}$. A complete set of generators was calculated in \cite[\S2.1]{cn17}. In particular, it was shown that the Picard rank of $\overline{X}_{\AA}$ is $20$, and the Galois action on $\Pic(\Xbar_\AA)$ factors through a group of order $864 = 2^5 \cdot 3^3$. Explicitly,  if $G_{\mathbf{A}}$ is the associated Galois group acting on $\Pic(\Xbar_\AA)$, then by \cite[\S3.1]{cn17} one has
$$G_\AA = \Gal(\Q(\sqrt[3]{2},\zeta_{12}, \sqrt[6]{A_1/A_2}, \sqrt[6]{A_2/A_3}, \sqrt{A_1})/\Q).$$
We denote the generic Galois group by $G$. The group $G_\AA$ is naturally a subgroup of $G$, which is well-defined up to conjugacy.

\begin{lemma}\label{lem:subgroups}
Let $A_1,A_2,A_3\in \Z$ be nonzero integers such that  $G_{\mathbf{A}} \subset G$ is a subgroup of index $d=2^x3^y$. Then there exists at least $x$ many conditions of the form (2a), (2b), or (2c), and $y$ many conditions of the form (3a) or (3b) from Table \ref{table:cond} that $A_1,A_2,A_3$ must satisfy,
with stated upper bound for $\max_i |A_i| \leq T$.
\end{lemma}
\begin{table}[ht]
\begin{center}
    \begin{tabular}{ | l | l | l |}
    \hline
    Label & Condition & Upper bound\\ \hline
    (2a) & $tA_i$ is a square &  $\ll T^{5/2}$\\ \hline
    (2b) & $tA_iA_j$ is a square &  $\ll T^{2+\varepsilon}$\\ \hline
	(2c) & $tA_1A_2A_3$ is a square &  $\ll T^{3/2+\varepsilon}$\\ \hline
	(3a) & $sA_i/A_j$ is a cube &  $\ll T^{2+\varepsilon}$\\ \hline
	(3b) & $sA_1A_2A_3$ is a cube & $\ll T^{1 + \varepsilon}$\\
    \hline
	\end{tabular}
\end{center}
\label{table:cond}
\caption{Subgroups of $G$. Here $t\in\{\pm1,\pm3\}$ and $s\in\{1,2,4\}$.}
\end{table}
\begin{proof}
Let $K=\Q(\zeta_{12},\sqrt[3]{2})$ and
$$G_\AA' = \Gal(K(\sqrt[6]{A_1/A_2},\sqrt[6]{A_2/A_3},\sqrt{A_1})/K) \cap G_\AA.$$
As our surface is defined over $\Q$,  the map $G_\AA\to \Gal(K/\Q)$ is surjective, hence $G_\AA'$ also has index $d$ inside $(\Z/6\Z)^2\times\Z/2\Z$, by Kummer theory. If we rewrite $K(\sqrt[6]{A_1/A_2},\sqrt[6]{A_2/A_3},\sqrt{A_1})=K(\sqrt{A_1},\sqrt{A_2},\sqrt{A_3})K(\sqrt[3]{A_1/A_2},\sqrt[3]{A_2/A_3})$ as a compositum, and the Galois group as $(\Z/2\Z)^3\times(\Z/3\Z)^2$, then $G_\AA'$ corresponds to a subgroup $H_2\times H_3$ where $H_2\subset(\Z/2\Z)^3$ of index $2^x$ and $H_3\subset (\Z/3\Z)^2$ of index $3^y$.
It is clear that this subgroup corresponds to $x$ square conditions and $y$ cube conditions listed in Table \ref{table:cond}, but over $K$.
One deduces the statement over $\Q$ using Lemma \ref{lem:power_base_change}.
The stated upper bounds are computed using Lemma \ref{lem:power_count}.
\end{proof}

\begin{lemma}\label{lem:power_base_change}
 Let $k=\BQ(\omega)$ and $K=\Q(\zeta_{12},\sqrt[3]{2})$. Then
 \begin{align*}
 &\ker(\Q^{\times}/\Q^{\times2} \to k^\times/k^{\times2}) = \{ 1, -3\},
 \quad \ker(\Q^{\times}/\Q^{\times6} \to k^\times/k^{\times6}) 
 = \{ 1, -27\}, \\
 & \mathbb{Q}^{\times}/\Q^{\times3} \hookrightarrow k^\times/k^{\times3}, \\
 & \ker(\Q^{\times}/\Q^{\times2} \to K^\times/K^{\times2}) = \{ \pm 1, \pm 3\},
\quad 
\ker(\Q^{\times}/\Q^{\times3} \to K^\times/K^{\times3}) = \{ 1,2,4\}.
 \end{align*}
\end{lemma}
\begin{proof}   
    It suffices to calculate the cardinality
    of the kernels, as the stated elements
    are easily verified to lie in kernel.
    Moreover as $6=2 \cdot 3$ and $\gcd(2,3) = 1$,
    it suffices to treat the cases of $2,3$.
    So let $p$ be a prime and
    $E \subset F$ an extension of fields of characteristic $0$. Then the inflation-restriction sequence
    and Kummer theory yields
 \[\RH^1(\Gal(F/E),\mu_p(F))\cong\ker(E^\times/{E^\times}^p \to F^\times/{F^\times}^p).\]
    If $F/E$ is abelian and $\mu_p \subset F$,
    then \cite[Prop.~2.3]{JL16} implies that
    \begin{equation} \label{eqn:JL}
    \RH^1(\Gal(F/E),\mu_p(F)) \cong \Hom(\Gal(F/E(\mu_p)),\mu_p(E)).
    \end{equation}
    Applying this with $E=\Q$ and $F=k$ gives 
$$
    \#\ker(\Q^{\times}/\Q^{\times p} \to k^\times/k^{\times p}) = 
 \begin{cases}
    2, & p = 2, \\
    1, & p = 3,
 \end{cases}
$$
as required. For $K$, we use the tower $\Q \subset \Q(\zeta_{12}) \subset K$ of abelian extensions and apply \eqref{eqn:JL} to each extension in the tower.
This gives
$$ \#\ker(\Q^{\times}/\Q^{\times p} \to K^\times/K^{\times p})  \quad
 \begin{cases}
    =4, & p = 2, \\
    \leq 3, & p = 3.
 \end{cases}
$$    
Given the explicit elements already found, this completes the proof.
\end{proof}

\begin{lemma} \label{lem:power_count}
    Let $\alpha \in \Z$ and $d,d_1,\dots,d_n \in \N$. Then for all $\varepsilon > 0$ we have
    $$\#\{ a_1,\dots, a_n \in \Z : \max |a_i| \leq T, \alpha a_1^{d_1} \dots a_n^{d_n} \in \Q^{\times d} \} 
    \ll_\varepsilon T^{(d_1 + \dots + d_n)/d + \varepsilon}.$$
\end{lemma}
\begin{proof}
    Writing $a^d = \alpha a_1^{d_1} \dots a_n^{d_n}$ and $b_i = a_i^{d_i}$, the quantity in the statement is
    $$ \ll \sum_{|a| \leq T^{(d_1 + \dots + d_n)/d}} \sum_{b_i \mid a^d}1
    \ll T^{(d_1 + \dots + d_n)/d + \varepsilon}$$
    on using the standard bound $\tau(a) \ll_\varepsilon a^{\varepsilon}$ for the divisor function.
\end{proof}

\subsection{Algebraic Brauer group}
We now explain how to calculate the algebraic Brauer group of the surfaces $X_\AA$ over $\Q$. Recall from the Hochschild--Serre spectral sequence \cite[Cor.~6.7.8]{Poo17}, that there is an isomorphism
\begin{equation} \label{eqn:Br_1}
    \Br_1 X_\AA / \Br_0 X_\AA \cong \HH^1(G_\AA, \Pic \Xbar).
\end{equation}
From \S\ref{sec:Pic}, the group $G_\AA$ acts via a subgroup of the generic Galois group $G$. Thus to calculate all possible algebraic Brauer groups, it suffices to enumerate the subgroups of $G$ and compute the corresponding group cohomology \eqref{eqn:Br_1}.

\begin{table}
\begin{center}
    \begin{tabular}{ | l | l | l | l |}
    \hline

    Label & Condition &
    \parbox{15em}{Generators of $\Br_1 X_\AA/ \Br_0 X_\AA$\\
         or negligible upper bound}\\ \hline
    (1) & - & 0 \\ \hline
(2a) & $tA_i\in \Q^{\times 2}$. &
    $\begin{array}{@{}l}
    t=-3: \calA_i \\
    \text{otherwise}: 0
    \end{array}$
    
    \\ \hline
(2b) & $tA_iA_j\in \Q^{\times 2}$. & 0 \\ \hline
(2c) & $tA_1A_2A_3 \in \Q^{\times 2}$. & 0 \\ \hline
(3a) & $sA_i/A_j\in \Q^{\times 3}$. & $\begin{array}{@{}l}
    s=1: \calB_k \\
    \text{otherwise}: 0
    \end{array}$
    
    \\ \hline
(3b) & $sA_1A_2A_3\in \Q^{\times 3}$. & $\ll T^{1 + \varepsilon}$\\
\hline
(4a) & $tA_i\in \Q^{\times 2}$. $t'A_j\in \Q^{\times 2}$.&
$\begin{array}{@{}l}
    t=t'=-3: \calA_i,\calA_j\\
    t=-3,t'\neq-3: \calA_i\\
    t,t'\neq-3: 0
    \end{array}$
    
    \\ \hline
(4b) & $tA_i\in \Q^{\times 2}$. $t'A_jA_k\in \Q^{\times 2}$.&
$\begin{array}{@{}l}
    t=-3:\calA_i\\
    \text{otherwise}: 0
    \end{array}$\\ \hline
(4c) & $tA_iA_j\in \Q^{\times 2}$. $t'A_jA_k\in \Q^{\times 2}$.& 
0\\ \hline

(6a) & $tA_i\in \Q^{\times 2}$. $sA_i/A_j\in \Q^{\times 3}$.&
$\begin{array}{@{}l}
    t=-3,s=1: \calA_i,\calB_k\\
    t=-3,s=2: \calA_i\\
    t\neq-3,s=1: \calB_k \\
    t\neq-3,s=2: 0
    \end{array}$
    
    \\ \hline
(6b) & $tA_i\in \Q^{\times 2}$. $sA_j/A_k\in \Q^{\times 3}$.&
$\begin{array}{@{}l}
    t\neq-3,s=1:\calB_i\\
    t=-3,s=2: \calA_i\\
    \text{otherwise}: 0
    \end{array}$\\ \hline
(6c) & $tA_iA_j\in \Q^{\times 2}$. $sA_i/A_j\in \Q^{\times 3}$.& 
$\begin{array}{@{}l}
    t=3,s=1:\calB_k,\calC_k\\
    t=1,-3,s=1: \calB_k\\
    s=2: 0
    \end{array}$\\ \hline
(6d) & $tA_iA_j\in \Q^{\times 2}$. $sA_j/A_k\in \Q^{\times 3}$.& $\ll T^{1+\varepsilon}$\\ \hline
(6e) & $tA_1A_2A_3\in \Q^{\times 2}$. $sA_i/A_j\in \Q^{\times 3}$.& 
$\begin{array}{@{}l}
    t=1,-3,s=1:\calB_k\\
    \textnormal{otherwise}:0
    \end{array}$ \\
    \hline
    
(8) & $tA_1,t'A_2,t''A_3 \in \Q^{\times 2}$. & $\begin{array}{@{}l}
    t=t'=t''=-3:\calA_1,\calA_2,\calA_3\\
    t=t'=-3,t''\neq-3:\calA_1,\calA_2\\
    t=-3,t',t''\neq-3:\calA_1\\
    t,t',t''\neq-3:0
    \end{array}$ \\
    \hline
    
(9)&$sA_1/A_2, s'A_3/A_2 \in \Q^{\times 3}$.& $\ll T^{1 + \varepsilon}$\\ \hline

(12a) & 
$tA_i, t'A_j \in \Q^{\times 2}$. $sA_i/A_j \in \Q^{\times 3}.$ &
$\begin{array}{@{}l}
    t=t'=-3,s=1: \calA_i,
    \calA_j,\calB_k\\
    t=t'=-3,s=2: \calA_i,
    \calA_j\\
    t=-3,t'=3,s\neq 1: \calA_i \\
    t=-3,t'=-1,s\neq 1: \calA_i \\
    t=-3,t'=-1,s=1: \calA_i,\calB_k,\calC_k\\
    t,t'\neq-3,s=1: \calB_k\\
    t,t'\neq-3,s\neq1:0 \\
    \end{array}$
    
    \\ \hline
(12b) & 
$tA_i, t'A_j \in \Q^{\times 2}$. $sA_i/A_k \in \Q^{\times 3}.$ &
$\ll T^{1 + \varepsilon}$ \\ \hline
(12c) & $tA_i, t'A_jA_k \in \Q^{\times 2}$. $sA_i/A_j \in \Q^{\times 3}.$ & $\ll T^{1 + \varepsilon}$\\ \hline
(12d) & $tA_i, t'A_jA_k \in \Q^{\times 2}$. $sA_j/A_k \in \Q^{\times 3}.$ & 
$\begin{array}{@{}l}
    t=t'=-1,3,s=1: \calB_i,\calC_i\\
    t=-3,s=2: \calA_i\\
    \textnormal{otherwise}: 0
    \end{array}$\\ \hline
    \end{tabular}
\end{center}
\caption{Algebraic Brauer groups of low index subgroups}
\label{table:brsub}
\end{table}

We did this for all subgroups of index $\leq 12$ using \texttt{Magma} and Lemma \ref{lem:subgroups}. The results can be found in Table \ref{table:brsub} below. We list the generators for the algebraic Brauer group in terms of the elements $\calA_i,\calB_i,\calC_i$ from \S\ref{sec:Brauer_elements}, or alternatively just state the upper bound of this subgroup if it is negligible for our counting results. One caveat is that we ignore those subgroups for which a coefficient is a square or if  $-A_i/A_j$ is a sixth power, since such surfaces  have rational points. 

We briefly explain how to obtain the stated upper bounds. For (3b) this follows from Lemma \ref{lem:power_count}. For (6d), the conditions imply that $t^3s^4A_i^3A_jA_k^2 \in \Q^{\times6}$, which is $O(T^{1+ \varepsilon})$ by Lemma \ref{lem:power_count}. For (9) the conditions imply that $ss'A_1A_2A_3$ is a cube, which is $O(T^{1+ \varepsilon})$ by Lemma \ref{lem:power_count}. For (12b) we first sum over $A_j$ which gives $O(T^{1/2})$, whilst the remaining conditions imply that $t^3s^4A_iA_k^2 \in \Q^{\times6}$, which is $O(T^{1/2+ \varepsilon})$ by Lemma \ref{lem:power_count}. Finally (12c) is a specialisation of (6d).

\begin{remark}
    If $-A_i/A_j \in \Q^{\times 6}$ then utilising the factorisation of 
    $A_ix_i^6 + A_jx_j^6$ over $\Q$, one may write down more quaternion algebras using a similar method to \S\ref{sec:Brauer_elements}, which do \emph{not} lie
    in the subgroup generated by $\calA_i,\calB_i,\calC_i$. Such surfaces are not
    of interest to us, as they always contain a rational point.
\end{remark}

\section{Proofs}\label{sec:proofs}

We are now ready to prove the main theorems from the introduction.

\subsection{Large index subgroups} 

\begin{lemma}\label{lem:largeindex}
	$$\#\left\{(A_1,A_2,A_3) \in \Z^3 :
    \begin{array}{c}
    |A_i| \leq T,\\
    |G/G_{\mathbf{A}}|=9\textup{ or }>12
    \end{array} \right\}  \ll T^{1+\varepsilon}.$$
\end{lemma}

\begin{proof}
    Recall that $|G| = 864 = 2^5 \cdot 3^3$.
 If $9\mid |G/G_{\mathbf{A}}|$, then by Table \ref{table:cond}, we just need to consider the conditions $s_1A_1/A_2$ and $s_2A_2/A_3$ are cubes. There are $\ll T^{1+\varepsilon}$ such surfaces by Table \ref{table:brsub}.

Since $G_{\mathbf{A}}$ surjects onto $\Gal(\Q(\zeta_{12},\sqrt[3]{2})/\Q)$, it follows that $12\mid |G_{\mathbf{A}}|$. This cuts down the possibilities for $|G/G_{\mathbf{A}}|$ we need to consider to $18,24,36,72$. It suffices to just consider $|G/G_{\mathbf{A}}|=24$, since $9$ divides the index in all other cases. We now appeal to Lemma \ref{lem:subgroups} using $24 = 3 \cdot 2^3$. If $sA_1A_2A_3$ is a cube then there are $O(T^{1 + \varepsilon})$ such surfaces. Otherwise $sA_1/A_2$ is a cube and either $t_1A_1,t_2A_2,t_3A_3$ are all squares or $t_1A_1A_2,t_2A_1A_3,t_3A_2A_3$ are all squares. Both these cases are a specialisation of (6d), so are $\ll T^{1 + \varepsilon}$ by Table \ref{table:brsub}
\end{proof}

\subsection{Proof of Theorem \ref{maintranscendental}}

Let $k = \Q(\omega)$. Suppose that $X_\AA$ has a non-trivial transcendental Brauer group element whose order is a power of some prime $\ell$.
If $\ell \neq 2,3$, then it follows from Proposition \ref{invariants} that there is some $\alpha \in \Q$ such that $\alpha A_1A_2A_3 \in k^{\times6}$. Then Lemma \ref{lem:power_base_change} implies that $\alpha' A_1A_2A_3 \in \Q^{\times6}$ for some $\alpha' \in \Q^\times$. But by Lemma \ref{lem:power_count}, the number of such surfaces is $O(T^{1/2 + \varepsilon})$, which is satisfactory. For $\ell=2$, \Cref{prop:2transcendental} yields the same. For $\ell = 3$, the only case in Proposition \ref{invariants} which is not negligible is $-A_1A_2A_3 \in k^{\times2}/k^{\times6}$. But it follows from Proposition \ref{prop:3transcendental} that in fact there is no transcendental $3$-torsion Brauer class over $k$, hence neither over $\Q$, as required. 
\qed

\subsection{Proof of Theorem \ref{mainbrauer}}
By Theorem \ref{maintranscendental}, we may assume that $\Br X_\AA = \Br_1 X_\AA$. By Table \ref{table:brsub}, if $\Br_1 X_\AA/\Br_0 Z_\AA \neq 0$, then $G_\AA$ must be a proper subgroup of $G$. By  Lemma \ref{lem:subgroups}, the most common proper subgroup occurs if $-3A_i$ is a square for some $i$. But we already know that the Brauer group is  non-constant in this case. As the intersection of two of these conditions is $O(T^2)$, the quantity in question is thus
\begin{equation} \tag*{\qed}
	\sim 3 \#\{A_i \in \Z : |A_i| \leq T, -3A_1 \in \Q^{\times 2}\} \sim 3 \cdot (2T)^2 \cdot (T/3)^{1/2} \sim 4 \cdot \sqrt{3}T^{5/2}.
\end{equation}

\subsection{Proof of Theorem \ref{thm:main2}}
All the serious counting was done in Part 1, and the main terms come from Theorems \ref{thm:2} and \ref{thm:3+6}. The proof requires examining any new Brauer elements that pop up when restricting to smaller subfamilies. 
By Lemma \ref{lem:largeindex}, we may ignore those surfaces $X_{\mathbf{A}}$ with $|G/G_{\mathbf{A}}|=9\textup{ or }>12$. Moreover by Theorem \ref{maintranscendental}, we may assume that $\Br X_\AA = \Br_1 X_\AA$. 

\begin{table}[!htb]
\begin{center}
    \begin{tabular}{ | l | l | l | }
    \hline
    Generators for $\Br_1 X_\AA/\Br_0 X_\AA$ & Upper bound & Reference\\ \hline
    $\calA_i$ & $T^{3/2}\log\log T/(\log T)^{2/3}$ & Theorem \ref{thm:2} \\ \hline
    $\calB_i$ & $T^{3/2}/(\log T)^{3/8}$ & Theorem \ref{thm:3+6}\\ \hline
    $\calA_i,\calA_j$ & $T^{4/3+\varepsilon}$ & Theorem \ref{thm:multalg} \\ \hline
    $\calA_i,\calB_k$ & $T^{1+\varepsilon}$ & Theorem \ref{thm:multalg}\\ \hline
    $\calB_k,\calC_k$ & $T^{3/2} \log \log T/(\log T)^{1/2}$ & Theorem \ref{thm:multalg}\\ \hline
    $\calA_1,\calA_2,\calA_3$ & 0 & No $\R$-point \\ \hline
    $\calA_i,\calA_j,\calB_k$ & $T^{1+\varepsilon}$ & Theorem \ref{thm:multalg}\\ \hline
    $\calA_i,\calB_k,\calC_k$ & $T^{1+\varepsilon}$ & Theorem \ref{thm:multalg}\\ \hline
	\end{tabular}
\end{center}
\caption{}
\label{table:bounds}
\end{table}

In Table \ref{table:bounds}, we list all possible generators for $\Br_1 X_\AA/\Br_0 X_\AA$ from Table \ref{table:brsub}, and give an upper bound for how often these generators give a Brauer--Manin obstruction to the Hasse principle using results from the previous sections.

\noindent
$N_1:$ According to Table \ref{table:brsub}, any subgroup where $-3A_i$ is not a square and $A_j/A_k$ is not a cube for all $i,j,k$ is either $O(T^{1+\varepsilon})$ or $\Br X_\AA/\Br_0 X_\AA$ is trivial.

\noindent
$N_2:$ Theorem \ref{thm:2}  proves the lower bound. The upper bound is in Table \ref{table:bounds}.

\noindent
$N_3:$ Theorem \ref{thm:3+6}  proves the lower bound. The upper bound is in Table \ref{table:bounds}.
\qed

\subsection{Proof of Theorem \ref{thm:main1}}

This is a direct consequence of Theorem \ref{thm:main2}. \qed

\subsection{Proof of Theorem \ref{thm:oddtor}}
By Lemma \ref{lem:largeindex} we only need to consider those surfaces $X_{\mathbf{A}}$ with $|G/G_{\mathbf{A}}|\neq 9$ or $\leq 12$. Moreover by Theorem \ref{maintranscendental}, we may assume that $\Br X = \Br_1 X$.
For the lower bound, Theorem \ref{thm:2} shows that
 $$\#\left\{(A_1,A_2,A_3) \in \Z^3 :
    \begin{array}{c}
    |A_i| \leq T,\\
    X_\AA(\Adele_\Q) \neq \emptyset,
X_\AA(\Adele_\Q)^{\Br_{2^\perp}} = \emptyset 
    \end{array} \right\}  \gg \frac{T^{3/2} \log \log T}{(\log T)^{2/3}}.$$
It therefore suffices to show that for most of these surfaces we have $X_\AA(\Adele_\Q)^{\Br_{2^{\phantom{\perp\mkern-15mu}}}}\neq\emptyset$. If $X_\AA(\Adele_\Q)^{\Br_{2^{\phantom{\perp\mkern-15mu}}}}=\emptyset$ as well, then $\Br_1 X_\AA/\Br_0 X_\AA$ contains nontrivial elements of both odd order and even order. However Table \ref{table:bounds} shows that such surfaces with a Brauer--Manin obstruction are $O(T^{1+\varepsilon})$.

For the upper bound, Table \ref{table:brsub} shows that the only time there is a nontrivial odd torsion Brauer class in $\Br X_\AA /\Br_0 X_\AA$ is when $-3A_i$ is a square for some $i$. The upper bound then follows from the upper bound for $N_2(T)$ from Theorem \ref{thm:main2}.
\qed

\end{document}